\documentclass[11pt]{article}
\usepackage{graphicx}%
\usepackage{multirow}%
\usepackage{amsmath,amssymb,amsfonts}%
\usepackage{amsthm}%
\usepackage{mathrsfs}%
\usepackage[title]{appendix}%
\usepackage{xcolor}%
\usepackage{textcomp}%
\usepackage{manyfoot}%
\usepackage{booktabs}%
\usepackage{algorithm}%
\usepackage{algorithmicx}%
\usepackage{algpseudocode}
\usepackage{soul}%
\usepackage{listings}%
\usepackage{ulem}

\usepackage{float}
\usepackage{geometry,amsmath}
\usepackage{enumerate}
\usepackage{dsfont}
\usepackage{hyperref}
\usepackage{natbib}
\usepackage{color}
\usepackage{epsfig}
\usepackage{diagbox}
\usepackage{dsfont}
\usepackage{bbm,bm,color}
\usepackage{multirow,multicol}
\usepackage{tikz,tabularx}
\usepackage{caption}
\usepackage{subcaption}
\usepackage{enumitem}

 \usepackage{setspace}
 \hypersetup{hidelinks}

\newcommand{\bgeqn}{\begin{eqnarray}}
\newcommand{\edeqn}{\end{eqnarray}}
\newcommand{\bgeq}{\begin{eqnarray*}}
\newcommand{\edeq}{\end{eqnarray*}}

\newtheoremstyle{propositionstyle} 
  {\topsep}{}{\itshape}{}{\bfseries}{}{ }{\thmname{#1}~\thmnumber{#2}.\thmnote{(#3)~}}

\newcommand{\dd}{\textsf {d\kern -0.07em l}} 

\newcommand{\diam}{\operatorname{diam}}

\newcommand{\UU}{\mathcal{U}}
\newcommand{\XX}{\mathcal{X}}
\newcommand{\PP}{\mathcal{P}}
\newcommand{\e}{\bm e}

\newcommand{\ubar}{\overline{x}}
\newcommand{\ulow}{\underline{x}}

\newcommand{\Nk}{N_k}
\newcommand{\hk}{h_k}
\newcommand{\Pik}{\Pi_k}
\newcommand{\Tk}{T_k}
\newcommand{\Uk}{\UU^k}
\newcommand{\Ukp}{\UU^{k,+}}

\newcommand{\Pkp}{\PP^{k,+}}

\newcommand{\B}{\mathcal{B}}

\newtheorem{proposition}{Proposition}[section]
\newtheorem{example}{Example}[section]%
\newtheorem{remark}{Remark}[section]%
\newtheorem{corollary}{Corollary}[section]%
\newtheorem{definition}{Definition}[section]%
\newtheorem{assumption}{Assumption}%
\newtheorem{lemma}{Lemma}[section]%
\newtheorem{theorem}{Theorem}[section]%

\newcommand{\vv}{{\bm{v}}}
\newcommand{\x}{{\bm{x}}}
\newcommand{\R}{{\mathbb{R}}}

\newcommand{\p}{{\bm{p}}}
\newcommand{\q}{{\bm{q}}}
\newcommand{\bal}{{\bm{\alpha}}}
\newcommand{\bbe}{{\bm{\beta}}}

\newcommand{\g}{{\bm{g}}}
\newcommand{\G}{{\bm{G}}}
\newcommand{\rr}{{\bm{r}}}
\newcommand{\RR}{{\bm{R}}}
\newcommand{\cc}{{\bm{c}}}

\newcommand{\inmat}[1]{\mbox{\rm {#1}}}

\makeatletter
  \newenvironment{breakablealgorithm}
  {%
    \par\vspace{1ex}%
    \refstepcounter{algorithm}
    \begingroup
      \renewcommand{\caption}[2][\relax]{%
        \noindent\hrule height .8pt\relax\par
        \vspace{2pt}%
        {\noindent\textbf{\ALG@name~\thealgorithm}\ ##2\par}%
        \ifx\relax##1\relax
          \addcontentsline{loa}{algorithm}{%
            \protect\numberline{\thealgorithm}##2}%
        \else
          \addcontentsline{loa}{algorithm}{%
            \protect\numberline{\thealgorithm}##1}%
        \fi
        \vspace{2pt}%
        \noindent\hrule height .4pt\relax\par
        \vspace{4pt}%
      }%
  }{%
    \vspace{4pt}%
    \noindent\hrule height .4pt\relax\par
    \vspace{1ex}%
    \endgroup
  }
\makeatother

\AtBeginDocument{%
  \setlength{\abovedisplayskip}{6pt plus 2pt minus 2pt}
  \setlength{\belowdisplayskip}{6pt plus 2pt minus 2pt}
  \setlength{\abovedisplayshortskip}{3pt plus 1pt minus 1pt}
  \setlength{\belowdisplayshortskip}{5pt plus 1pt minus 1pt}
  \setlength{\jot}{2pt}
}
\begin{document}
\title{Coordinate-wise Polyhedral Method for Eliciting Multivariate Linear Utility and Univariate Nonlinear Utility Functions 
}

\author{Jiaxin Wei\thanks{School of Mathematics and Statistics, Xi’an Jiaotong University, Xi’an, Shaanxi, P. R. China, \texttt{weijx0609@stu.xjtu.edu.cn}}, 
Jia Liu\thanks{School of Mathematics and Statistics, Xi’an Jiaotong University, Xi’an, Shaanxi, P. R. China, \texttt{jialiu@xjtu.edu.cn}}\; and Huifu Xu\thanks{Department of Systems Engineering and Engineering Management, The Chinese University of Hong Kong, Hong Kong, \texttt{h.xu@cuhk.edu.hk}}}







\maketitle

\begin{abstract}

Preference elicitation is an important topic in behavioral economics and decision analytics. The well-known polyhedral method provides an efficient approach to eliciting decision makers’ preferences in discrete choice models. However, 
the question of whether the associated algorithmic procedure is guaranteed to generate a sequence of polyhedra that converges to a singleton remains open.
In this paper, we propose a coordinate-wise polyhedral method (CPM) for cutting polyhedra
with theoretical guarantees of convergence. Unlike the existing polyhedral method, which designs pairwise comparison queries by solving coupled optimization problems and performs cuts subsequently, CPM specifies coordinate-wise cuts in advance and then designs corresponding pairwise comparison queries by solving a linear system of equations.
Under this framework, we show that CPM reduces the diameter of the polyhedron 
at a linear convergence rate. 
{\color{black} Moreover, we extend CPM to a univariate piecewise-linear utility function by 
representing 
it with its increment vector over
consecutive breakpoints. We 
show that the Kantorovich distance between the true utility function and 
the estimated one 
obtained by CPM decreases 
at a linear convergence rate. 
Further,
we extend the CPM to general nondecreasing Lipschitz 
continuous utility functions by piecewise-linear approximation (PLA). 
We introduce an 
adaptive-breakpoint strategy to avoid direction errors caused by PLA.
We prove that the 
set of piecewise-linear utility functions
corresponding to the ambiguity set of increment vectors
converges to the true utility function as the number of queries increases, 
and derive an explicit bound for the approximation.
}
Finally, to evaluate the performance of CPM, we conduct a series of numerical experiments. 
The results demonstrate 
comparable convergence 
to the standard polyhedral method 
in the case of linear multivariate utility functions. 
For nonlinear univariate utility functions, CPM achieves 
stable convergence 
to the true utility function, in line with the theoretical findings.

\end{abstract}

\noindent
\textbf{Keywords.} 
Coordinate-wise polyhedral method, multivariate linear utility function, univariate nonlinear utility function, piecewise-linear approximation, error bound



\section{Introduction}

Understanding and identifying individual preferences is a central theme in behavioral economics, decision analytics, marketing, and more recently AI-driven recommendation systems. Over recent
decades, 
many methods have been proposed 
to elicit preferences, each tailored to the structure of the problem and the nature of the functions representing preferences.

In the discrete choice models commonly used in marketing to capture customers' preferences, linear multivariate parametric utility functions are often employed to represent customers' evaluation of products. The goal of preference elicitation in this context is to identify the true parameter values that characterize a customer's evaluation of product features. 
This is typically carried out through pairwise comparison questionnaires (known as choice-based conjoint analysis),  where customers are asked to choose between two 
options (profiles). Each response incrementally reduces the uncertainty (or ambiguity) in the underlying parameter. 
There are two popular approaches to reducing this ambiguity.
One is the polyhedral method. This approach begins with a polyhedron that is known to contain the true value of the parameter. After observing the customer's response to a 
purposely designed question, 
the method 
constructs a separating 
hyperplane that 
cuts 
the 
polyhedron in approximately half, 
thereby 
reducing the set of plausible parameter values rapidly. A prerequisite 
for the approach is that there is no response error in the preference elicitation process, 
see \citet{toubia2003fast,toubia2004polyhedral} 
for details. 

The other is the Bayesian method which
complements the polyhedral method
by dealing with response errors.
This approach starts with a prior distribution over the parameter space, representing existing (possibly partial) knowledge. 
As the customer responds to a series of pairwise comparison questions, the method updates the distribution using Bayes' rule, resulting in a posterior distribution that reflects the updated belief about the parameter. 
These questions are designed to minimize the expected intermediate D-error defined as
a function of the determinant of the covariance matrix
of the estimates, which can be interpreted as a function
of the volume of a credibility region or confidence set
around the estimate,
see 
e.g.~\citet{toubia2007probabilistic,toubia2013dynamic,yu2012comparison,saure2019ellipsoidal} and references therein.
There are also some other optimization-based or statistical  approaches.
\citet{bertsimas2013learning} propose a robust optimization-based methodology that integrates behavioral insights such as loss aversion and response inconsistency, using adaptive questionnaires and integer programming to update individual preferences dynamically. 
\citet{vayanos2020robust} develop an adaptive robust optimization method to elicit preferences strategically using pairwise comparisons, enabling high-quality decision recommendations in high-stakes settings under preference uncertainty.
\citet{chen2026eliciting} use maximum likelihood estimation to elicit a non‑parametric univariate utility function with response error, while incorporating known structural preference information (monotonicity, concavity, Lipschitz continuity) as linear constraints on the utility function.
\citet{wei2025personalized} estimate 
the partworth vector directly through optimization based on responses instead of cutting the ambiguity set for partworth. They preserve information from previous queries by incorporating it into the objective function.

In the literature on decision analytics and behavioral economics, considerable effort has been devoted to eliciting a decision maker's (DM's) univariate utility function (which is typically nonlinear). Among various heuristic approaches, the random utility split (RUS) and random relative utility split (RRUS) methods have been proposed to reduce the ambiguity in utility estimates through structured pairwise comparisons.
In the RUS method, one lottery is constructed with two fixed outcomes typically corresponding to the worst and best possible returns of the reward function in a given decision-making problem, see \cite{wakker1996eliciting}.
The other lottery features a random outcome, uniformly drawn from the range between the worst and best outcomes. By asking the DM to choose between these two lotteries, the method aims to halve the range of admissible utility functions in the ambiguity set at the point of comparison.
The RRUS method modifies RUS by introducing randomness in a different way, where the first lottery has two outcomes selected randomly between the worst and best possible values and the second is a deterministic outcome, located at the midpoint of the two outcomes in the first lottery. This approach also reduces ambiguity in utility estimation but focuses on relative comparisons rather than absolute extremes, see 
\citet{armbruster2015decision,guo2024utility}.
\citet{zhang2025modified} propose a modified 
polyhedral method  which extends the existing polyhedral method 
 to general nonlinear univariate utility functions.
A key step in their approach is to use the piecewise-linear approximation (PLA) 
of a general nonlinear utility where the PLA is parameterized by the vector of 
increments over two consecutive breakpoints, and then represents the set of  
feasible utility functions by 
a polyhedron of vectors of increments in a Euclidean space. Since the PLAs are linear with respect to the vector of increments, 
a modified polyhedral method is proposed in the space of the vectors of increments subsequently.

While Bayesian methods display convergence in probability, the deterministic polyhedral methods and their modifications 
do not provide a 
qualitative or quantitative analysis of the convergence of the
elicitation process,
i.e.,
whether and how fast the ambiguity set is reduced to a singleton as the number of queries increases. 
This limitation is 
fundamentally related to the lack of principled query-generation mechanisms that explicitly and measurably control the reduction of the ambiguity set.
In this paper, we address this issue
by proposing a coordinate-wise polyhedral method (CPM) for both linear parametric multivariate utility functions and nonlinear univariate utility functions.
The new method follows the basic framework of the polyhedral method, but differs from it in two important aspects: the construction of polyhedral cuts and the design of queries. {Specifically, we replace the analytic center with the center of the minimum enclosing ball of the polyhedron, which ensures a strict decrease in the radius of the ambiguity set after each major iteration. Moreover, the construction of cuts in CPM no longer relies on Sonnevend’s ellipsoid or any outer ellipsoid.}
Furthermore, instead of performing a sequence of independent cuts, CPM generates a set of mutually orthogonal 
cuts within a single major iteration, with the number of hyperplanes equal to the dimension of the ambiguity set. By treating these cuts and their corresponding queries as a unified operation, CPM explicitly enforces the angular separation of successive hyperplanes, enabling geometric control over the elicitation process.

The main contributions of this paper can be summarized as follows.

\begin{itemize}
    \item \textbf{A geometrically controlled polyhedral elicitation framework.}
    We propose CPM for adaptive preference elicitation based on pairwise comparisons. Unlike the existing polyhedral methods, which first generate a query and then derive the associated cut, CPM reverses this order: it first specifies geometrically controlled cutting planes and then constructs pairwise comparison queries whose separating hyperplanes coincide with the prescribed cuts. This cut-first design creates an explicit correspondence between cuts and queries. At each major iteration, CPM uses the center of the minimum enclosing ball (MEB) as the reference point and generates a set of mutually orthogonal cutting planes, thereby enforcing the angular separation between cuts and preventing the nearly parallel cuts that may arise in the classical polyhedral method. This geometric structure is the key mechanism behind the guaranteed contraction of the ambiguity set.

    \item \textbf{Extension from multivariate linear utility to nonlinear univariate utility functions.}
    We develop CPM not only for multivariate linear utility functions but also for univariate piecewise-linear utility functions under expected utility theory. By representing a piecewise-linear univariate utility function through its vector of increments over consecutive breakpoints, the corresponding elicitation problem is transformed mathematically into a polyhedral cut problem with finite dimensions. For any prescribed cutting plane in the incremental vector space, we provide a construction for generating a query with a pair of lotteries whose expected-utility comparison induces that hyperplane exactly. We further extend the framework to general normalized, nondecreasing, Lipschitz-continuous utility functions by introducing an adaptive-breakpoint CPM. The algorithm refines the set of breakpoints through
    an outer loop and performs CPM on the current piecewise-linear approximation in the inner loop.
    To avoid direction errors caused by approximation away from the breakpoints, we restrict the lottery outcomes in each inner query to the current set of breakpoints, which ensures that the true utility comparison and its piecewise-linear representation remain consistent.

\item \textbf{Quantitative convergence and complexity guarantees.}
To the best of our knowledge, this paper provides the first quantitative convergence analysis for a deterministic polyhedral elicitation scheme based on adaptive pairwise comparisons. For multivariate linear utilities, we prove that the diameter of the ambiguity set decreases at a linear convergence rate with the number of major CPM iterations. This yields explicit query budget bounds for achieving a prescribed accuracy of estimation and for recovering the correct signs of the partworth components. For univariate piecewise-linear utilities, we prove convergence of the induced ambiguity set of utility functions under both the Kantorovich and Kolmogorov metrics. For general Lipschitz-continuous utilities, we establish an explicit error bound for the 
CPM with adaptively updated breakpoints in the utility-function space, where the bound consists of two sources of errors: the piecewise-linear approximation error induced by the current set of breakpoints  and the estimation error from the current ambiguity set after CPM cuts. 

    \item \textbf{Numerical performance and downstream decision implications.}
For multivariate linear utility functions, we conduct an illustrative numerical experiment to visualize the behavior of CPM and compare it with
the 
well-known polyhedral method 
by \citet{toubia2004polyhedral}.
The experiment intuitively demonstrates their
difference in 
the 
cutting mechanisms. For univariate piecewise-linear utility functions, we verify the convergence of CPM numerically and provide an additional experiment showing its robustness compared with
the modified polyhedral method by \citet{zhang2025modified}. Moreover, an illustrative experiment on general univariate utility elicitation demonstrates that the CPM with adaptively updated breakpoints converges stably toward the true nonlinear utility function, which is consistent with the theoretical convergence results. Finally, we examine a practical wireless-headset application and a downstream robust personalized pricing problem, where the elicited ambiguity set is used to determine conservative individualized prices directly. The results show that CPM generally leads to a more stable reduction in residual uncertainty in preference, partworth sign correctness and pricing conservatism.
\end{itemize}

The 
rest of the paper is organized as follows.
Section \ref{sec:multivariate} 
recalls the basics of the well-known polyhedral method for eliciting a
multivariate linear utility
function.
Section \ref{sec:coordinate-wise} 
introduces 
the CPM for multivariate linear utility functions and analyzes
convergence.
Section~\ref{sec:univariate}
extends the CPM to the univariate piecewise-linear utility function.
Section~\ref{sec:continuous-utility} further extends the CPM to the general nonlinear 
univariate utility function.
Section \ref{sec:numerical} 
showcases the application of CPM
to a wireless-headset pricing problem.
Finally, in Section~\ref{sec:conclusion}, we 
conclude with some remarks 
for future research.
To ease reading, we defer 
proofs to the appendix.

Throughout the paper, we use the following notation.
By convention, we use
$\R^n$ to denote the $n$-dimensional Euclidean space and $\R_+^n$ denote 
its nonnegative orthant. 
 We write $(x_1,\ldots,x_n)^\top$ for an $n$-dimensional column 
vector $\x$ 
and 
$\x^\top$ for the 
transpose of $\x$. Given 
a vector $\x=(x_1,\ldots,x_n)^\top\in\R^n$, for any integer $m\leqslant n$,
we write $\x_{[m]}$ for $(x_1,\ldots,x_m)^\top$. 
As usual, $\bm{e}$ stands 
for the vector with all components being 1 and $\bm{e}_j$ represents $\bm e_j=(0,\dots,0,\underset{j\text{-th}}{1},0,\dots,0)^\top$.
To ease reading, we use
calligraphic letters to denote sets, bold lowercase letters to denote vectors, and bold uppercase letters to denote matrices.
We use $|\mathcal{S}|$ to represent the cardinality of a set $\mathcal{S}$. 
Given any
matrix $\bm{A}=(a_{ij})_{{i=1,\ldots,m},{j=1,\ldots,n}}$ in the space $\R^{m\times n}$ of $m\times n$ real matrices, $\bm{a}_i^\top$ denotes its $i$-th row vector. 
For $x\in\R$, we use $\lceil x\rceil$ to denote the smallest integer not less than $x$ and $\lfloor x\rfloor$ to denote the largest integer no greater than $x$.

\section{Multivariate linear utility function and 
preference elicitation}\label{sec:multivariate}

We begin by considering the case where the DM's\footnote{In the literature of discrete choice models, DM refers to customers whose preferences 
are identical.
In the case where different customers may have different preferences, $\vv$ is a vector of 
random variables whose distribution represents 
the spread of preferences across the whole population of customers and consequently  $u(\x)$ is a random utility function and Bayesian method is used to update the posterior distribution of $\vv$ based on elicited information, see \citet{saure2019ellipsoidal}. 
} preferences can be represented by a parametric multivariate linear utility function and information about the vector of parameters is incomplete. We describe the model and recall the well-known polyhedral method 
of \citet{toubia2004polyhedral}.


\subsection{Multivariate linear utility function}\label{sec:utility}

    



    

We consider a decision-making environment where the DM evaluates a set of 
homogeneous products, 
each described by a vector of attributes $\x\in \mathbb{R}^N$.
The DM's preference is represented 
by a normalized parametric 
linear utility function
\begin{equation}\label{eq:LU}
    u(\x) := \vv^\top \x, \quad \bm{e}^\top\vv=1,
\end{equation}
where $\vv=(v_1,\ldots,v_N)^\top\in \R^N$ is the partworth vector  that characterizes the relative importance the DM assigns to product attributes. 
The normalization condition $\bm{e}^\top\vv=1$ is imposed to remove scale indeterminacy and ensure that utilities are unique up to ordinal transformations \citep{vayanos2020robust,wei2025personalized}.
Because of the normalization condition, 
we can represent the partworth vector 
equivalently 
with the first $(N-1)$ components by introducing 
 a  {\em recovery mapping} 
 $\RR:\R^{N-1}\rightarrow \R^{N}$ 
    \begin{equation}
    \label{map:recover}
   \RR(\bm{y})=(\bm{y},1-\bm{e}^\top\bm{y}),
    \end{equation}
where $\RR(\cdot)$ lifts a vector $\bm{y}\in\mathbb{R}^{N-1}$ to an $N$-dimensional vector 
satisfying the normalization $\bm{e}^\top\RR(\bm{y})=1$.
Hence, 
the space of feasible partworth vectors $\vv$ can, without loss of generality, be regarded as $\R^{N-1}$.
For brevity, we denote by $\vv\in\R^{N-1}$ the reduced partworth vector and by $\RR\vv\in\mathbb{R}^N$ 
its recovery.
Accordingly, the linear utility function in \eqref{eq:LU}
can be expressed by 
 \begin{equation}\label{eq-utility}
u(\x)=(\RR\vv)^\top \x .
 \end{equation}
 We assume that the true value $\vv^*$ corresponding to the true utility function is unknown and must be elicited through pairwise comparison questionnaires 
and that $\vv^*$ lies
in the initial 
polyhedral ambiguity set\footnote{Throughout the paper, we use the terms ``ambiguity set'' and ``polyhedron'' interchangeably for ${\cal P}^k$ depending on the context. The former emphasizes incomplete information about the true partworths whereas the latter describes its geometric structure which is to be reduced/updated 
during the process of eliciting preference.
} 
  \begin{equation}
  \label{eq:initial-polyhedralset}
  \mathcal{P}^0=\{\vv\in\RR^{N-1}\mid \bm{A}\vv\leqslant\bm{b},\ \bm{A}\in\R^{m\times (N-1)},\ \bm{b}\in \R^m\}
   \end{equation}
   which is
   constructed based on 
   available information or prior judgment, 
and $\mathcal{P}^0$ is 
bounded. 
The preference elicitation process should be understood as interactions between the DM and a modeler.   


\subsection{Preference elicitation process based on pairwise comparisons}\label{sec:elict}

In the literature of preference learning, various 
methods have been proposed 
to elicit a DM's preference. Here we focus on 
a pairwise comparison approach which is based on interactions between the DM and 
a modeler.
In the process, the modeler designs a pair of 
products $A$ and $B$ (called a query) and asks the DM's preference. 
An example of a query is shown in Figure~\ref{fig:query}.
Throughout this paper, we make the following assumption.
\begin{assumption}
There exists a unique true utility function $u^*$ 
which characterizes 
the DM's 
preferences and there
is no error in the process of eliciting preference, whether in the DM's response or in the observation of the DM's response by the modeler.
It means that the DM prefers $A$ to $B$ if and only if $u^*(\x^A)\geqslant u^*(\x^B)$. 
\end{assumption}
\noindent
Figure~\ref{fig:elic-process} depicts 
the interactions between the DM and the modeler.
The elicitation process is
adaptive in the sense that 
the currently elicited preference information is incorporated in the design  
of the next query.
The general procedure
can be described as follows.

\begin{itemize}
    \item[Step  1.] Initialization:
The modeler constructs an initial ambiguity set $\mathcal{P}^0$ of $\vv^*$ defined in \eqref{eq:initial-polyhedralset}. Set the number of iterations $k :=0$.
   
 \item[Step 2.] Design of query:
based on the current ambiguity set, the modeler designs a new query $(A^k,B^k)$ with attribute vector $\x^{A^k},\x^{B^k}\in \mathbb{R}^N$. 

 \item[Step 3.] Observation of the DM's preference:
the DM is given the query and makes a choice. Let
\begin{equation}\label{eq:observe-response}
\ \mu^k=\left\{\begin{array}{ll}-1, & \inmat{ if } u^*(x^{A^k})\geqslant u^*(x^{B^k}),\\
1, &\inmat{ if } u^*(x^{A^k})\leqslant u^*(x^{B^k}).
\end{array}
\right.
\end{equation}
 \item[Step 4.] Update of the ambiguity set:
the modeler updates the ambiguity set $\mathcal{P}^{k-1}$ after obtaining the DM's response to query 
$({\bm x}^{A^k},{\bm x}^{B^k})$
by incorporating the region consistent with the response,
\begin{equation}
\label{eq:poly-cut-market}
\mathcal{P}^{k}:=\mathcal{P}^{k-1}
\cap\Big\{{\vv}\in \R^{N-1}|\mathcal{H}^k(\vv) :=\mu^k\left[(\RR\vv)^{\top}{\bm x}^{A^k}- (\RR\vv)^{\top}{\bm x}^{B^k}\right]\leqslant0\Big\},
\end{equation}
where $\mathcal{H}^k(\vv)=0$ is a separating hyperplane 
cutting 
$\mathcal{P}^{k-1}$ into two sub-polyhedra.


 \item[Step 5.]	Stopping criterion:
 if $\mathcal{P}^{k}$ 
 reaches a prescribed 
 size, then stop the iteration and output $\mathcal{P}^{k}$; otherwise, set $k :=k+1$ and go to Step 2.
\end{itemize}




\begin{figure}[htb]
    \centering
    \includegraphics[width=0.5\textwidth]
    {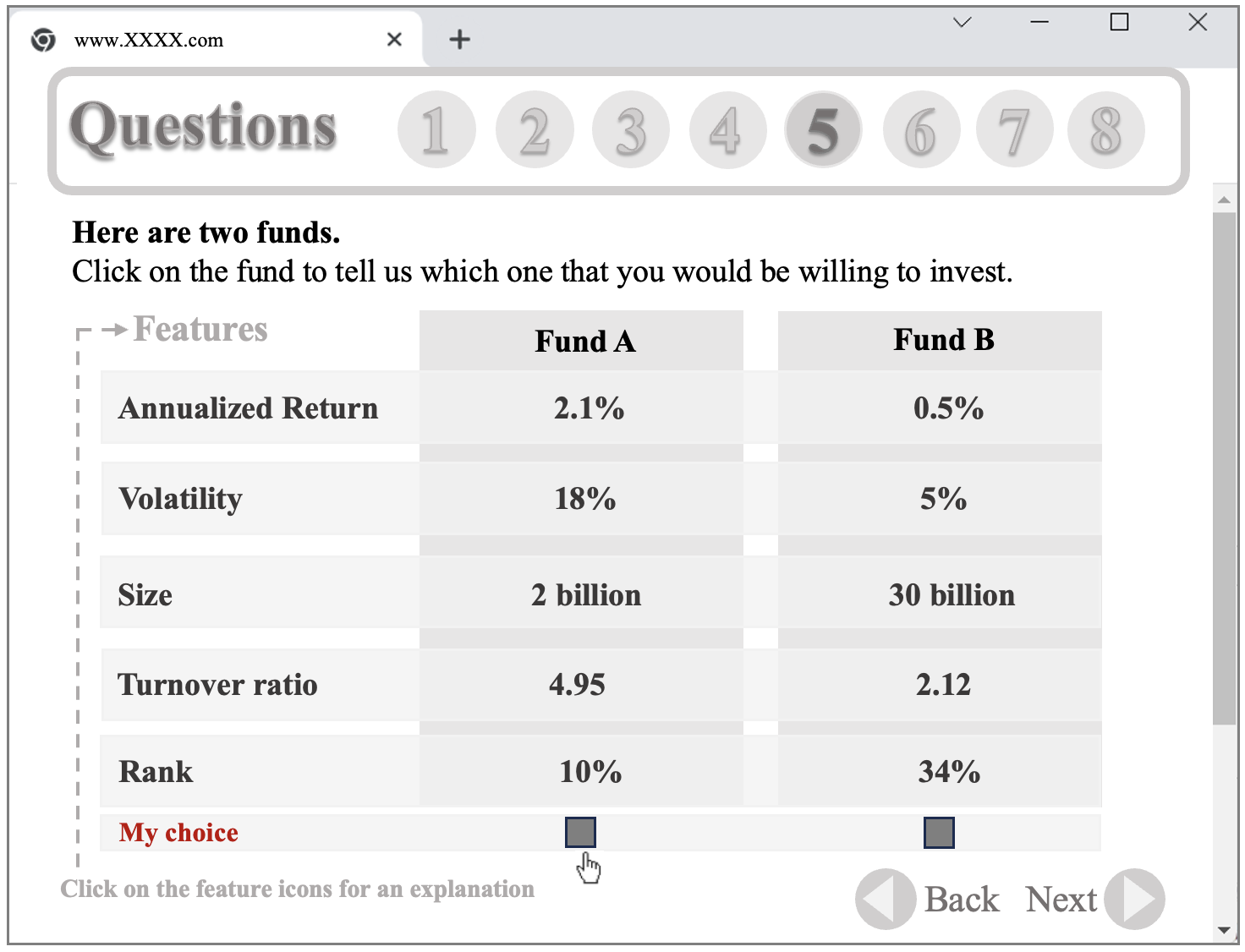}
    \caption{Example of a query for funds}
    \label{fig:query}
\end{figure}
\begin{figure}[htb]
    \centering
    \includegraphics[width=0.8\textwidth]{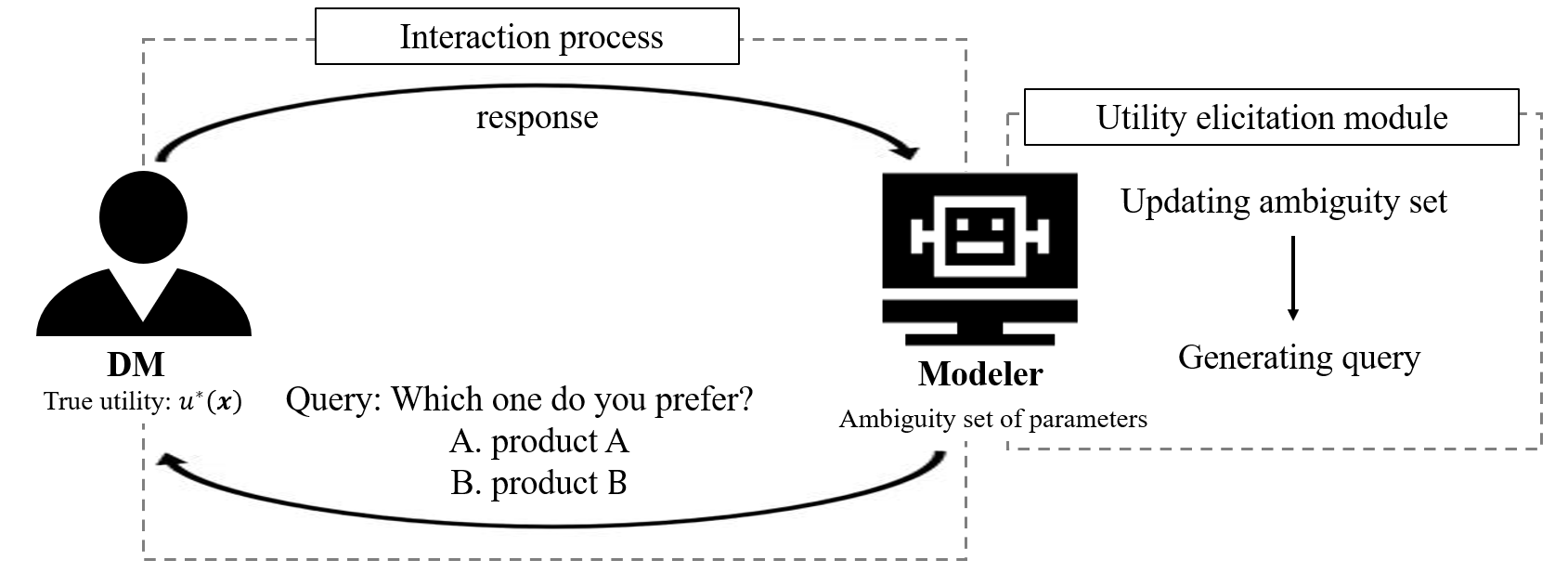}
    \caption{The process of utility elicitation between the modeler and the DM}
    \label{fig:elic-process}
\end{figure}

The main challenge in the procedure is in the design of queries that enable the modeler to reduce the ambiguity set efficiently. 
\citet{toubia2004polyhedral,toubia2007probabilistic} select queries by exploiting the geometry of the current polyhedron through its analytic center. They propose an optimal query design framework that is guided by four practical criteria including nondominance, feasibility, choice balance and post‑choice symmetry.
The resulting algorithm constructs two profiles whose separating hyperplane passes through the analytic center and is approximately orthogonal to the longest axis of an outer ellipsoid, thereby aiming for near‑balanced cuts. We will return to this in detail in the next subsection.
\citet{bertsimas2013learning} retain these geometric criteria but embed them in a robust mixed‑integer formulation with binary “flip” variables, so that queries remain informative even when a bounded number of past answers are inconsistent. 
\citet{saure2019ellipsoidal} propose a Bayesian approach that quantifies 
the epistemic uncertainty about $\vv^*$ 
with a multivariate
distribution that is updated after the DM's responses to queries. By using normal approximations to posterior distributions, they
can include response error in an approximate Bayesian approach that is as intuitive as the
polyhedral approach and allows the use of effective optimization-based techniques for adaptive
question selection.

\subsection{The polyhedral method 
}
\label{sec:toubia}

In this subsection, we 
recall the 
mathematical framework
of the polyhedral cut method introduced by \citet{toubia2004polyhedral} 
as our proposed CPM is closely related to it. We begin by  recalling the definitions of 
analytic center and Sonnevend's ellipsoid of a bounded convex polyhedron.

\begin{definition}
[Analytic Center (AC) and Sonnevend’s ellipsoid]
\label{def:AC}
Consider a bounded
polyhedron 
$\mathcal{P}=\{\vv\in\R^{N-1}\mid {\bm a}_i^\top {\bm v} \leqslant {b}_i, i=1,\ldots,m\}$.
The analytic center $\tilde{\bm{c}}$ of 
the polyhedron  
is defined as an optimal 
solution of the minimization problem $-\sum_{i=1}^m \log({b}_i-{\bm a}_i^\top {\bm v})$.
Given the analytic center $\tilde{\bm{c}}$ of $\mathcal{P}$, the Sonnevend’s (outer) ellipsoid of polyhedron $\mathcal{P}$ is
$$E_{\text{Sonn}}:=\Big\{\vv\in \R^{N-1}:(\vv-\tilde{\bm{c}})^\top\bm{H}(\vv-\tilde{\bm{c}})\leqslant m(m-1)\Big\},$$
where $\bm{H}=\sum^{m}_{i=1}\frac{1}{(b_i-\bm{a}_i^\top\tilde{\bm{c}})^2}\bm{a}_i^\top \bm{a}_i.$
\end{definition}

From the definition, we can see that the analytic center is the point that maximizes the product of its
distances to 
all
hyperplanes ${\bm a}_i^\top {\bm v} = {b}_i$,
$i=1,\ldots,m$.
By Definition~\ref{def:AC},
the analytic center $\tilde{\bm{c}}$ of polyhedron $\mathcal{P}$ is the optimal solution ${\bm v}$ of 
the following problem 
\begin{subequations}
\label{eq:ana_center}
\begin{align}
\max\limits_{{\bm v},{\bm s}}~~& \sum_{i=1}^{m} \log(s_i) \\
{\rm s.t.}\quad & {\bm A} {\bm v}+{\bm s}={\bm b}, {\bm s}\geqslant 0,
\end{align}
\end{subequations}
where ${\bm s}:=(s_1,\cdots,s_{m})^\top \in \R^{m}$ are slack variables.
\citep{toubia2004polyhedral}  propose a polyhedral method 
which aims to elicit a multivariate linear utility with 
feature vector $\x$ in a prespecified feasible set $\XX$.
\bgeqn 
u^*(\x)=(\RR\vv)^\top\x,\ \RR\vv\in\R^N.
\edeqn 
They 
propose a specific iterative method for generating queries and implementing polyhedral cuts as follows:
at the $k$-th iteration,
\begin{itemize}

\item compute the AC of the current polyhedron $\mathcal{P}^{k-1}$, denoted by $\tilde{\bm c}^{k-1}$, construct the associated Sonnevend's ellipsoid, and identify its longest axis;

\item 
identify the two intersections ${\bm v}_1^{k-1}$ and ${\bm v}_2^{k-1}$ between
the longest axis and the boundary of the polyhedron $\mathcal{P}^{k-1}$;

\item select ${\bm x}^{A^k}$ and
${\bm x}^{B^k}$ by solving two  
optimization
problems such that the separating hyperplane divides the region into two approximately equal-sized 
sub-polyhedra:
\begin{subequations}
\begin{align}
& {\bm x}^{A^k}\in \mathop{\arg\max} 
\Big\{(\RR{\bm v}_1^{k-1})^{\top}{\bm x}\mid \x\in\XX,\;
{(\RR\tilde{\bm{c}}^{k-1})}^{\top}{\bm x} \leqslant D\Big\},\label{eq:original_PM_a}\\
& {\bm x}^{B^k}\in \mathop{\arg\max} 
\Big\{(\RR{\bm v}^{k-1}_2)^{\top}{\bm x}\mid \x\in\XX,\;
(\RR{\tilde{\bm{c}}^{k-1}})^{\top}{\bm x} \leqslant  D\Big\},
\label{eq:original_PM_b}
\end{align}\label{eq:original_PM}
\end{subequations}
where $D>0$ is a constant drawn from a uniform distribution over some interval and re-drawn until ${\bm x}^{A^k}\neq {\bm x}^{B^k}$.

\end{itemize}
The polyhedral cut method
is adaptive in the sense that 
the design of queries \eqref{eq:original_PM_a} and \eqref{eq:original_PM_b}
depends on ${\cal P}^{k-1}$ via 
$\vv_1^{k-1}$ and $\vv_2^{k-1}$.
It
works very well in practical applications. However, as far as we are concerned, 
there is
no theoretical guarantee of convergence 
for the following two primary reasons.
\begin{itemize} 
\item 
\textbf{Lack of angular separation between two consecutive  cuts.} The 
cut in each iteration is constructed based on the longest axis of the Sonnevend's ellipsoid of the current feasible region. 
It passes through the analytic center 
and yields two subregions of roughly equal volume (known as the criteria of choice and post-choice symmetry). 
However, the mechanism to generate queries \eqref{eq:original_PM} does not guarantee that the separating plane is perpendicular to the longest axis (since ${\bm x}^{A^k}-{\bm x}^{B^k}$ 
is not necessarily parallel to $\vv_1^{k-1}-\vv_2^{k-1}$)
nor that 
any two consecutive separating planes are perpendicular.
This is primarily because
the queries are  designed by
solving the optimization 
problems \eqref{eq:original_PM_a}-\eqref{eq:original_PM_b} rather than
 specified geometric cuts.



\item \textbf{Uncontrolled radius of the outer ellipsoid.} 
The method 
designs 
queries by constructing an outer ellipsoid centered on the analytic center that encloses the polyhedron. 
This ellipsoid is 
useful for identifying 
cut directions, but it is not designed to minimize the worst-case Euclidean distance from its center to any points in the ambiguity set. 
Consequently, the resulting sequence of cuts does not yield a deterministic contraction bound for the diameter/radius of the polyhedron (ambiguity set of the partworth vectors).

\end{itemize}
This prompts us to propose the CPM.





\section{CPM for linear utility function}
\label{sec:coordinate-wise}


The main purpose of the Coordinate-wise Polyhedral Method
is to 
avoid 
two consecutive 
cutting planes being nearly parallel. 
Specifically,
at each iteration, we conduct $N-1$ cuts
and the normal vectors of the cutting planes are mutually orthogonal. This ensures that
the size of the polyhedron after cutting is reduced to a fixed proportion of the previous one. 

\subsection{Minimum enclosing ball and its center}

We begin by introducing 
the concepts of the minimum enclosing ball 
since we need it to construct cuts and estimate 
the size of the polyhedron at the end of each iteration of coordinate-wise cuts.

\begin{definition}[Minimum enclosing ball (MEB) of polyhedron 
\citep{yildirim2008two}
]
\label{def:circle}
Let  
$\mathcal{P}=\{\vv\in\R^{N-1}\mid A\vv\leqslant b\}$
be a bounded polyhedron in $\R^{N-1}$.
The minimal 
ball, written $\mathcal{B}$,
enclosing $\mathcal{P}$ is 
\bgeqn 
\mathcal{B}(\cc,r)=\{\vv\in\R^{N-1}\mid \|\vv-\bm{c}\|_2\leqslant r\},
\edeqn 
where the center of the MEB, written 
$\cc$, and the radius $r$ are defined by
\begin{subequations}
    \begin{eqnarray}
  \{\cc,r\}= & \mathop{\arg\min}\limits_{\cc\in\R^{N-1},r\in\R}   & r \\
   &\rm{s.t.}  & \|\vv-\cc\|_2\leqslant r, \ \forall\vv\in \mathcal{P}.
    \end{eqnarray}
\end{subequations} 
\end{definition}

Below, we give two technical results concerning the center of MEB.


\begin{lemma}\label{lemma:MEBcenter}  
Let $\Omega \subset \mathbb{R}^N$ be a nonempty bounded closed convex set, and let $\cc$ be the center of the minimum enclosing ball of $\Omega$. Then $\cc\in\Omega$.
\end{lemma}

{\color{black}
Using this lemma, we can determine the center of MEB 
containing any remaining part  
of a ball in $\R^{N-1}$ after the ball is cut
by $(N-1)$ mutually orthogonal hyperplanes
passing through center $\cc$. This result is
essential to the proof of convergence of CPM discussed in Section~\ref{sec:convergence}.}


\begin{proposition}\label{prop:C_radius-reduction}
Let 
$
\mathcal B(\cc,r)=\{\vv\in\R^{N-1}:\|\vv-\cc\|_2\leqslant r\}
$
be the closed Euclidean ball with center $\cc\in\R^{N-1}$, radius $r>0$ and $N\geqslant3$.
Partition $\mathcal B(\cc,r)$ by
${N-1}$ mutually orthogonal hyperplanes passing
through the center $\cc$ which are coordinate-wise hyperplanes 
$\{\vv:\;v_i-c_i=0\}$, 
$i=1,\dots,{N-1}$. Let
$
\mathcal O:=\{\vv\in\mathcal B(\cc,r):\;v_i-c_i\geqslant0\ \text{for all }i\}
$
be one of the $2^{N-1}$ congruent orthants obtained. Then the minimum enclosing
ball of $\mathcal O$
is centered at
\bgeqn 
\label{eq:center-update-after-cut}
\cc^* = \cc + \frac{r}{{N-1}}\bm{e},\; \text{where} \; \bm{e}=(1,\dots,1)^\top
\edeqn 
with radius
\bgeqn 
\label{eq:radius-reduction}
r^* = r\sqrt{1-\frac{1}{{N-1}}}.
\edeqn 
\end{proposition}

\subsection{Structure of pairwise comparison queries and the CPM algorithm}

Analogous to the polyhedral method, we use pairwise 
comparison query
to elicit a DM's preferences.
However, the way we construct 
the queries will be completely different.
Let
\bgeqn 
\label{eq:Sep-pln-H}
\mathcal{H}=\{\vv\in\R^{N-1}\mid \bal^\top(\vv-\vv_0)=0\}
\edeqn 
be a given hyperplane passing through $\vv_0\in \R^{N-1}$. To ease the exposition, 
we call $\bal$ the {\em cut direction}.
${\cal H}$ acts as a separating hyperplane
of  query $(A,B)$ 
  if it satisfies the following: 
    \begin{subequations}
        \begin{align}
             &{\RR\vv}^\top(\x^A-\x^B)=0,\ \forall \vv\in {\cal H},\label{eq:sep_def1}\\
             &[(\RR\vv^{+})^\top(\x^A-\x^B)][(\RR\vv^{-})^\top(\x^A-\x^B)]
            \leqslant0,\ \forall \vv^+\in {\cal H}_{+},\ \vv^-\in {\cal H}_{-},\label{eq:sep_def2}
        \end{align}\label{eq:sep_def}
        \end{subequations}
        where 
        ${\cal H}_{+}=
        \{\vv\mid\bal^{\top}(\vv-\vv_0)  \geqslant0\}$ 
        and 
        ${\cal H}_{-}
        =
        \{\vv\mid\bal^{\top}(\vv-\vv_0)\leqslant0\}.
        $
        We are now ready to describe 
        the CPM algorithm 
        in Algorithm \ref{alg:VCPM_LU}.




\vspace{2mm}
\begin{breakablealgorithm}\label{alg:VCPM_LU}
\caption{CPM for eliciting a linear utility function} 
\begin{flushleft}
\textbf{Input:} initial 
polyhedron
$\mathcal{P}^0\in\R^{N-1}$, iteration index $k = 0$, 
and precision $\epsilon$.\\
\end{flushleft}
\begin{itemize}[leftmargin=0pt,label={},itemsep=0.6ex]
\item\textbf{Step 1.} 
Calculate the MEB center $\bm{c}^{k}$ and radius $r^k$ of ${\cal P}^k$.
\item\textbf{If} $r^k>\epsilon$ \textbf{do}
\item\textbf{Step 2.}
Identify $\bal^{k,1},\bal^{k,2},\ldots,\bal^{k,N-1}\in \R^{N-1}$ 
 which 
 form a normalized orthogonal basis 
  in $\R^{N-1}$.

\item\textbf{Step 3.} 
  \textbf{For} $i=1,\ldots,N-1$ \textbf{do}
    \begin{itemize}
     \item[-] Design a pairwise comparison query $(A^{k,i},B^{k,i})$
       with attribute vectors $(\x^{A^{k,i}},\x^{B^{k,i}})$ 
       corresponding to hyperplane
     \begin{equation} 
     {\cal H}^{k,i}=\{\vv\mid(\bal^{k,i})^\top(\vv-\bm{c}^k)=0\}.
     \end{equation}
 

    

\item[-] 
Ask the DM to choose between $A^{k,i}$ and $B^{k,i}$.
Denote by $\mu^{k,i}=-1$ if $A^{k,i}$ is chosen,
i.e., $$(\RR\vv^*)^\top \x^{A^{k,i}}{\geqslant}(\RR\vv^*)^\top \x^{B^{k,i}},$$  
and $\mu^{k,i}=1$ otherwise, see \eqref{eq:observe-response}.

  \item[-] Update the polyhedron $\mathcal{P}^{k,i-1}$:
  \begin{equation}
  \label{eq:update_rule}
             \mathcal{P}^{k,i} :=\mathcal{P}^{k,i-1}\cap 
             \Big\{\vv\mid
               \mu^{k,i}(\RR\vv)^\top (\x^{A^{k,i}}-\x^{B^{k,i}})\leqslant0\Big\}.
             \end{equation}
    \end{itemize}

    
    \item\textbf{Step 4.} Set $\PP^{k+1} :=\PP^{k,N-1}$ and $k :=k+1$. Go back to Step 1.
\end{itemize}
\begin{flushleft}
    \textbf{Output:} $\PP^{k}$.
\end{flushleft}
\end{breakablealgorithm}
\vspace{2mm}

In Step 2, the selection of $N-1$ orthogonal cut directions is not unique. 
Indeed, the convergence analysis of the algorithm (see Theorem \ref{theo:convergence_linear} later)
relies only on the orthogonality 
of the basis rather than the specific 
selection of the basis. One way to construct the orthogonal basis is to 
set $\bal^{k,1}$ as the direction of the longest axis of the minimum outer ellipsoid $\mathcal{E}^k$ (MOE) of $\mathcal{P}^k$, centered at the MEB center $\cc^k$:
\begin{equation}\label{def:MOE}
\mathcal{E}^k(\cc^k,\bm{Q})=\{\vv\in\R^{N-1}\mid(\vv-{\bm{c}}^k)^\top \bm{Q}^{-1}(\vv-{\cc^k})\leqslant 1\},
\end{equation} 
where the matrix $\bm{Q}$ solves 
\begin{subequations}
    \begin{eqnarray}
       & \min\limits_{\bm{Q}\in \mathbb{S}^{N-1}_{++}}\ \ & \log\det(\bm{Q})  \label{eq-basis-elp-obj}\\
       & \text{\rm s.t.}  &(\vv-{\bm{c}^k})^\top \bm{Q}^{-1}(\vv-{\cc^k})\leqslant 1, \ \forall\vv\in\mathcal{P}^k,
    \end{eqnarray}
\end{subequations}
Let $\{\lambda_i\}_{i=1}^{N-1}$ denote the eigenvalues of  $\bm Q$. The direction of the longest axis of $\mathcal{E}^k(\cc^k,\bm Q)$ is given by the eigenvector associated with the largest eigenvalue $\lambda_{\max}:=\max_i\lambda_i$ of $\bm Q$.
Then $\bal^{k,1}$ is set to be the unit eigenvector corresponding to $\lambda_{\max}$. After determining $\bal^{k,1}$, we select the remaining basis vectors recursively as follows.
Let ${\bm e_1,\ldots,\bm e_{N-1}}$ denote the coordinate-wise canonical basis of $\mathbb{R}^{N-1}$ with a fixed order, 
where $\bm e_j=(0,\dots,0,\underset{j\text{-th}}{1},0,\dots,0)^\top,\; j=1,\ldots,N-1$. 
Suppose that $\bal^{k,1},\ldots,\bal^{k,p}$ have already been constructed for $1\leqslant p\leqslant N-2$. To construct $\bal^{k,p+1}$, 
we 
compute
$\tilde{\bal}_{j}
:=
\bm e_j-\sum_{\ell=1}^{p}(\bm e_j^\top\bal^{k,\ell})\bal^{k,\ell}$
from $j=1$. Once an index $j$ satisfies $\|\tilde{\bal}_{j}\|_2>0$,  set
$\bal^{k,p+1}:=\frac{
\tilde{\bal}_j
}{\|
\tilde{\bal}_j
\|_2}$.
and then proceed to constructing the next direction.

There are potentially other ways to construct an orthogonal basis. 
For instance, After determining $\bal^{k,1}$, we may project the polyhedron $\mathcal P^k$ and the center $\cc^k$ onto the orthogonal complement of $\bal^{k,1}$. The projected set remains a polyhedron, and the projected center lies in the projected polyhedron although 
it is not necessary the MEB center of the latter. 
We then construct the MOE of the projected polyhedron in this lower-dimensional subspace, centered at the projected point, and choose its longest-axis direction as $\bal^{k,2}$. 
Continuing the process, 
once $\bal^{k,1},\ldots,\bal^{k,i-1}$ have been determined, we project $\mathcal P^k$ and $\cc^k$ onto the orthogonal complement of $\operatorname{span}\{\bal^{k,1},\ldots,\bal^{k,i-1}\}$, construct the MOE of the projected polyhedron centered at the projected point, and choose its longest-axis direction as $\bal^{k,i}$.
It is worth noting that the MOE used here is constrained to be centered at the MEB center. Therefore, its center may differ from that of the minimum-volume enclosing ellipsoid (see \cite{damla2008linear}, \citet{todd2016minimum}).

In Step 3, each $\mathcal H^{k,i}$ passes through $\cc^k$, and the $N-1$ hyperplanes are mutually orthogonal. The coordinate-wise 
cuts reduce 
the ambiguity set
in all directions simultaneously and guarantee a strict decrease 
in the radius of the MEB after each major iteration.
Figure~\ref{fig-process} illustrates the CPM cutting process when the orthogonal basis is designed based on the MOE. 
With a specified orthogonal basis, we move on to discuss 
how to design 
pairwise comparison query in Step~3 of the algorithm. The next theorem addresses this.

\begin{figure}[htbp]
  \centering
  \begin{subfigure}[t]{0.32\textwidth}
    \centering
    \includegraphics[width=\textwidth]{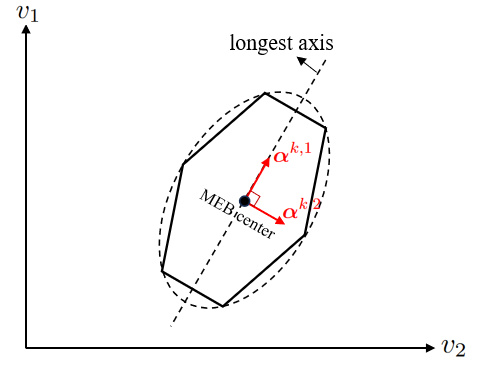}
    \caption{}\label{fig-process-1}
  \end{subfigure}
  \hfill
  \begin{subfigure}[t]{0.32\textwidth}
    \centering
    \includegraphics[width=\textwidth]{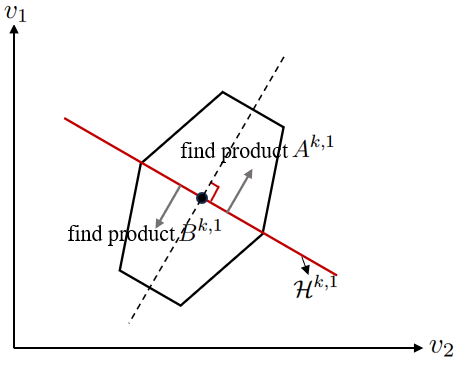}
    \caption{}\label{fig-process-2}
  \end{subfigure}
  \hfill
  \begin{subfigure}[t]{0.32\textwidth}
    \centering
    \includegraphics[width=\textwidth]{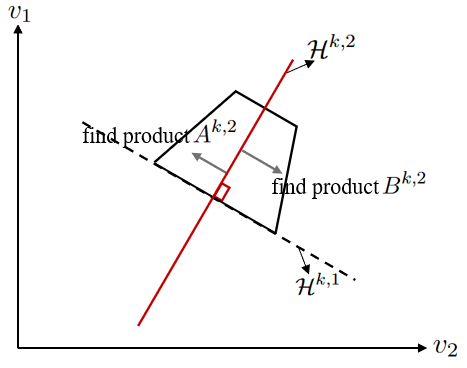}
    \caption{}\label{fig-process-3}
  \end{subfigure}
  \caption{\footnotesize Procedures of CPM. (a) Find the MEB center, the longest axis of MOE and orthogonal directions. (b) Find query and implement cutting. (c) Find coordinate-wise cutting plane and the query.}
  \label{fig-process}
\end{figure}

\begin{theorem}[Design of queries]
\label{thm:find_query}
Let ${\cal H}=\{\vv\in \R^{N-1}\mid\bal^\top(\vv-\bm{c})=0\}$ be a given hyperplane. Then there exists a pair of products $A,B$ with attribute vectors $\x^A,\ \x^B\in \R^N$  such that the hyperplane ${\cal H}$
 is the separating hyperplane of the query $(A,B)$, 
 where
\begin{subequations}\label{eq:def_x}
    \begin{align}
        x^A_N-x^B_N&=-\sum^{N-1}_{i=1}\alpha_i c_i,\label{eq:query_A}\\
        x^A_i-x^B_i&=(x^A_N-x^B_N)+\alpha_i=-\sum^{N-1}_{i=1}\alpha_i c_i+\alpha_i,\; \text{for}\; i=1,\ldots,N-1.\label{eq:query_B}
    \end{align}
\end{subequations}
\end{theorem}




\begin{remark}[Criterion of determining $\x^A$ and $\x^B$]
\label{remark:query_choose}
Let $\Delta:=\x^A-\x^B$ denote the vector of the differences between the attributes 
of the two products. Using $\Delta$, we can recast 
the  system of equalities \eqref{eq:def_x} 
in terms of 
$\Delta$.
The solution $\Delta$ of the system is unique. However, 
the  vectors of attributes $\x^A$ and $\x^B$
are not:
for any $\gamma>0$ and
$\bm{z}\in\R^N$ the pair
$(\gamma\x^A+\bm{z},\; \gamma\x^B+\bm{z})$ yields 
$\gamma\Delta$
and hence the same
separating hyperplane. 
This raises a question as to 
how to identify an appropriate 
pair of feasible profiles ($\x^A$ and $\x^B$).
Here we provide some guidelines.
First, boundedness. The attribute vectors should not be arbitrarily large for comparisons.
We can limit the value of attributes in a prespecified feasible set $\XX$. 
Second, sparsity.
The vectors are chosen to be as sparse as possible.
For example,  
a pair of profiles
$\x^A = (3,\,2,\,1,\,1,\,1)^\top$ and 
$\x^B = (2,\,3,\,1,\,1,\,1)^\top$ gives 
the same $\Delta$ as the pair $\hat{\x}^A = (1,\,0,\,0,\,0,\,0)^\top,
\hat{\x}^B = (0,\,1,\,0,\,0,\,0)^\top$. The latter is preferable
since it is easier to calculate the related utility values.
On the other hand, since the non-zero parts of  $\hat{\x}^A$
and $\hat{\x}^B$ differ, it might make it difficult for DMs (customers) to choose between the two products. 
Third, comparability versus practicality. 
The profiles should be as 
comparable as possible. 
For example, 
 $\x^A = (5,\,3,\,4)^\top$ and 
$\x^B = (4,\,4,\,4)^\top$ gives 
the same $\Delta$ as the pair $\hat{\x}^A = 
(7.3,\,4.8,\,6.1)^\top,
\hat{\x}^B = (6.3,\,5.8,\,6.1)^\top$.
However, the 
former makes it relatively easier for the DM/customer to make a choice whereas the latter might characterize products more realistically.

\end{remark}

\subsection{Estimation of the MEB center
and convergence  of Algorithm~\ref{alg:VCPM_LU}}\label{sec:convergence}

After $k(N-1)+i$ queries, the 
initial  
polyhedral 
set for the partworth vector is reduced to  $\mathcal{P}^{k,i},\ i=1,\ldots,N-1$.
We calculate the MEB center $\bm{c}^{k,i}$ of $\mathcal{P}^{k,i}$ 
and use it as an estimate 
of the 
true utility partworth vector. We refer to 
the
corresponding
utility function $u^{k,i}(\x)=(\RR\bm{c}^{k,i})^\top\x$ as the nominal utility function.
A potential advantage 
of using the MEB center as opposed to 
the AC 
is that it provides 
a better estimate in the case where the true partworth vector is located at the boundary of ${\cal P}^{k,i}$, see Figure~\ref{fig:MOE} for an illustration.
This is because by definition,
$$
\sup_{\vv \in {\cal P}^{k,i}}\|\vv-\bm{c}^{k,i}\|_2
\leqslant 
\sup_{\vv \in {\cal P}^{k,i}}\|\vv-\bm{c}_{AC}^{k,i}\|_2,
$$
thus a ball centered at the MEB center provides 
a smaller cover of the polyhedron than the one centered at the AC. 
In the discussions below, we will exploit this geometric structure to
derive the theoretical convergence of
CPM whereas we cannot do this for one based on the AC.

\begin{figure}
    \centering  \includegraphics[width=0.5\linewidth]{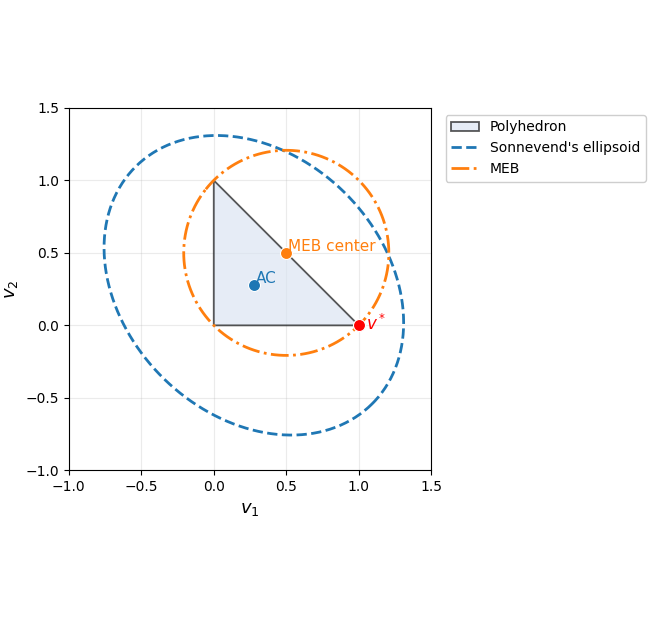}
    \caption{MEB  vs Sonnevend's ellipsoid}
    \label{fig:MOE}
\end{figure}


    

\begin{theorem}
    [Convergence of Algorithm~\ref{alg:VCPM_LU}]
    \label{theo:convergence_linear}
    Assume that the initial ambiguity set $\mathcal P^0\subset\R^{N-1}$ is a  bounded polyhedron containing the true partworth $\vv^*$.
    Let $\mathcal B^0$ denote the minimum enclosing ball of $\mathcal P^0$ with radius $r(\mathcal B^0)$. Then 
    for $i\in\{1,\dots,N-1\}$,
${\cal P}^{k+1}={\cal P}^{\,k,N-1}\subseteq \cdots \subseteq {\cal P}^{k,1}\subseteq {\cal P}^{k}\subset \mathcal B^k$
 and    
\bgeqn 
\label{eq:convergence_linear}
\| \vv - \vv^* \|_\infty\leqslant\| \vv - \vv^* \|_2 \leqslant\max\limits_{\vv',\vv''\in\mathcal{P}^{k,i}}\|\vv'-\vv''\|_2 \leqslant  2\, r(\mathcal B^0)\left(\sqrt{1-\frac{1}{N-1}}\right)^{\,k},
\qquad \forall\, \vv\in\mathcal P^{\,k,i}.
\edeqn 
In particular, $\mathcal P^{k}\to\{\vv^*\}$ as $k\to\infty$.


\end{theorem}

\begin{remark} 
\label{Rem:CPM-V-convergence}
Theorem~\ref{theo:convergence_linear} establishes the convergence of the CPM algorithm, and Theorem~\ref{thm:find_query} guarantees that, at each iteration, there exist queries corresponding to the prescribed cuts. 
The main difference between the CPM approach and the existing polyhedral method lies in 
how the queries/cuts are constructed.
CPM designs queries based on prespecified, geometrically defined coordinate‑wise cuts, whereas the polyhedral method first obtains queries by solving the optimization problems in \eqref{eq:original_PM} and then forms a cut according to \eqref{eq:observe-response}–\eqref{eq:poly-cut-market}. As noted earlier, the latter’s query‑generation mechanism 
does not guarantee perpendicularity between consecutive separating planes. Consequently, it cannot ensure a fixed reduction in the size of the polyhedron at each iteration despite excellent empirical convergence \cite{saure2019ellipsoidal}.
A potential drawback of CPM is that it requires $N-1$ cuts at each major iteration, which may lead to 
a lower overall average 
reduction per cut.


Note also that the convergence rate in the last inequality of \eqref{eq:convergence_linear} relies on the MEB-based radius reduction, but not on the specific choice of the orthonormal basis. It is possible to replace MEB with an MOE or a minimum-volume enclosing ellipsoid \citep{damla2008linear} to guide the construction of coordinate-wise cuts. We expect that such choices may affect the upper bound in \eqref{eq:convergence_linear}, but not necessarily the rate of shrinkage of the polyhedra. A detailed comparison of different ellipsoid-guided algorithm is left for future research.

\end{remark}


Theorem \ref{theo:convergence_linear}
allows us to derive the complexity of the CPM procedures in terms of 
the number of queries required to reduce the size of 
the initial ambiguity set/polyhedron ${\cal P}^0$ to a specified 
precision. The next corollary states this.

\begin{corollary}[Maximum number of queries]
\label{cor:sample-complex}
With at most
$(N-1)\left\lceil\frac{2[\log\epsilon-\log2r(\B^0)]}{\log(N-2)-\log(N-1)}\right\rceil$
queries, 
the maximum error is bounded by $\epsilon$, i.e.,
$\max\limits_{\vv\in\PP^{k,i}} \|\vv-\vv^*\|_2\leqslant \epsilon$.
\end{corollary}

This bound on the error ensures the consistency of the signs of the components of
$\vv$ with
the signs of the ground-truth ($\vv^*$), a metric 
to be used to examine the quality/precision of $\vv$ as an estimate of $\vv^*$.

\begin{corollary}[{\color{black}
Maximum number of queries for partworth sign correctness}]\label{cor:sign correctness}
Let $v_{\min}^*:=\min\limits_{i\in\{1,\ldots,N\}}\{|v_i^*|:v^*_N=1-\sum_{i=1}^{N-1}v^*_i\}$.
With at most
$k=\left\lceil\frac{2
[\log v_{\min}^*-\log2(N-1)r(\B^0)]}{\log(N-2)-\log(N-1)}\right\rceil$ major iterations (equivalently, $k(N-1)$ queries), 
it is guaranteed that
$v_iv_i^*\geqslant0$, $i=1,\cdots,N$, for any 
$\RR\vv,\;\vv\in\PP^{k+1}$, 
that is, the signs of the components of the partworths in the current ambiguity set 
are all consistent with the signs of the ground-truth.
\end{corollary}



Corollary~\ref{cor:sign correctness} characterizes the sample complexity required to achieve \underline{partworth sign correctness}. 
This metric is important in practice. For example, when a firm designs a new product, correctly identifying the sign of each partworth component provides a direct indication of whether an attribute should be increased or decreased. 
If the sign is inferred incorrectly, the resulting product design may move in the wrong direction and therefore fail to attract customers.
Moreover, Corollary~\ref{cor:sign correctness} shows that the sample complexity for sign correctness depends on the absolute magnitudes of the DM's true partworth components. 
In particular, if the DM has a pronounced preference for every attribute, so that $v_{\min}^*$ is relatively large, then fewer queries are needed to identify the signs correctly. 
By contrast, if the DM has only weak preferences on some attributes, so that $v_{\min}^*$ is close to $0$, then more queries are required. 
This is intuitive, because it is harder to distinguish the correct sign of a component when its true value is close to zero.


\subsection{An illustrative example\label{Sec:Multi_example}}

In this section, we present a simple academic example to illustrate the cutting process of
the coordinate-wise polyhedral method.
We simulate a virtual decision maker (DM) with true partworth
$\RR\vv^*=(0.2,0.5,0.3)^\top$, corresponding to the linear utility function
$u^*(\x)=0.2x_1+0.5x_2+0.3x_3.$
The initial uncertainty set is chosen as
$$
\mathcal{P}^0=\{\vv\in\R^2 \mid v_1\geqslant 0,\ v_2\geqslant 0,\ v_1+v_2\leqslant 1\}.
$$
\paragraph{Major iteration $k=0$.}
The MEB center of $\mathcal{P}^0$ is
$\cc^0=(0.5,0.5)^\top$, with lifted vector $\RR\cc^0=(0.5,0.5,0)^\top$.
We identify the basis based on MOE. Since the corresponding minimum outer ellipsoid (MOE) is a Euclidean ball,
we specify the first normal vector as
$\bal^{0,1}=(-0.78,\,0.62)^\top$, yielding the cutting plane
$\mathcal{H}^{0,1}=\{\vv\in\R^2 \mid (-0.78,0.62) [\vv-(0.5,0.5)^\top] = 0\}.$
Solving \eqref{eq:def_x} gives the attribute difference
$\Delta\x=(-0.7,\,0.7,\,0.08)^\top$.
One feasible query realizing this separating hyperplane is
$\x^A=(0,0.7,0.08)^\top$ and
$\x^B=(0.7,0,0)^\top$.


The second cut in this major iteration is orthogonal to the first.
The normal vector is
$\bal^{0,2}=(0.62,\,0.78)^\top$, yielding the hyperplane
$\mathcal{H}^{0,2}=\{\vv\in\R^2 \mid (0.62,0.78) [\vv-(0.5,0.5)^\top] = 0\}.$
Solving \eqref{eq:def_x} produces
$\Delta\x=(-0.08,\,0.08,\,-0.7)^\top$.
A feasible query is
$\x^A=(0,1,0)^\top$ and
$\x^B=(0.08,0.92,0.7)^\top$.
The cutting planes and the retained region after this major iteration are shown in $\R^2$ in Figure~\ref{fig:major1}(a) and the corresponding cutting planes in $\RR\vv\in\R^3$ are shown in Figure~\ref{fig:major1}(b)-(c).
Given the DM’s truthful response according to $u^*$, the half-space
containing the true partworth is preserved.
After completing the two coordinate-wise cuts, the updated uncertainty set
$\mathcal{P}^{0,2}$ is defined by
$$
\bm{A}^{0,2}=\bigl((-1,0);(0,-1);(1,1);(0.78,-0.62);(0.62,0.78)\bigr)\in\R^{5\times2},\quad
\bm{b}^{0,2}=(0,0,1,0.082,0.702)^\top.
$$

\paragraph{Major iteration $k=1$.}
Let $\mathcal{P}^1=\PP^{0,2}$. The MEB center of $\mathcal{P}^1$ is
$\cc^1=(0.052,0.448)^\top$, with lifted vector $\RR\cc^1=(0.052,0.448,0.5)^\top$.
The longest axis of the MOE is approximately $(0.79,-0.61)^\top$,
and thus the first normal vector in this iteration is chosen to be $\bal^{1,1}=(0.61,0.79)^\top$.
The corresponding cutting plane is
$
\mathcal{H}^{1,1}=\{\vv\in\R^2 \mid (0.61,0.79) [\vv-(0.052,0.448)^\top] = 0\}.
$
Solving \eqref{eq:def_x} yields
$\Delta\x=(0.22,0.40,-0.39)^\top$,
with a feasible query given by
$\x^A=(1,1,0)^\top$ and
$\x^B=(0.78,0.60,0.39)^\top$.

The second cut is orthogonal to $\bal^{1,1}$, with normal vector
$\bal^{1,2}=(0.79,-0.61)^\top$.
Solving \eqref{eq:def_x} gives
$\Delta\x=(0.51,-0.19,0.12)^\top$,
and one feasible query is
$\x^A=(1,0,1)^\top$ and
$\x^B=(0.49,0.19,0.88)^\top$.
After completing the two coordinate-wise cuts in this iteration,
the uncertainty set becomes $\mathcal{P}^{1,2}$, characterized by
\vspace{-5mm}
\begin{align*}
    &\bm{A}^{1,2}=\bigl((-1,0);(0,-1);(1,1);(0.78,-0.62);(0.62,0.78);
(0.61,0.79);(0.79,-0.61)\bigr)\in\R^{7\times2},\\ 
&\bm{b}^{1,2}=(0,0,1,0.082,0.702,-0.384,0.236)^\top.
\end{align*}


The cutting planes and the retained region after this major iteration are shown in $\R^2$ in Figure~\ref{fig:major2}(a) and the corresponding cutting planes in $\RR\vv\in\R^3$ are shown in Figure~\ref{fig:major2}(b)-(c).


Figure~\ref{fig:comparison}(a) summarizes the overall cutting process induced by CPM. For reference, we also apply the polyhedral method (POLY) by \citet{toubia2004polyhedral} to the same setting, and the corresponding cutting process is shown in Figure~\ref{fig:comparison}(b).
These figures illustrate the structural differences between the cutting mechanisms of the two methods.

As illustrated in Figure~\ref{fig:comparison}(a), CPM organizes cuts into major iterations, within each of which a group of coordinate-wise cuts is performed. The cutting planes in the same major iteration are mutually orthogonal and all pass through the current MEB center.
This design leads to a balanced contraction of the uncertainty set across all dimensions of the partworth vector and shrinks the polyhedron in a geometrically regular manner.
By contrast, Figure~\ref{fig:comparison}(b) shows that POLY generates a sequence of cuts that primarily shorten the feasible region along the direction of the longest axis at each step without angular constraints on successive cutting planes.

\begin{figure}[!htbp]
    \centering
    \begin{subfigure}[t]{0.32\textwidth}
        \centering
        \includegraphics[width=\textwidth]{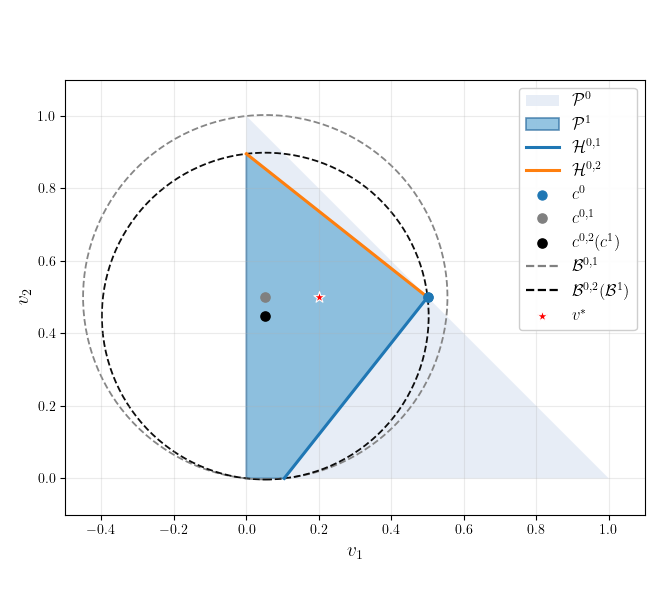}
        \caption{\label{fig:major1_2d}}
    \end{subfigure}
    \begin{subfigure}[t]{0.32\textwidth}
        \centering
        \includegraphics[width=\textwidth]{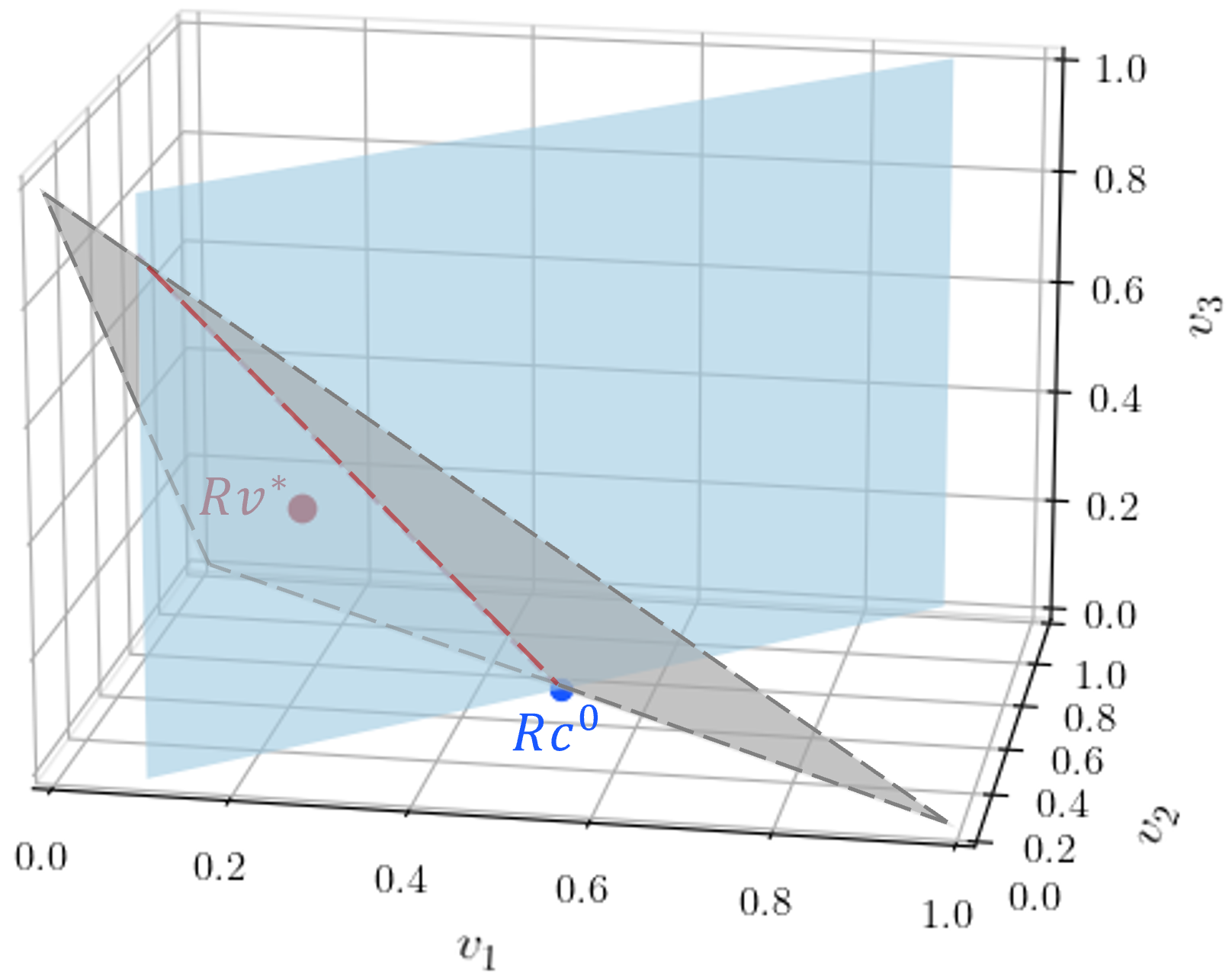}
        \caption{\label{fig:major_cut1_3d}}
    \end{subfigure}
    \begin{subfigure}[t]{0.32\textwidth}
        \centering
        \includegraphics[width=\textwidth]{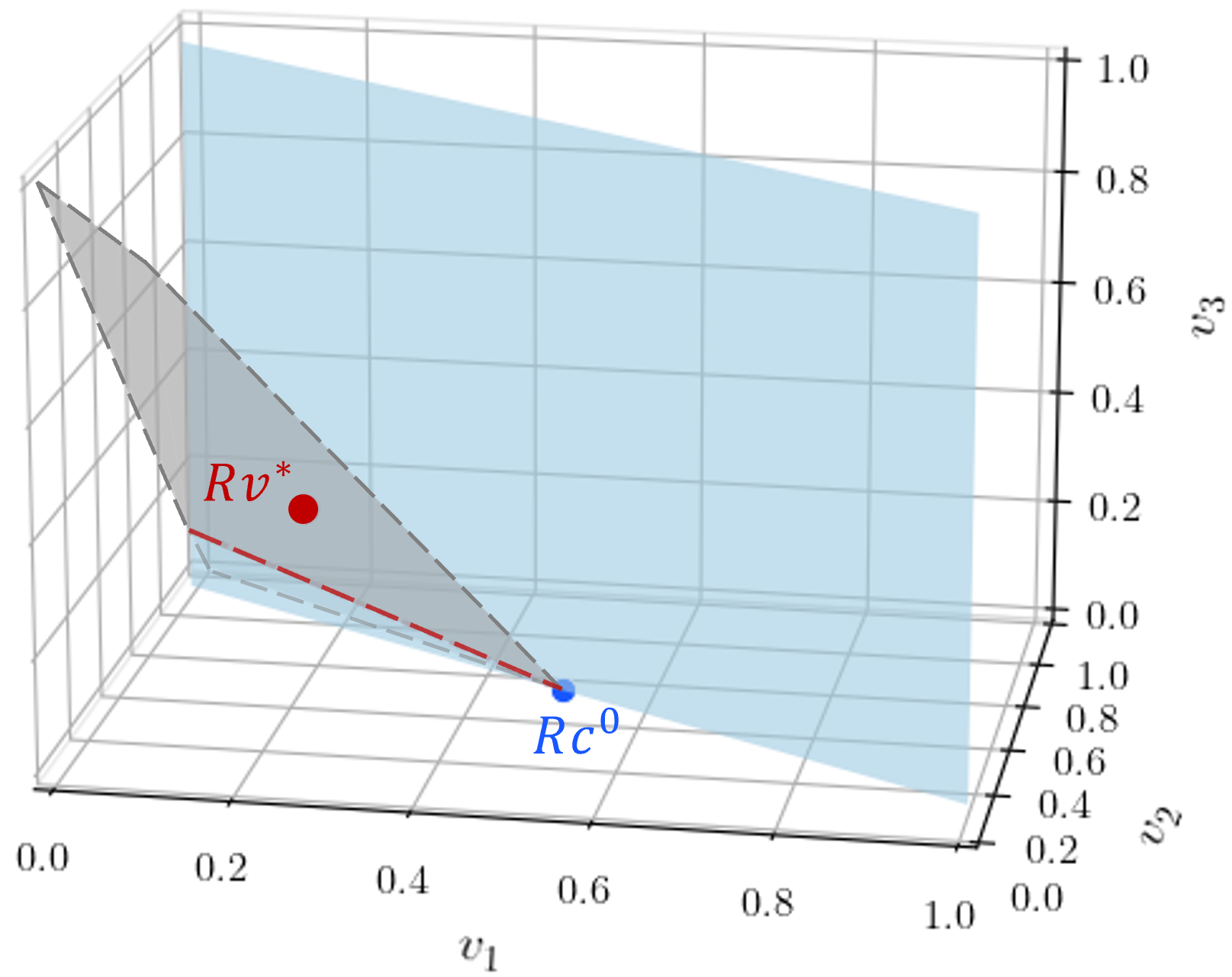}
        \caption{\label{fig:major_cut1.1_3d}}
    \end{subfigure}
    \caption{{\footnotesize Illustration of MEBs and cutting planes in major iteration one.    
    (a) MEBs after each coordinate-wise cut
    in $\R^2$. 
    (b) Cutting plane
    in $\R^3$ after one coordinate-wise cut.
    (c)  Cutting plane
    in $\R^3$ after second coordinate-wise cut.
    \label{fig:major1}}}
    \end{figure}
    
    \begin{figure}[!htbp]   
    \begin{subfigure}[t]{0.32\textwidth}
        \centering
        \includegraphics[width=\textwidth]{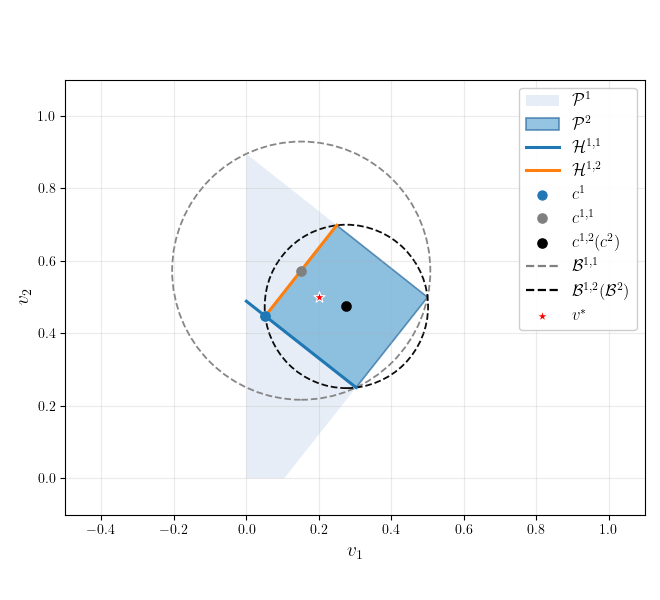}
        \caption{\label{fig:major2_2d}}
    \end{subfigure}\hfill  
    \begin{subfigure}[t]{0.32\textwidth}
        \centering
        \includegraphics[width=\textwidth]{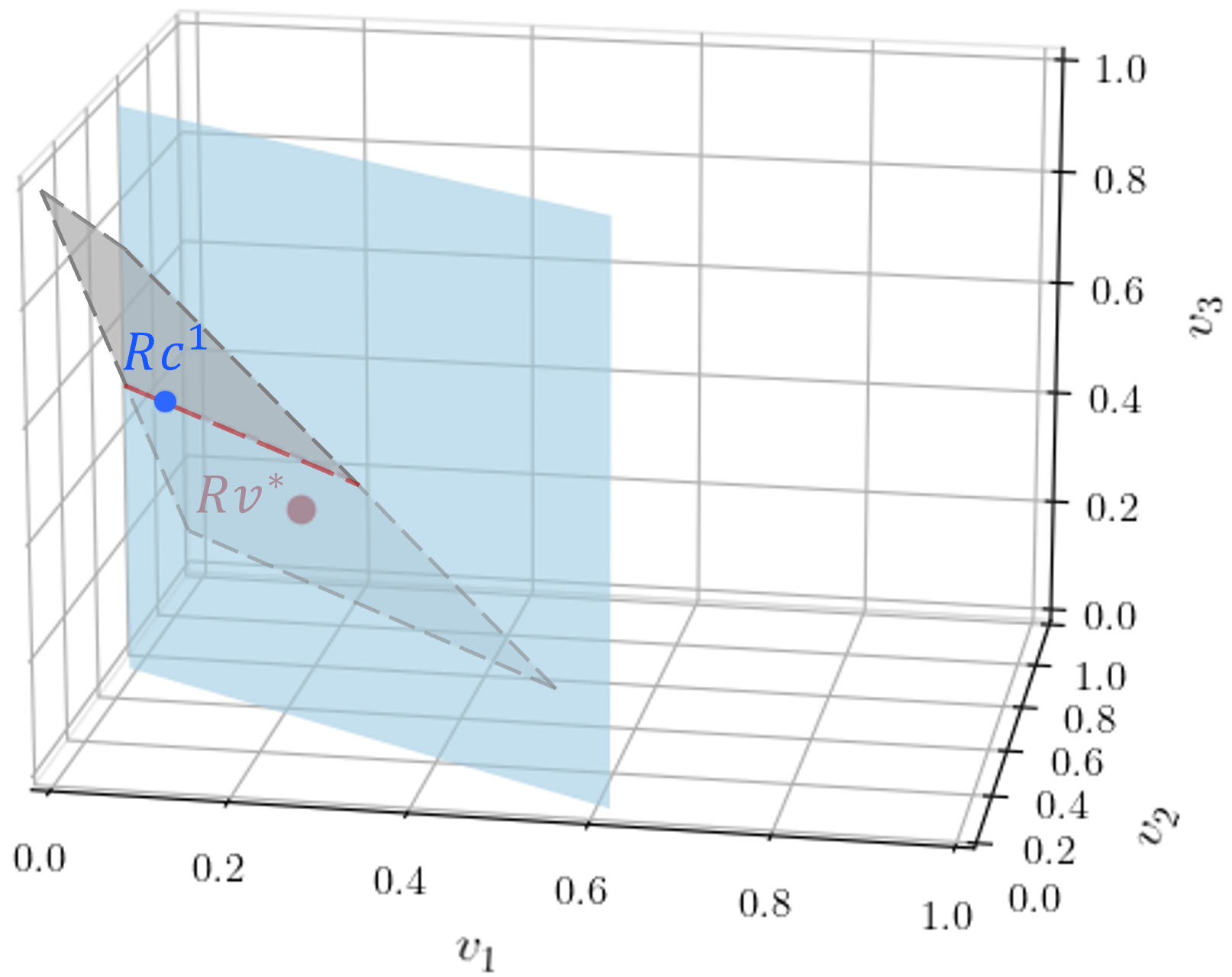}
        \caption{\label{fig:major_cut2_3d}}
    \end{subfigure}\hfill  
    \begin{subfigure}[t]{0.32\textwidth}
        \centering
        \includegraphics[width=\textwidth]{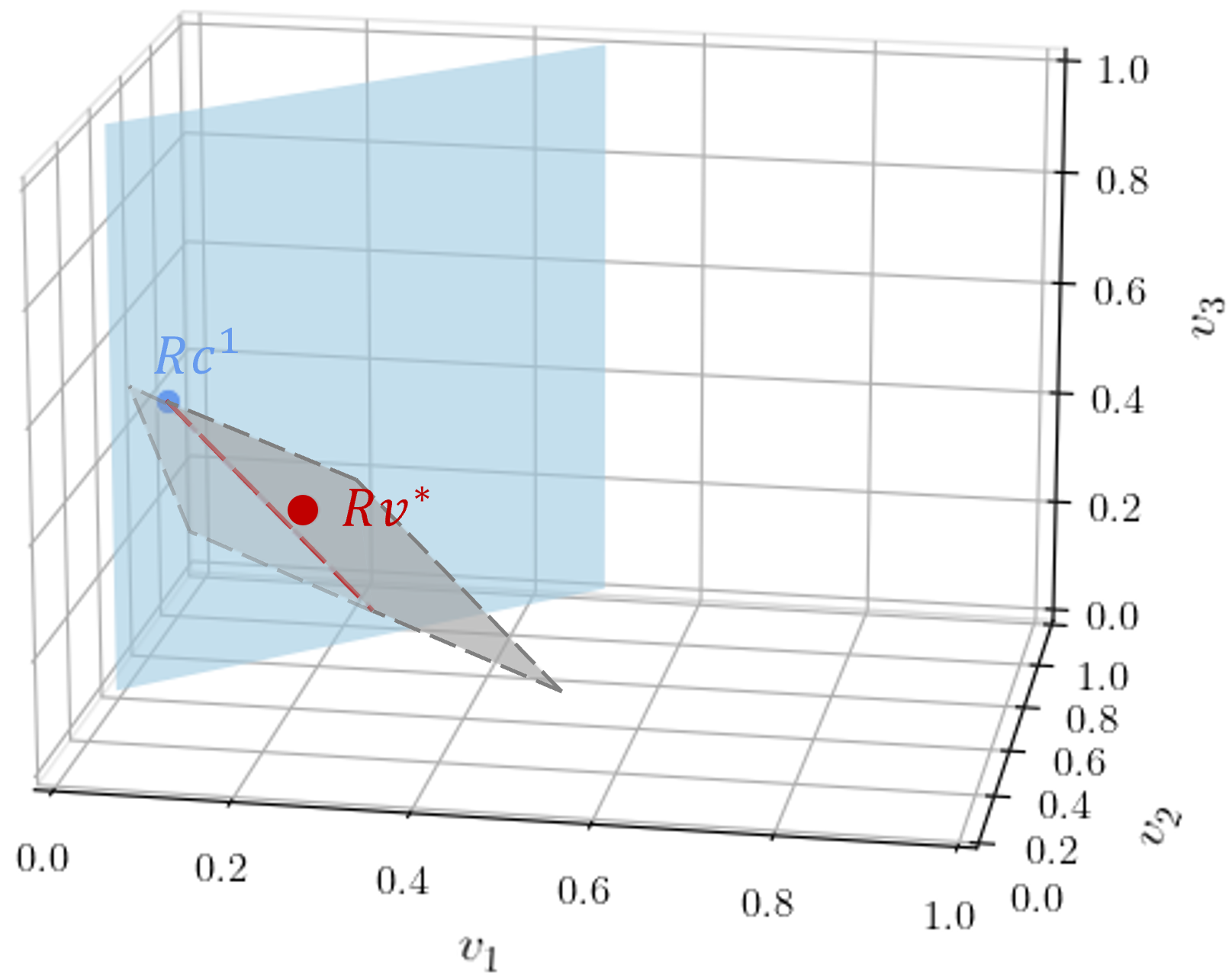}
    \caption{\label{fig:major_cut2.1_3d}}
    \end{subfigure}\hfill  
   \caption{{\footnotesize Illustration of MEBs and cutting planes in major iteration two.    
    (a) MEBs after each coordinate cut 
    in $\R^2$. 
    (b) Cutting plane
    in $\R^3$ after one coordinate-wise cut.
    (c)  Cutting plane
    in $\R^3$ after second coordinate-wise cut.
    \label{fig:major2}}}
\end{figure}

\begin{figure}[!t]
    \centering
    \begin{subfigure}[t]{0.48\textwidth}
        \centering
        \includegraphics[width=0.8\textwidth]{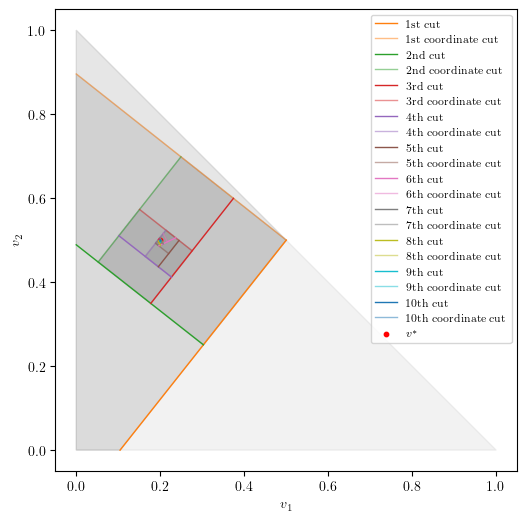}
        \caption{CPM \label{fig:CPM_all}}
    \end{subfigure}  
    \begin{subfigure}[t]{0.48\textwidth}
        \centering
        \includegraphics[width=0.8\textwidth]{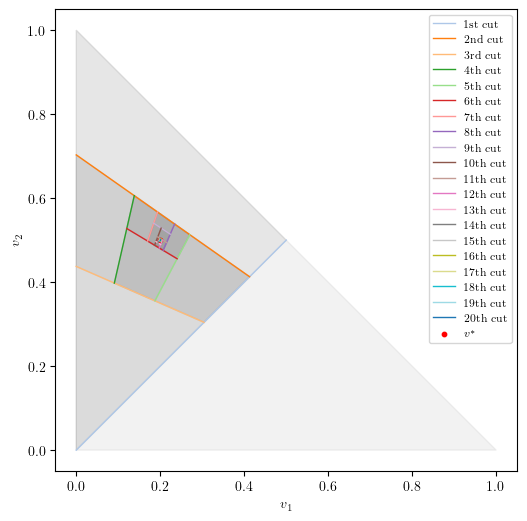}
        \caption{POLY\label{fig:POLY_all}}
    \end{subfigure}  
    \caption{The cutting process\label{fig:comparison}} 
\end{figure}
\section{Univariate piecewise-linear utility function}
\label{sec:univariate}

In the domain of individual decision-making under risk, preferences are often modeled using nonlinear univariate utility or distortion functions. One of the foundational frameworks in this area is the von Neumann–Morgenstern expected utility theory (EUT) \citep{von1947theory}, which postulates that preferences between lotteries can be captured by the expected value of a utility function defined over outcomes. Building on this, a range of generalized models have been developed. For instance, Yaari’s dual theory \citep{yaari1987dual}, rank-dependent expected utility (RDEU) \citep{quiggin1982theory,wakker1994separating}, and cumulative prospect theory (CPT) \citep{tversky1992advances} allow for nonlinear transformations of both outcomes and probabilities. These models often incorporate univariate nonlinear utility and/or distortion functions to represent individual attitudes toward risk, loss, and probability weighting flexibly. In practice, the true utility function/distortion function which captures the DM's preferences is often unknown. Consequently, various methods to elicit preferences have been proposed to obtain an approximate utility/distortion.
Existing elicitation methods can broadly be categorized into parametric approaches, which assume specific functional forms for utility or distortion functions (e.g., \citet{tversky1992advances,camerer1994nonlinear,hey1994investigating,tversky1995weighing}), and non-parametric approaches, which do not impose a fixed structure and instead estimate values pointwise from data (e.g., \citet{wakker1996eliciting,abdellaoui2000parameter,abdellaoui2007loss}). A common technique used in both settings involves presenting the decision maker with a sequence of comparisons between certain outcomes and probabilistic lotteries. One well-known method of this kind is the Becker-DeGroot-Marschak (BDM) procedure \citep{becker1964measuring}, which enables the identification of utility values at selected outcome levels. The elicited discrete points are typically connected through interpolation, resulting in a single approximated utility function.


\citet{zhang2025modified} 
seem to be the first to extend
the polyhedral cut method 
to elicit an approximate univariate nonlinear utility function. They do so by approximating the true unknown 
utility function 
through a class of piecewise-linear utility functions parameterized by the vector of increments of 
the linear pieces first, 
and then propose a modified polyhedral cut method 
in the space of the vector of parameters.
While the method shows promising empirical convergence results, it does not provide a deterministic guarantee of contraction for the ambiguity set. We next show how CPM can be adapted to
the increment space of univariate piecewise-linear utilities and how the resulting coordinate-wise
cuts lead to guarantees of convergence.




\subsection{Piecewise-linear utility function}



Let $(\Omega,{\cal F},\mathbb P)$ be a probability space and
${\cal L}^0$ denote 
the set of all random variables mapping from $\Omega$ to $\R$.
The preference functional based on the expected utility theory is defined by $V_u:{\cal L}^0\to \R$,
$V_u(\xi):=
\mathbb{E}_{\mathbb{P}}[u(\xi(\omega))]
$,
where $u:\R \to \R$ is a utility function.
In this section, we follow the strand of research by considering the case that 
a DM's true utility function is parametric and 
we develop a CPM to elicit the DM's preferences and use the elicited information to reduce the ambiguity of the 
true utility function. We make the following assumption.

\begin{assumption} The decision maker's preferences can be described by von Neumann-Morgenstern's expected utility theory and there exists a 
nondecreasing
piecewise-linear utility function (PLU) which characterizes the DM's preferences. The utility function is 
defined over interval 
$[\underline{x}, \overline{x}]$ and normalized with 
$u(\underline{x})=0$, $u(\overline{x})=1$.   
\end{assumption}




We specify the PLU as follows.
Let  ${\cal X}:=\{x_1,\cdots,x_{N+1}\}$ with $\underline{x}=x_1<\cdots<x_{N+1}=\overline{x}$ being fixed breakpoints, $N\geqslant3$, and $u(x)$ be
a PLU 
\begin{equation}
\label{eq:PLA_utility_org}
u(x):=(v_1,\cdots,v_N)^{\top} {\bm g}(x),\ \sum_{i=1}^Nv_i=1,
\end{equation}
where $\vv=(v_1,\ldots,v_{N-1})^\top\in\R_+^{N-1}$ is the vector of increments 
with $v_i=u(x_{i+1})-u(x_i)$ for $i=1,\cdots,N-1$ and $v_N$ is uniquely determined by $\vv$, 
and $\g:\R\mapsto \R^N$ is a basis feature mapping defined by
    \begin{equation}
\label{eq:g(x)-incre}
 {\bm g}(x_1)=0\quad\text{and}\quad{\bm g}(x)=\Big(\underbrace{1,\cdots,1}_{i-1},\frac{x-x_i}{x_{i+1}-x_i},\underbrace{0,\cdots,0}_{N-i}\Big)^\top
   \in \R^{N},\  \forall x\in (x_{i},x_{i+1}].
\end{equation}
Let $\RR:\R^{N-1}\rightarrow\R^N$ be
the recovery mapping defined in \eqref{map:recover}.
Then 
\begin{equation}
\label{eq:PLA_utility}
u(x):=(\RR\vv)^{\top} {\bm g}(x).
\end{equation}
Under formulation \eqref{eq:PLA_utility}, 
we can see that a PLU is uniquely characterized by 
the vector of increments $\vv$,
thereby transforming the problem of eliciting a PLU into the task of eliciting its corresponding increment vector. Consequently, the polyhedral ambiguity set is constructed over the space of increment vectors.

\subsection{Structure of lotteries
for pairwise comparisons}
We consider that each query presented to a DM consists of two lotteries. 
Let $\mathcal{S}_l$ denote the set of all 
feasible lotteries for pairwise comparison.
For  each lottery $A\in\mathcal{S}_l$, we use 
$\rr^A\in \R^{n+1}$ to denote the vector of $n+1$ outcomes and
$\p^A$ the vector of corresponding probabilities.
{This differs greatly from the lotteries in the existing preference elicitation approaches such as RUS, RRUS in \citep{armbruster2015decision} and the modified polyhedral method in \citep{zhang2025modified} where each lottery typically has two outcomes  or one outcome.}
In the forthcoming discussions, we will write $A$ and $(\rr^A,\p^A)$ for a lottery interchangeably depending on the context. With this convention,
we define set
\bgeqn
\label{eq:D_S_l}
\mathcal{D}_{\mathcal{S}_l}:=\big\{(\rr^A,\p^A)\mid A\in\mathcal{S}_l\big\},
\edeqn 
which is an alternative representation of $\mathcal{S}_l$.
 Without loss of generality, we assume that the outcomes of each lottery are ordered in strictly increasing magnitude, that is, $r_{i+1}>r_i$ for all $i=0,\ldots,n-1$.





\begin{definition}[Lotteries for pairwise comparison 
under EUT]\label{def:lottery}
A query 
to be designed for eliciting a DM's preferences via pairwise 
comparison comprises 
two lotteries $A$ and $B$ structured as follows: 
\begin{equation}
\label{eq:lotteries-A-B}
A =\{\rr^A,\p^A\}:= \{(r_i^A,p_i^A)\}_{i=0}^n,\quad \text{and}\quad 
B =\{\rr^B,\p^B\}:= \{(r_i^B,p_i^B)\}_{i=0}^n.
\end{equation}
Let  \begin{equation}\label{eq:G}
        \G(\bm{r},\bm{p}):=\sum_{i=0}^{n} p_i \g(r_i)= \left(\sum_{i=0}^{n} p_ig_1(r_i),\ldots,\sum_{i=0}^{n} p_ig_N(r_i)\right)^\top.
\end{equation}
$A$ is preferred to $B$, written $A \succeq B$, if and only if
\bgeqn 
\mathbb{E}[u({A})]=({\bm R}{\vv^*})^{\top}\G(\rr^A,\p^A)\geqslant ({\bm R}{\vv^*})^{\top}\G(\rr^B,\p^B) =\mathbb{E}[u({B})],
\edeqn 
where 
${\bm R}$ is a recovery mapping defined as in \eqref{map:recover} and
$\vv^*$ denotes the vector of increments
of the true  unknown piecewise-linear utility function.
\end{definition}

The definition implicitly assumes that there is no error in the process of preference elicitation.
To ease the exposition, we write $\G^A$ and $\G^B$ for $\G(\rr^A,\p^A)$ and $\G(\rr^B,\p^B)$ respectively. 
\begin{assumption}
\label{ass:initial_poly_PLU}
 There exists an initial polyhedral ambiguity set $\mathcal{P}^0=\{\vv\in\R^{N-1}_+\mid \bm{A}\vv\leqslant\bm{b},\ \bm{A}\in\R^{m\times (N-1)},\ \bm{b}\in \R^m\}$ constructed based on the prior information about the DM’s preferences, which contains the true value of increment vector $\vv^*$ of the DM's utility function. 
There is no response error in the elicitation process.
   \end{assumption}

     Let $\mathcal{H}=\{\vv\in\R^{N-1}\mid \bal^\top(\vv-\vv_0)=0\}$ be a given hyperplane passing through $\vv_0\in \R^{N-1}$. 
    $\mathcal{H}$ acts as a  separating hyperplane
     of query $(A,B)$ 
   if it satisfies 
    \begin{subequations}
        \begin{align}
             &{\RR\vv}^\top(\G^A-\G^B)=0,\ \forall \vv\in {\cal H},\label{eq:sep_def1_PLU}\\
             &\left[(\RR\vv^{+})^\top(\G^A-\G^B)\right]\left[(\RR\vv^{-})^\top(\G^A-\G^B)\right]
            \leqslant0,\ \forall \vv^+\in {\cal H}_{+},\ \vv^-\in {\cal H}_{-},\label{eq:sep_def2_PLU}\\
&{\cal H}_{+}=\{\vv\mid\bal^{\top}(\vv-\vv_0)\geqslant0\},\ {\cal H}_{-}=\{\vv\mid\bal^{\top}(\vv-\vv_0)\leqslant0\}.
\end{align}
\label{eq:sep_def_PLU}
\end{subequations}
\subsection{Modified polyhedral method}
\label{sec:zhang-modified}
\citet{zhang2025modified} 
propose a modified polyhedral method to elicit a decision maker's (DM's) nonlinear univariate utility function
through 
pairwise comparison
questionnaires.
They do so by
(a) approximating the true unknown nonlinear utility with a continuous piecewise-linear (PL) function parameterized 
by the vector of increments,  (b) 
designing adaptive risky–certain lottery queries that generate separating hyperplanes in the space of the increment vectors of piecewise-linear utility functions. 
Let $\mathcal{X}=\{x_1<\cdots<x_{N+1}\}\subset [\underline x,\bar x]$  be a fixed set of breakpoints. Following the discussions in the preceding sections, we can write 
the  piecewise-linear approximate (PLA) utility function as follows:
\bgeqn 
u_N(x) =(\RR\bm v)^\top \g(x),\quad
\text{where} \; \RR\vv\in\R_+^{N}\; \text{and} \;\bm e^\top(\RR \vv)=1.
\edeqn 
At iteration $k$, the query 
to be designed 
is composed of 
a risky lottery with two uncertain outcomes 
and one with a 
sure outcome,
$A^k=\{r_1^k,1-p^k;\ r_3^k\},B^k=r_2^k$.
Let 
\bgeqn 
\G^A(r_1,r_3,p)=(1-p)\g(r_1)+p\g(r_3),\qquad \G^B(r_2)=\g(r_2).
\edeqn 
Then
\begin{equation}
\label{eq:zhang-update}
\mathcal{P}^k
=\mathcal{P}^{k-1}\cap\big\{\bm v:\ \mu_k(\RR\bm v)^\top\bigl(\G^A(r_1^k,r_3^k,p^k)-\G^B(r_2^k)\bigr)\leqslant 0\big\},
\end{equation}
where $\mu_k\in\{1,-1\}$ (with $\mu_k=1$ if $B^k$ is preferred).
Let $\tilde{\bm c}^k$ be the analytic center of $\mathcal{P}^{k-1}$ and $\bm v_1^k,\bm v_2^k$ be the two intersections  between the longest axis of the Sonnevend’s ellipsoid
and the boundary of the polyhedron $\mathcal{P}^{k-1}$. The risky–certain pair is chosen by solving two optimization problems with a “budget” $D\in(0,1)$ measured at the center, that is, solve
\begin{align}
(r_1^k,r_3^k,p^k)
&\in \mathop{\arg\max}_{r_1 \leqslant r_3,\; p \in [0,1]}
\left\{
(\bm R\vv_2^k)^\top \G^A(r_1,r_3,p)
:
\;(\bm R\tilde{\cc}^k)^\top \G^A(r_1,r_3,p)\leqslant D
\right\}, \\[2mm]
r_2^k 
&\in \mathop{\arg\max}_{r_2\in[\underline{x},\bar{x}]}
\left\{
(\bm R\vv_1^k)^\top \G^B(r_2)
:
\;(\bm R\tilde{\cc}^k)^\top \G^B(r_2)\leqslant D
\right\}. 
\end{align}
for $A^k$ and $B^k$ respectively.
The structure of the 
optimization problems 
is similar to that in \citet{toubia2004polyhedral}.
However, they are highly non-convex, which 
poses challenges for solving these optimization problems. 
Moreover, because of the gap between the piecewise-linear approximation and the true utility, a cut computed under the piecewise-linear approximation may exclude the true increment vector, called direction error. Zhang et al.~tackle the issue
by 
augmenting the set of breakpoints after each question: $$
\mathcal{X}^k:=\mathcal{X}^{k-1}\cup\{r_1^k,r_2^k,r_3^k\},
$$
rebuilding $g_{X^k}(\cdot)$, and lifting all past constraints into the new basis (updating $\G^A,\G^B$ accordingly) before applying the fresh cut. 
{
This 
approach effectively addresses the direction error problem.}
 The modified method  generalizes polyhedral elicitation to shape-free nonlinear utilities through an increment-based piecewise-linear representation,  uses analytic-center/Sonnevend's ellipsoid to select informative, near-balanced queries, and repairs approximation-induced cut errors by adaptively enriching breakpoints.

\subsection{CPM for eliciting a true piecewise-linear utility function}

For piecewise-linear utilities (PLUs), the process of utility elicitation is analogous to that of linear utilities, and the CPM algorithm follows the same overall procedure. The key difference lies in the form of the queries: the specific structure of the expected basis mapping $\G$ under PLUs introduces additional challenges in questionnaire generation, which will be discussed in detail in Theorem \ref{theo:solution}. We first establish the CPM for eliciting PLU.


\vspace{2mm}

\begin{breakablealgorithm}\label{alg:VCPM_PLU}
    \caption{CPM for eliciting a piecewise-linear utility} 
\begin{flushleft}
\textbf{Input:} The set of breakpoints $\XX$ with $\vert\XX\vert=N+1$,
initial ambiguity set $\mathcal{P}^0\in\R^{N-1}$, iteration index $k = 0$, and precision $\epsilon$.
\end{flushleft}
\begin{itemize}[leftmargin=0pt,label={},itemsep=0.6ex]
\item\textbf{Step 1.}  Calculate the MEB center $\bm{c}^{k}$ and radius $r^k$.
\item\textbf{If} $r^k>\epsilon$ \textbf{do}
\item\textbf{Step 2.} 
Identify $\bal^{k,1},\bal^{k,2},\ldots,\bal^{k,N-1}\in\R^{N-1}$ which form a normalized orthogonal basis in $\R^{N-1}$.

\item\textbf{Step 3.} 
 \textbf{For} $i=1,\ldots,N-1$ \textbf{do}
    \begin{itemize}
     \item[-] Construct a pairwise comparison query $(A^{k,i},B^{k,i})$ with lottery pairs $\{\rr^A,\p^A\} ,
     \ \{\rr^B,\p^B\}$ defined as in \eqref{eq:lotteries-A-B}
     so that
     \begin{equation}\label{def:plane}
         {\cal H}^{k,i}=\{\vv\mid(\bal^{k,i})^\top(\vv-\bm{c}^k)=0\}
     \end{equation}
     acts as a separating hyperplane. 
    
     
   \item[-] Ask the DM to choose between $A^{k,i}$ and $B^{k,i}$.
Let $\mu^{k,i}=-1$ if $A^{k,i}$ is chosen,
i.e., $(\RR\vv^*)^\top \G^{A^{k,i}}\geqslant (\RR\vv^*)^\top \G^{B^{k,i}}$
and $\mu^{k,i}=1$ otherwise
(see Definition~\ref{def:lottery}). 
         \item[-] Update the polyhedron $\mathcal{P}^{k,i-1}$:
             \bgeqn 
             \mathcal{P}^{k,i}:=\mathcal{P}^{k,i-1}\cap \{\vv\mid
              \mu^{k,i}(\RR\vv)^\top(\G^{A^{k,i}}-\G^{B^{k,i}})\leqslant0\}.
              \edeqn 
    \end{itemize}
    \item\textbf{Step 4.} Set $\PP^{k+1}:=\PP^{k,N-1}$ and $k:=k+1$. Go back to Step 1.
   
\end{itemize}
\begin{flushleft}
\textbf{Output:} $\PP^{k}$.
\end{flushleft}
\end{breakablealgorithm}
\vspace{2mm}
The structure
of Algorithm~\ref{alg:VCPM_PLU} 
is similar to that of Algorithm~\ref{alg:VCPM_LU} despite 
the cut hyperplanes are constructed in different spaces.
The latter difference leads to different ways to construct pairwise comparison 
queries  for a given separating hyperplane.
In multivariate linear utility function case,
$\Delta\x=\x^{A}-\x^{B}$ (see Theorem~\ref{thm:find_query}). 
In the univariate 
piecewise-linear utility case, the separating hyperplane 
is determined by the difference of expected
basis vectors $\Delta\G:=\G^{A}-\G^{B}$ of two lotteries.
Algorithm~\ref{alg:GQ} provides a constructive procedure (see Proposition~\ref{prop:find_query_G} and Theorem~\ref{theo:solution}) to generate such lottery pairs
for a prescribed hyperplane. 



\begin{proposition}[Sufficient conditions for design of queries]
\label{prop:find_query_G}
Let ${\cal H}=\{\vv|\bal^\top(\vv-\bm{c})=0\}$ be a given hyperplane. There exists a 
query $(A, B)$ structured  as in Definition~\ref{def:lottery} 
with $\G^A=(G_1^A,\ldots,G^A_N),\ \G^B=(G_1^B,\ldots,G^B_N)$, where
\begin{subequations}
    \begin{align}
        G^A_N-G^B_N&=-\gamma\sum^{N-1}_{i=1}\alpha_i c_i,\\
        G^A_i-G^B_i&=(G^A_N-G^B_N)+\gamma\alpha_i=\gamma\left(\alpha_i-\sum^{N-1}_{i=1}\alpha_i c_i\right),
    \end{align}
    \label{eq:def_G}
\end{subequations}
and $\gamma>0$ is a relaxation factor, such that  
$\G^A$ and $\G^B$ 
satisfy \eqref{eq:sep_def_PLU}.

\end{proposition} 



Proposition~\ref{prop:find_query_G} shows how
$\G^A$ and $\G^B$ may be constructed such that 
the resulting separating hyperplane corresponds to a prespecified hyperplane. It remains to determine how lotteries $A=\{(r_i^A,p_i^A)\}_{i=0}^n$ and $B=\{(r_i^B,p_i^B)\}_{i=0}^n$ can be identified with the specified conditions \eqref{eq:def_G} because 
the system of equalities and inequalities have more variables than equations. 
It is important to note that not every type of queries can recover all possible hyperplanes, e.g., the two-outcome vs one-outcome queries used in \cite{zhang2025modified}.
The next theorem provides a sufficient condition: by properly selecting $N+1$ outcomes and assigning their probabilities appropriately, we can construct a pair of lotteries $\{r_i^A\}_{i=0}^N$ and $\{r_i^B\}_{i=0}^N$ that can recover any desired hyperplane satisfying \eqref{def:plane} exactly.



\begin{theorem}[Design of the query corresponding to a prespecified ${\mathcal H}$] \label{theo:solution} 
Given an outcome vector $\rr=(r_0,\ldots,r_N)^\top\in\R^{N+1}$ such that
 \bgeqn
        r_0=x_1\ \text{and}\; r_i\in (x_i,x_{i+1}],  i=1,\ldots,N,
\label{eq:so_ex_con}
\edeqn 
there exists two probability vectors $({\bm p}^A, {\bm p}^B) =
(\{p_i^A\}_{i=0}^N, \{p_i^B\}_{i=0}^N)$ such that the pair of lotteries $A=\{\rr,\p^A\}$,\  $B=\{\rr,\p^B\}$ with $n=N$
is a solution of \eqref{eq:def_G} for some $\gamma$.


\end{theorem}

Based on
Theorem \ref{theo:solution} and its proof (see Appendix~\ref{sec:proof_thm4.1}), 
we can 
present an algorithm to generate a query $(A,B)$.

\vspace{2mm}
\begin{breakablealgorithm}\label{alg:GQ}
\caption{ Query generation for pairwise comparisons in CPM}
\begin{flushleft}
\noindent\textbf{Input:}  number of outcomes $n=N$, the set of breakpoints  $\mathcal{X}=\{x_1,\ldots,x_{N+1}\}$, 
a hyperplane ${\cal H}=\{\vv\mid\bal^\top(\vv-\bm{c})=0\}$ 
obtained from Step~3 of Algorithm~\ref{alg:VCPM_PLU}.
\end{flushleft}


\begin{itemize}[leftmargin=42pt,label={},itemsep=0.6ex]
\item[\textbf{Step 1}.] Generate $\rr=(r_0,\ldots,r_N)^\top\in\R^{N+1}$ such that  
$r_0=x_1$ and $r_i=\frac{x_i+x_{i+1}}{2}$, for $i=1,\ldots,N$, set $\rr^A=\rr^B=\rr$. 

\item[\textbf{Step 2}.] Compute $\bbe\in\R^N$ via \eqref{eq:def_beta}. 

\item[\textbf{Step 3}.] 
Let  $\tilde r_i=(r_i-x_i)/(x_{i+1}-x_i)$ for $i=1,\cdots,N$.
Calculate $\{L_i\}_{i=1}^N$ 
via \eqref{eq:L=p^A-p^B} and $\{L_i^+\}_{i=1}^N,\ \{L_i^-\}_{i=1}^N$ via \eqref{eq:L+_L-}.

\item[\textbf{Step 4}.] 
Set probabilities $\p^A$ and $\p^B$ via \eqref{eq:p^A} and \eqref{eq:p^B} respectively. 

\end{itemize}
\begin{flushleft}
\noindent\textbf{Output: }$A=\{\rr,\p^A\}$,\  $B=\{\rr,\p^B\}$.
\end{flushleft}

\end{breakablealgorithm}

\vspace{2mm}

\begin{remark}

A query satisfying \eqref{eq:so_ex_con} is not unique, since the  $r_i$ can be randomly sampled from $(x_i,x_{i+1}]$ for $i=1,\ldots,N$. Theorem \ref{theo:solution} ensures that as long as $\rr$ satisfies \eqref{eq:so_ex_con}, we can always find a pair of lotteries for which the given hyperplane serves as their separating hyperplane. For simplicity, we specify $r_i=\frac{x_i+x_{i+1}}{2}$, for $i=1,\ldots,N$ in Step 1 of Algorithm~\ref{alg:GQ}. 
The design of the queries may be further simplified.  From the definitions of $\{L_i^+\}_{i=1}^N$ and $\{L_i^-\}_{i=1}^N$ in Algorithm~\ref{alg:GQ}, it can be observed that their elements are mutually exclusive, except in the case where $L_i=0$. However, when $L_i=0$, both $p^A_i$ and $p^B_i$ are 0, the corresponding outcome can be deleted from the lottery. By doing so, the elements of $\p^A$ and $\p^B$ are mutually exclusive. Therefore, $A$ and $B$ can be effectively reduced to 
            \bgeqn\label{eq:shorten_query}
            A=\{(x_1,\rr_{\mathcal{I}^+}),(p_0^A,\p^A_{\mathcal{I}^+})\},\ {\mathcal{I}^+}=\{i\mid L_i> 0\}
            \quad \text{and} 
            \quad B=\{(x_1,\rr_{\mathcal{I}^-}),(p_0^B,\p^B_{\mathcal{I}^-})\},\ {\mathcal{I}^-}=\{i\mid L_i< 0\}.            
            \edeqn 
Moreover, $\dim(\rr_{\mathcal{I}^+})+\dim(\rr_{\mathcal{I}^-})\leqslant\dim(\rr)=N+1$.   This discussion shows that the outcomes of two lotteries are exclusive, which reduces the difficulty of the DM's decision.

 
\end{remark}

Theorem~\ref{theo:solution} provides sufficient conditions for constructing a pairwise-comparison query whose $\G^A$ and $\G^B$ satisfy \eqref{eq:def_G}. The geometric insight of this construction is that, if the outcome vector $\rr^A$ and $\rr^B$ satisfy \eqref{eq:so_ex_con}, the corresponding set of $\G^A-\G^B$ is a convex subset of $\mathbb{R}^N$ containing the origin. Consequently, for any prescribed normal direction $\bal$, we can find a solution of \eqref{eq:def_G}. Proposition~\ref{prop:G_convex} formalizes this convexity property.

\begin{proposition}\label{prop:G_convex}
Let $
\mathcal{D}_{\mathcal{S}_l}$  be defined as in \eqref{eq:D_S_l} and $(\rr,\p)\in
\mathcal{D}_{\mathcal{S}_l}$, let
$\G(\rr,\p)$ be defined
 as in \eqref{eq:G}, i.e., $\G(\rr,\p)=\sum_{i=0}^{N} p_i \g(r_i)$.
Define 
$
\hat{\mathcal{D}}=
\Big\{(\rr,\p)\in\mathcal{D}_{\mathcal{S}_l}\mid \rr\text{ satisfies \eqref{eq:so_ex_con}}\Big\}\subseteq\R^{N+1}\times[0,1]^{N+1}$, where
$[0,1]^{N+1}:= \underbrace{[0,1]\times\cdots\times [0,1]}_{N+1}$.
Then $\G(\hat{\mathcal{D}})$
is a convex set in $\R^N$. 
\end{proposition}

To facilitate understanding, 
we give 
a geometric interpretation of $\g$, $\G$ and $\G^A-\G^B$ in the case  
where $\g\in \R^3$, see Figure \ref{fig-8-illus-g}.
The thick blue solid line in 
Figure \ref{fig-8-illus-g}(a) represents the range of values for $\g(\cdot)$, while the blue tetrahedral cones in 
Figure \ref{fig-8-illus-g}(b) represent the ranges of values for $\G^A$ and $-\G^B$, which both are convex sets. The octahedron in 
Figure \ref{fig-8-illus-g}(c) illustrates the range of values for $\G^A-\G^B$.

\begin{figure}[htbp]
	\begin{subfigure}[t]{0.32\textwidth}
		\centering
	\includegraphics[width=0.8\textwidth]{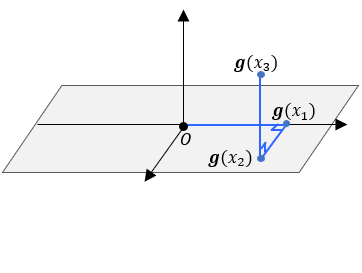}
		  \caption{}\label{fig-g}
	\end{subfigure}
    \begin{subfigure}[t]{0.32\textwidth}
		\centering
\includegraphics[width=0.8\textwidth]{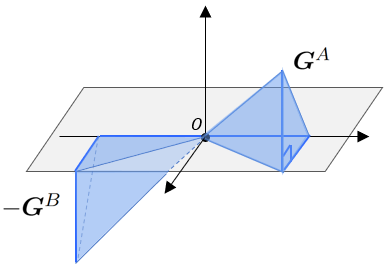}
		  \caption{\label{fig-G_AB}}
	\end{subfigure}
	\begin{subfigure}[t]{0.32\textwidth}
        \centering
\includegraphics[width=0.8\textwidth]{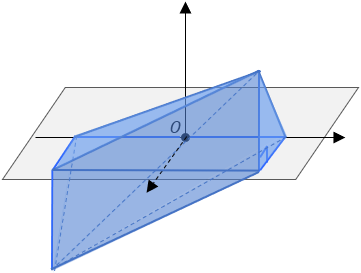}
		  \caption{\label{fig-G_A-G_B}}
	\end{subfigure}
    \caption{(a) The range of $\g(\cdot)$. (b) The range of $\G^A$ and $-\G^B$. (c) The range of $\G^A-\G^B$.}\label{fig-8-illus-g}
\end{figure}





Next, we present the convergence of Algorithm~\ref{alg:VCPM_PLU}. Unlike the multivariate linear utility function, 
we need to consider the convergence  of the utility functions estimated 
from the polyhedron elicited from the increment vectors.

\begin{theorem}[Convergence of Algorithm~\ref{alg:VCPM_PLU} 
in terms of utility functions]
\label{thm:conv_univariate_metric}
Let $\mathcal{X}=\{x_1<\cdots<x_{N+1}\}\subseteq[\underline{x},\overline{x}]$ be a  fixed set of breakpoints with $x_1=\underline{x}$ and $x_{N+1}=\overline{x}$.
Assume that the DM's true 
utility function is normalized piecewise-linear 
$
u^*(x)=(\RR \vv^*)^\top \g(x),
$
with the true increment vector $\vv^*\in\R_+^{N-1}$ contained in the initial bounded polyhedron $\mathcal P^0$.
Let 
$\mathcal P^{k,i}$ be the resulting ambiguity polyhedron
in the increment space after $k(N-1)+i$ queries in Algorithm~\ref{alg:VCPM_PLU}, where $k\geqslant 0$ and $i\in\{1,\ldots,N-1\}$.
Let
\begin{equation}\label{utility_amgSet}
    \mathcal U_{\mathcal{X}}^{k,i}:=\{u\mid  u(x)=(\RR\vv)^\top\g(x),\;\vv\in\mathcal P^{k,i}\}
\end{equation}
be the induced ambiguity set of PLU with breakpoints $\mathcal{X}$. Let $\mathcal B^0$ be the MEB
of $\mathcal P^0$ with radius $r(\mathcal B^0)$.
Then for each $k\geqslant 0$,  $i\in\{1,\ldots,N-1\}$, and $u\in\mathcal U_{\mathcal{X}}^{k,i}$,
\begin{equation}\label{eq:conv_metric_bound}
\dd_K(u^*,u)\ 
\leqslant
\ \dd_I(u^*,u)\ 
\leqslant\ 
\sup_{u',u''\in \mathcal{U}^{k,i}_{\mathcal{X}}} \dd_I(u',u'')
= \sup_{u',u''\in \mathcal{U}^{k,i}_{\mathcal{X}}}\delta\|u'-u''\|_\infty 
\leqslant\ 
2\delta\sqrt{N-1}\, r(\mathcal B^0)\Big(\sqrt{\tfrac{N-2}{N-1}}\Big)^{\,k},
\end{equation}
where $\dd_K$ and $\dd_I$ are
Kantorovich distance and scaled Kolmogorov distance, see \eqref{eq:uv-Kant-dist} and \eqref{eq:uv-Kolm-dist}, and
$\delta:=\overline{x}-\underline{x}$.
In particular, $$\sup\limits_{u\in\mathcal U_{\mathcal{X}}^{k,i}}\dd_K(u^*,u)\to 0,\quad \sup\limits_{u\in\mathcal U_{\mathcal{X}}^{k,i}}\dd_I(u^*,u)\to 0,\ \text{as}\ k\to\infty.
$$

\end{theorem}




\subsection{An illustrative example}
\label{sec-4-5-example}

In this subsection, we consider a simple academic/hand-computable example to examine the tightness of the theoretical bounds, in which we compare the real error between the true utility $u^*$ and the current ambiguity set $\mathcal{U}_{\mathcal{X}}^{k,i}$ of utility defined in \eqref{utility_amgSet} as well as the theoretical bounds derived in Theorem~\ref{thm:conv_univariate_metric} between them.


Consider a true piecewise-linear utility function 
\begin{eqnarray*}   
u^*(x)=\begin{cases}
0.1x, & \text{for} \; 0\leqslant x\leqslant 1\\
0.2x-0.1,  & \text{for} \;1<x\leqslant 2\\
0.3x-0.3,& \text{for} \;2< x\leqslant 3\\
0.4x-0.6,&\text{for} \;3<x\leqslant 4
\end{cases}
\quad=(0.1,0.2,0.3,0.4)^\top\ \g(x)
\end{eqnarray*}
with true increment vector $\RR\vv^*=(0.1,0.2,0.3,0.4)^\top\in\R^4$ and breakpoints $\mathcal{X}=\{0,1,2,3,4\}$. The initial polyhedron is given by 
$\mathcal{P}^0=\Big\{\vv\in\R^3\mid 0\leqslant v_i\leqslant1,i=1,2,3,\ \sum_{i=1}^3v_i\leqslant1\Big\}.$
\noindent At iteration $k=0$, the MEB center of $\mathcal{P}^0$ is $\cc^0=(1/3,1/3,1/3)^\top,\ \RR\cc^0=(1/3,1/3,1/3,0)^\top$ and the radius of the MEB is $r(\mathcal{B}^0)=\sqrt{6}/3$. 
We specify 
the basis set as\footnote{
For simplicity, in the numerical experiments, 
we construct the orthonormal basis by first taking the canonical coordinate basis in $\mathbb{R}^{N-1}$ and then applying an orthogonal rotation so that the first basis vector aligns with the $\alpha^{k,1}$ determined by the longest axis of the MOE of $\mathcal{P}^k$. 
The remaining basis vectors are obtained through the same rotation and therefore remain mutually orthogonal. 
}
$$\mathcal{A}^0=\Big\{(0.711,  -0.703, -0.008)^\top,(0.703,  0.711, 0.008)^\top,( 0,0.011,-1.000)^\top\Big\}.$$
We implement the first cut with hyperplane $\mathcal{H}^{0,1}=\{\vv\mid(0.711,  -0.703, -0.008)\vv=0 \}$ 
and corresponding 
query 
$A^{0,1}=\{\rr,\p\},\ B^{0,1}=\{\rr,\q\}$ where 
\begin{align*}
    \rr=(0, 0.5, 1.5, 2.5, 3.5), \quad
\p=(0.249, 0.751,0, 0, 0),\quad
\q=(0.751, 0, 0.246, 0.003,0).
\end{align*}
According to \eqref{eq:shorten_query}, the query can be simplified by 
\begin{align*}
    A^{0,1}=\{\rr^A,\p^A\}=\{(0,0.5),\ (0.249,0.751)\},\quad
    B^{0,1}=\{\rr^B,\p^B\}=\{(0,1.5,2.5),(0.751, 0.246, 0.003)\}.
\end{align*}
Likewise, for cut
$\mathcal{H}^{0,2}=\{\vv\mid(0.703,  0.711, 0.008)\vv=0 \}$, we can design 
a pairwise query
\begin{align*}
    A^{0,2}=\{(0, 1.5, 2.5),\ (0.500, 0.157, 0.343)\},\quad
    B^{0,2}=\{(0, 0.5, 3.5),(0.500, 0.163 ,0.337)\}.
\end{align*}
For cut $\mathcal{H}^{0,3}=\{\vv\mid( 0,0.011,-1.000)\vv=0 \}$, a feasible pairwise query is
\begin{align*}
    A^{0,3}=\{(0, 1.5,3.5),\ (0.580,  0.368,0.052)\},\quad
    B^{0,3}=\{(0, 0.5,  2.5),(0.421, 0.370,0.209)\}.
\end{align*}
Figure~\ref{fig:utility_range_and_max_dist}(a) 
depicts changes in the ranges of the utility function values $\left(\left[\inf_{u\in{\mathcal U_{\mathcal{X}}^{k,i}}}u(x),\sup_{u\in{\mathcal U_{\mathcal{X}}^{k,i}}}u(x)\right]\right)$ over its domain $[0,4]$
as the number of queries (N.q)
increases.
We can see that as the range of the utility functions shrinks to the true utility value pointwise, so does the utility function induced by
the MEB center. 
Figure~\ref{fig:utility_range_and_max_dist}(b) 
depicts the change of 
the diameter of $\mathcal U_{\mathcal{X}}^{k,i}$ under the Kantorovich distance and scaled Kolmogorov distance, and
the worst-case utility estimation errors induced by $\mathcal U_{\mathcal{X}}^{k,i}$
under the Kantorovich and scaled Kolmogorov distances (see Definition~\ref{def:kan_kolmo})
as well as the bound of the diameter $2\delta\sqrt{N-1}\, r(\mathcal B^0)\Big(\sqrt{\tfrac{N-2}{N-1}}\Big)^{\,k}$ in \eqref{eq:conv_metric_bound}.
We can see from the figure that
\begin{align*}
\sup\limits_{u\in\mathcal U_{\mathcal{X}}^{k,i}}\dd_K(u^*,u)&\leqslant \sup\limits_{u,u'\in\mathcal U_{\mathcal{X}}^{k,i}}\dd_K(u,u')\leqslant \sup\limits_{u,u'\in\mathcal U_{\mathcal{X}}^{k,i}}\dd_I(u,u'),\\
\sup\limits_{u\in\mathcal U_{\mathcal{X}}^{k,i}}\dd_K(u^*,u)&\leqslant \sup\limits_{u\in\mathcal U_{\mathcal{X}}^{k,i}}\dd_I(u^*,u)\leqslant 2\delta \sqrt{N-1}\, r(\mathcal B^0)\Big(\sqrt{\tfrac{N-2}{N-1}}\Big)^{\,k},  
\end{align*}
which confirms our theoretical results in Theorem~\ref{thm:conv_univariate_metric}.
The bound in Theorem~\ref{thm:conv_univariate_metric} is loose and corresponds to a worst-case scenario; in practice, CPM may converge much faster.

\begin{figure}[htbp]
\centering
\begin{subfigure}{0.49\linewidth}
    \centering
\includegraphics[width=0.8\linewidth]{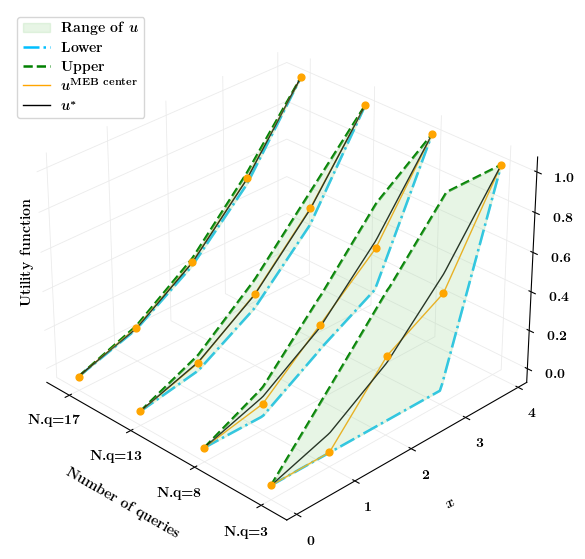}
    \caption{}
    \label{fig:single_example}
\end{subfigure}\hfill
\begin{subfigure}{0.49\linewidth}
    \centering
 \includegraphics[width=0.9\linewidth]{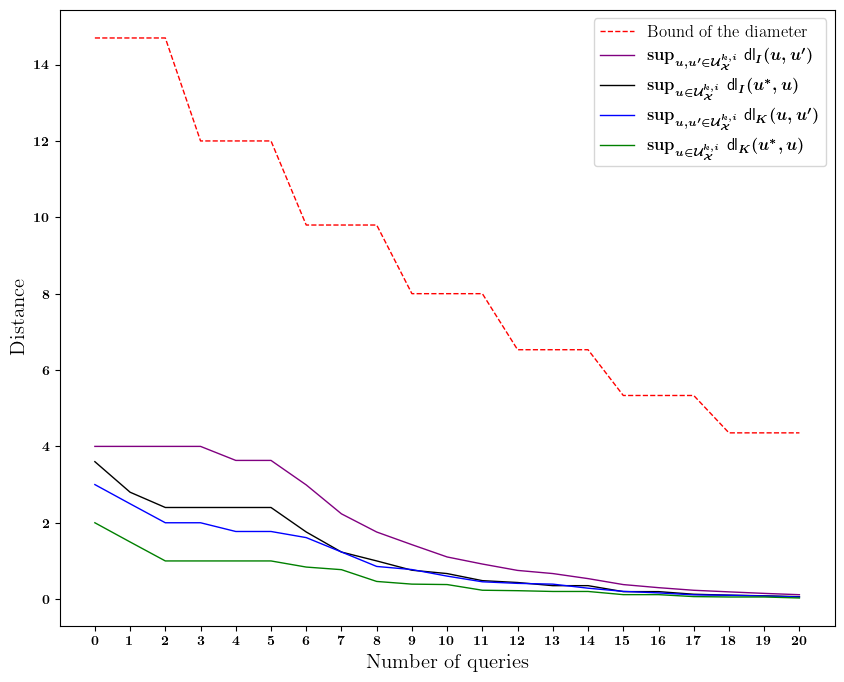}
    \caption{ }
    \label{fig:max_dist}
\end{subfigure}
\caption{\footnotesize(a) Changes of the range of point-wise utility functions $\mathcal U_{\mathcal{X}}^{k,i}$ as the number of queries increases.
(b) 
Variation of the diameter and its upper bound, as well as 
the worst-case error under the Kantorovich metric and scaled Kolmogorov metric.
\label{fig:utility_range_and_max_dist}}
\end{figure}



\subsection{
Comparison between CPM and POLY
}

{\color{black}
To see the difference in performance between 
CPM and POLY, we 
apply them to 
the following 
piecewise-linear utility function 
\begin{eqnarray*}   
u^*(x)=\begin{cases}
5.3846x, & \text{for} \; 0\leqslant x\leqslant 0.0377\\
4.7619x+0.0235,  & \text{for}\; 0.0377<x\leqslant 0.0398\\
0.8196x+0.1804,& \text{for}\; 0.0398< x\leqslant 1
\end{cases}
\quad=(0.203,0.01,0.787)^\top\ \g(x)
\end{eqnarray*}
with $\RR\vv^*=(0.203,0.01,0.787)^\top\in\R^3$ 
and 
breakpoints $\mathcal{X}=\{0,0.0377,0.0398,1\}$.
We begin with
initial polyhedron 
$
\mathcal{P}^0=\bigl\{\vv\in\R^2_+ \,\big|\,
v_1+0.01v_2\leqslant0.2225,\;
-v_1+0.01v_2\leqslant-0.2025,\;
v_1+v_2\leqslant1
\bigr\},
$
which forms a highly elongated and narrow polyhedron
in $\R^2$,
see the largest gray area in
Figure~\ref{fig:compare_piecewise}. 
Figures~\ref{fig:compare_piecewise}(a) and (b) 
depict how the initial 
polyhedron is cut
by
the two methods.
We can see in 
Figure~\ref{fig:compare_piecewise}(a) that
POLY initially generates cutting planes which are 
nearly parallel in the first four iterations, reflecting the fact that
the queries generated by the method (via solving optimization problems in \eqref{eq:original_PM})
may result in very small angles between two consecutive cuts.
By contrast, 
the CPM 
performs
coordinate-wise cuts
as designed.
The main difference is that POLY cuts
along direction 
$v_2$ which has the largest initial range (ambiguity),
whereas CPM cuts the ranges of all increments in each iteration.

To see 
how the ambiguity set of utility functions evolves during the process, we 
plot the range of  the ambiguity set of utility functions at three  
specified points
$\hat{x}_1=0.0377$, 
$\hat{x}_2=0.0398$ 
 and $\hat{x}_3=0.09$
within the domains 
of the first, second and third linear pieces of
$u^*$ respectively, i.e., 
$\max\limits_{\vv\in\mathcal{P}^k} \RR(\vv)^\top g(\hat{x}_i)$) and 
$\min\limits_{\vv\in\mathcal{P}^k} \RR(\vv)^\top g(\hat{x}_i)$ for 
$i=1,2,3$, see Figure~\ref{fig:Extrme_bound}.
We can make the following observations.
\begin{itemize}

\item Figure~\ref{fig:Extrme_bound}(a) shows that the blue curve
is flat during the first 4 queries, because 
the value of $u(\hat{x}_1)$ depends only on $v_1$ and its
range 
(represented by
$[
\min\limits_{\vv\in\mathcal{P}^k} \RR(\vv)^\top g(\hat{x}_1)),
\max\limits_{\vv\in\mathcal{P}^k} \RR(\vv)^\top g(\hat{x}_1))
]$ 
is unchanged under the POLY. 
This is consistent with our observation in Figure~\ref{fig:compare_piecewise}.

\item Figure~\ref{fig:Extrme_bound}(b) shows that both the blue curve and 
the green curve
drop rapidly over 
the first 5 queries, because 
the value of $u(\hat{x}_2)$ is determined by $v_1+v_2$ and the uncertainty of $v_2$ decreases rapidly,
leading the range
(represented by
$[
\min\limits_{\vv\in\mathcal{P}^k} \RR(\vv)^\top g(\hat{x}_2)),
\max\limits_{\vv\in\mathcal{P}^k} \RR(\vv)^\top g(\hat{x}_2))
]$ 
to decrease significantly, which is again consistent with 
our observation in Figure~\ref{fig:compare_piecewise}.
At this point, POLY reduces the range/ambiguity of the component (corresponding to the second linear piece of $u^*$) even more rapidly than CPM.
A similar observation can be made in 
Figure~\ref{fig:Extrme_bound}(c)  (corresponding to $1-v_1-v_2$, the increment of the third linear piece). 


    \item 
After 12 queries, no significant differences are observed, indicating that the two methods exhibit comparable local performance. 

    
\end{itemize}
Summarizing 
the observations above, 
we conclude that CPM demonstrates greater overall robustness and stability, as the ranges of utility values across all increments decrease in a steady and consistent manner.

}


\begin{figure}[htbp]
\vspace{-3mm}
    \begin{subfigure}[t]{\textwidth}
        \centering        \includegraphics[width=0.98\textwidth]{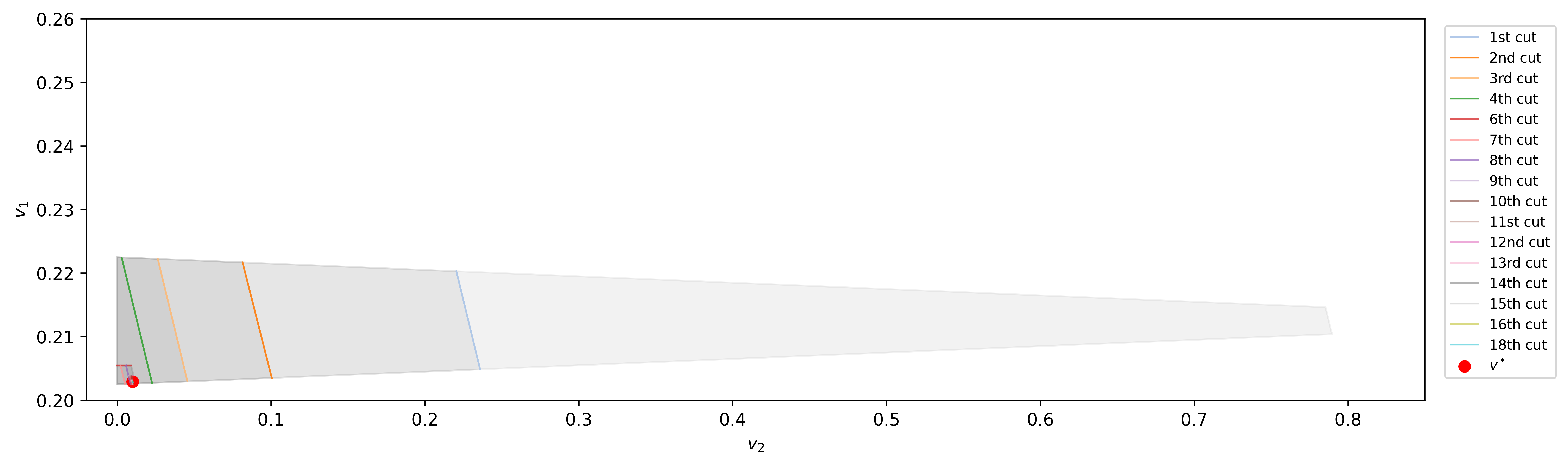}
        \caption{POLY\label{fig:Ex_Poly}}
    \end{subfigure}\\
    \begin{subfigure}[t]{\textwidth}
      \vspace{-1mm}
        \includegraphics[width=\textwidth]{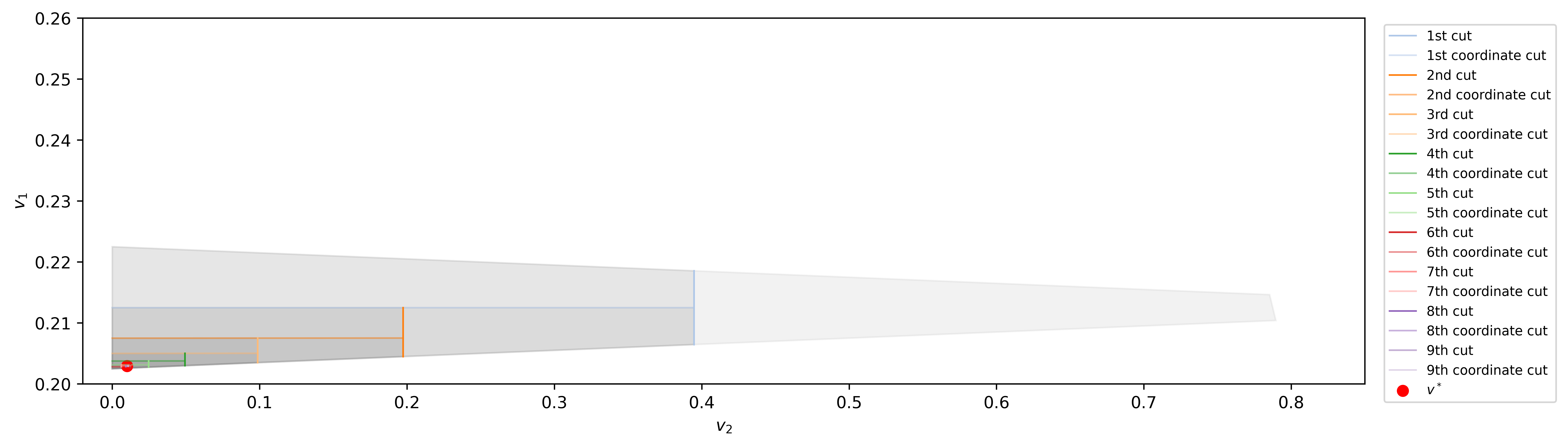}
        \caption{CPM\label{fig:Ex_CPM}}
    \end{subfigure}
    \vspace{-3mm}
    \caption{The cutting process\label{fig:compare_piecewise}}
    \end{figure}
\begin{figure}[!htbp]
\vspace{-3mm}
    \centering
    \begin{subfigure}[t]{0.32\textwidth}
        \centering
        \includegraphics[width=\textwidth]{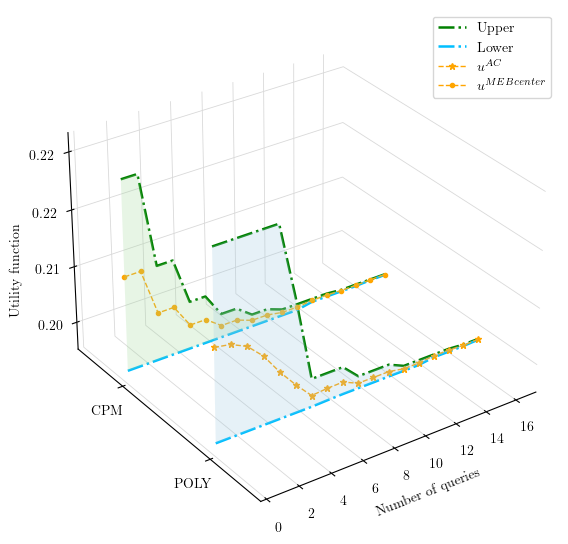}
    \end{subfigure}
    \begin{subfigure}[t]{0.32\textwidth}
        \centering
        \includegraphics[width=\textwidth]{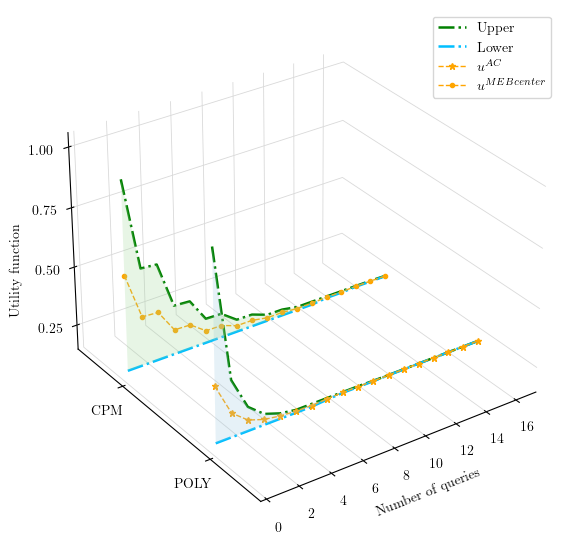}
    \end{subfigure}
    \begin{subfigure}[t]{0.32\textwidth}
        \centering
        \includegraphics[width=\textwidth]{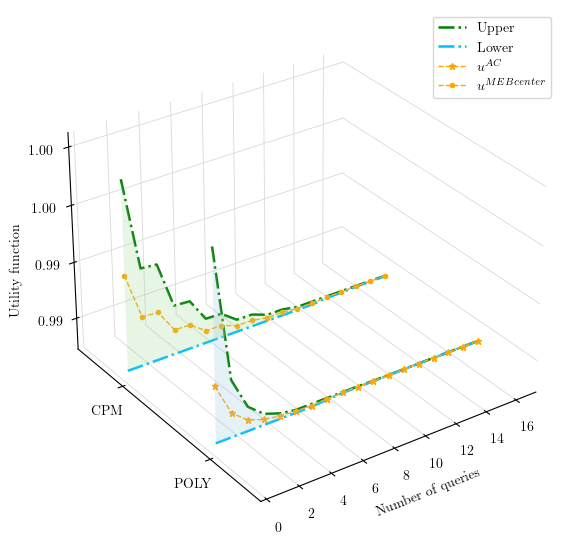}
    \end{subfigure}
    \begin{subfigure}[t]{0.32\textwidth}
        \centering
        \includegraphics[width=\textwidth]{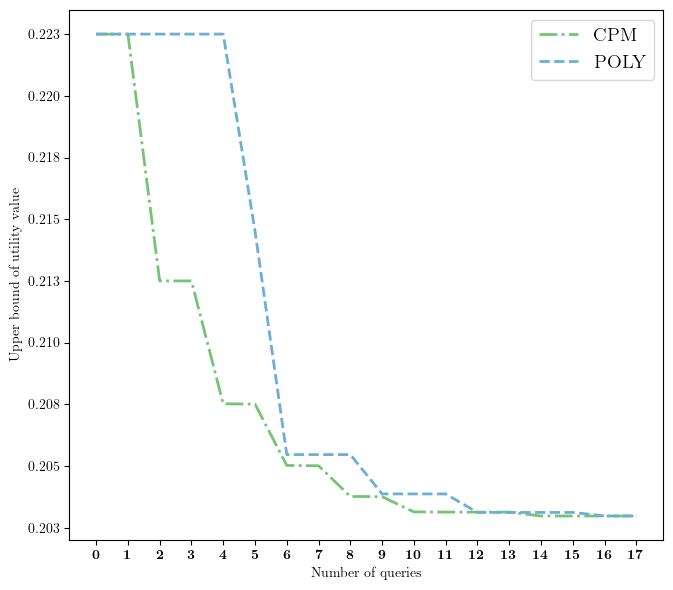}
        \caption{\label{fig:Ex_x1}}
    \end{subfigure}
    \begin{subfigure}[t]{0.32\textwidth}
        \centering
        \includegraphics[width=\textwidth]{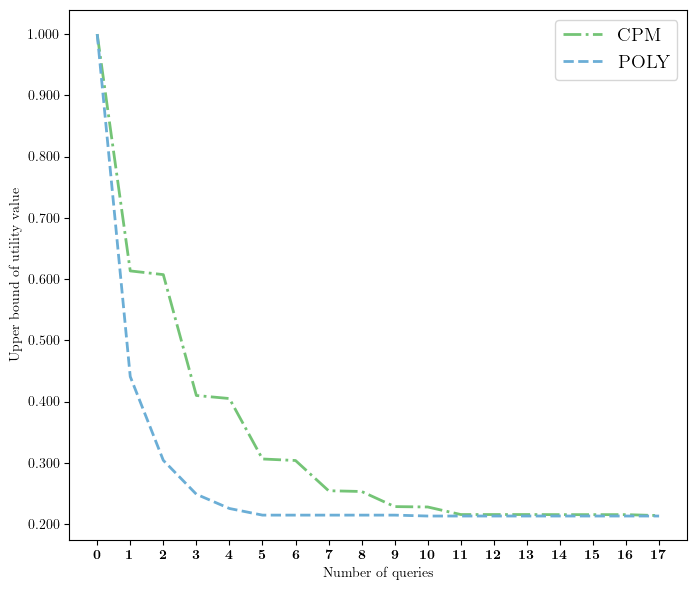}
        \caption{\label{fig:Ex_x2}}
    \end{subfigure}
    \begin{subfigure}[t]{0.32\textwidth}
        \centering
        \includegraphics[width=\textwidth]{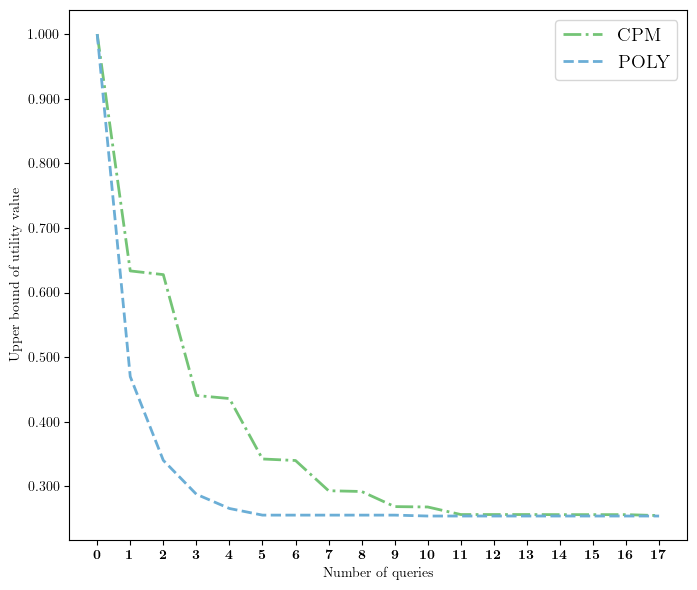}
        \caption{\label{fig:Ex_x3}}
    \end{subfigure}
\caption{\footnotesize The range of $u(\hat{x}_i)$ in polyhedra.
Upper row: changes of the ranges of
$\mathcal{U}_{\XX}^{k,i}$ at $u(\hat{x}_1),\;u(\hat{x}_2)$ and $u(\hat{x}_3)$ as the number of queries increases.
Lower row: 
  Evolution of the upper bound for $u(\hat{x}_1),\;u(\hat{x}_2)$ and $u(\hat{x}_3)$.
}
\label{fig:Extrme_bound}
\end{figure}

\section{Extension to general continuous utility functions}
\label{sec:continuous-utility}
In Section \ref{sec:univariate}, the true utility is assumed to be piecewise-linear 
with a fixed set of breakpoints. 
In this section we extend
the framework to the case where the DM’s true utility is a general continuous utility function on a
compact interval.

\begin{assumption}\label{ass:cont}
The decision maker's preference is represented by a utility function
$u^*:[\ulow,\ubar]\to [0,1]$ which is normalized, nondecreasing, and $L$-Lipschitz continuous, that is,
\bgeqn\label{def:Lip-utility}
u^*\in \UU_L:=\left\{u\mid u(\ulow)=0, \ u(\ubar)=1, u\;\text{is nondecreasing},\;\ |u(x)-u(y)|\leqslant L|x-y|, \ \forall x,y\in[\ulow,\ubar]\right\}.
\edeqn

\end{assumption}


    



\subsection{Piecewise-linear approximation}

To estimate a non-parametric shape-free utility function, it is natural to use a piecewise-linear approximation
(PLA), because the approximation on a finite breakpoint set yields a finite-dimensional representation. The
breakpoint set can then be refined stage by stage so that the approximation becomes increasingly accurate.
We begin with 
an initial set of breakpoints 
$
 \XX^0=\{x_1^0<\cdots<x_{N_0+1}^0\}\subset[\ulow,\ubar],\;N_0\geqslant3
$
and after $k-1$
iterations, the set is updated to
\bgeqn 
\label{eq:X_k-h_k}
 \XX^k:=\{x_1^k<\cdots <x_{\Nk+1}^k\}\subset[\ulow,\ubar],
\qquad \hk:=\max_{1\leqslant i\leqslant \Nk}(x_{i+1}^k-x_i^k).
\edeqn 
In the discussions below, 
we abuse terminology slightly by referring to ${\cal X}^k$ interchangeably as the grid and $h_k$ mesh size. This is sensible from the numerical approximation perspective.
The updated sets are nested, i.e.,
$\XX^k\subset \XX^{k+1}$, with
$x_1^k=\ulow$ and $x_{\Nk+1}^k=\ubar$ being fixed for all $k$.
Let
\bgeqn\label{def:Pi_k}
 \Pik(u):=\big(u(x_2^k)-u(x_1^k),\ldots,u(x_{\Nk}^k)-u(x_{\Nk-1}^k)\big)^\top\in \R^{\Nk-1}
\edeqn
denote the vector of 
increments over the successive breakpoints.
Let $\RR$ be the 
recovery map
defined in \eqref{map:recover}. 
The piecewise-linear interpolant
of $u$ on $\XX^k$ is defined by
\bgeqn \label{eq:T_k-PLA}
 (\Tk u)(x):=(\RR\Pik(u))^\top \g^k(x),
\edeqn 
where $\g^k\in\R^{\Nk}$ is the basis mapping on the grid $\XX^k$ exactly as in Section~4.
In this section, we 
consider 
design of a query $q=(A,B)$ 
which contains lotteries 
\bgeqn 
\label{eq:query-AB}
 A=\{(r_i^A,p^A_i)\}_{i=1}^{n_A}, \qquad B=\{(r^B_i,p^B_i)\}_{i=1}^{n_B}.
\edeqn
Unlike Definition~\ref{def:lottery}, here we allow $n_A$ to be different from $n_B$. 
To ease the exposition in the 
forthcoming discussions about pairwise comparisons, we introduce 
the following function
on $\UU_L$:
\bgeqn 
\label{eq:Phi_q(u)}
 \Phi_q(u):=\sum_{i=1}^{n_A} p^A_i u(r^A_i)-\sum_{s=1}^{n_B} p^B_s u(r^B_s),
\edeqn 
which is the difference between the expected utility of $A$ and 
$B$.


\subsection{Piecewise-linear approximation error}
Our fundamental idea is to use the piecewise-linear function to approximate the true unknown utility function and refine the approximation by increasing the number of pieces. 
The first technical result is about the error bound of the piecewise-linear approximation. 
The next proposition provides a sharp error bound when the original utility function is 
 a normalized nondecreasing Lipschitz  
function.

\begin{proposition}[Error bound for piecewise-linear approximation]\label{lem:pla_error}
Let $\Tk u$ be defined as in \eqref{eq:T_k-PLA}. Then for any $u\in\UU_L$,
\bgeqn 
\label{eq:PLA-Lip}
 \|u-\Tk u\|_\infty \leqslant \frac{L\hk}{4},
\edeqn  
where $h_k$ is defined as in \eqref{eq:X_k-h_k}.
\end{proposition}

Proposition~\ref{lem:pla_error} shows that the piecewise-linear approximation error is controlled by the mesh size $\hk$. This suggests that refining the breakpoint set should eventually make the approximation accurate. However, because of the PLA error, direction error may occur, which we discuss in Section~\ref{sec:direction_error}.

\subsection{Direction error of a cut caused by PLA\label{sec:direction_error}}

{\color{black}

Our CPM works on
piecewise-linear utility functions.
However, the approximation error of PLA may cause a wrong cut. 
To explain this issue, let us use ${\cal U}_{L,k}$ to denote the set of all normalized $L$-Lipschitz piecewise-linear 
utility functions with breakpoint set $\XX^k$ (see \eqref{eq:X_k-h_k}) and ${\cal U}_L$ 
be the set of all normalized 
$L$-Lipschitz utility functions which contains the true $u^*$ (see \eqref{def:Lip-utility}). 
By the definition of $T_k$ (see \eqref{eq:T_k-PLA}) 
\bgeqn 
T_ku^* \in T_k{\cal U}_L 
={\cal U}_{L,k}\subset {\cal U}_L, 
\edeqn 
where $T_k{\cal U}_L:=\{T_ku\mid u\in{\cal U}_L\}$.
For a given query $(A,B)$, when
the DM chooses $B$, from an expected-utility perspective, it means
$\mathbb{E}[u^*(A)]<\mathbb{E}[u^*(B)]$. 
Our CPM uses this information to construct a cut
$\mathbb{E}[u_k(A)]<\mathbb{E}[u_k(B)]$ where $ u_k\in{\cal U}_{L,k}$. However, 
there is a potential risk 
that
$\mathbb{E}[T_k u^*(A)]\geqslant \mathbb{E}[T_ku^*(B)]$ because of the discrepancy 
between $u^*$ and $T_ku^*$.
If this occurs, we will have 
$T_ku^*\not\in {\cal M}_k$, where
${\cal M}_k:={\cal U}_{L,k}\cap\{u_k\in{\cal U}_{L,k}\mid\mathbb{E}[u_k(A)]<\mathbb{E}[u_k(B)]\}$.
This is also known as the \underline{direction error} 
defined by 
\cite{zhang2025modified}. 
}

\begin{definition}[Direction error of a cut caused by PLA]
\label{def:direction_error}
 Let
$\Phi_q(u)$ be defined as in \eqref{eq:Phi_q(u)}.
Consider query $q=(A,B)$ defined in \eqref{eq:query-AB} and piecewise-linear interpolant $\Tk u$ defined in 
\eqref{eq:T_k-PLA}.
$q$ is said to  
incur
a \emph{direction error} 
on the grid
$\XX^k$ if the response based on the true utility $u^*$ 
and the response predicted by the current piecewise-linear
approximation $\Tk u^*$ 
lead to opposite preferences, i.e.,
\bgeqn 
\label{eq:direction-error}
 \Phi_q(u^*)\,\Phi_q(\Tk u^*)<0.
\edeqn 
\end{definition}

We give a simple example to illustrate.

\begin{example}[Direction error and preference conflict]
    Consider the true utility function
$u^*(x)=1-(x-1)^2$  defined over $[0,1]$ and grid 
 $\XX^0=\{0,0.1,0.5,1\}$.
The piecewise-linear interpolant on $\XX^0$ is $(T_0u^*)(x)=(0.19,0.56,0.25)\g(x)$. 
Consider the query with
$
 A=\{\rr^A=(0.3),\p^A=(1)\},
 $ and $
 B=\{\rr^B=(0.1,0.5),\p^B=(0.45,0.55)\}.
$
Then
$
\Phi_q(u^*)
=u^*(0.3)-0.45\,u^*(0.1)-0.55\,u^*(0.5)=0.51-(0.45\times 0.19+0.55\times 0.75)=0.012>0.
$
By contrast,
$
\Phi_q(T_0u^*)
=(T_0u^*)(0.3)-0.45\,(T_0u^*)(0.1)-0.55\,(T_0u^*)(0.5)=0.47-0.498=-0.028<0.
$
In this case, the query suffers
a direction error on the grid $\XX^0$.

In practice, $u^*$ is unknown. What the modeler
knows is that the DM chooses $A$. If we use the piecewise-linear interpolant to describe 
the DM's choice, then we obtain
$$
\vv^\top\big(\g^0(0.3)-0.45\,\g^0(0.1)-0.55\,\g^0(0.5)\big)>0.
$$
Observe that
$$
\g^0(0.3)-0.45\,\g^0(0.1)-0.55\,\g^0(0.5)=(0,-0.05,0)^\top.
$$
Substituting  the quantity into the inequality above leads to 
$
-0.05\,v_2>0,
$
and hence $v_2<0$, which violates the monotonicity of $u^*$.
This observation demonstrates that 
the direction error may be detected 
in the process of eliciting preference based on the PLA interpolant without knowing the true $u^*$.
In other words,  a query with direction error may lead to an infeasible or contradictory update of the
 polyhedron. We call this phenomenon preference conflict.
\end{example}

To avoid 
the direction error,
we 
confine the outcomes 
 of $A$ and $B$ to the current set of breakpoints in which case
the preference represented by $u^*$ is consistent with the preference represented by $u_k$
 (e.g.~$\mathbb{E}[u^*(A)]<\mathbb{E}[u^*(B)]\Leftrightarrow
\mathbb{E}[u_k(A)]<\mathbb{E}[u_k(B)]$).
 The next proposition addresses this.


\begin{proposition}[No direction error on breakpoint-supported queries]\label{lem:no_direction_error}
If every support point of a query $q=(A,B)$ belongs to $\XX^k$, then for every $u\in\UU_L$,
$
 \Phi_q(u)=\Phi_q(\Tk u).
$
In particular, no breakpoint-supported query  
causes a direction error.
\end{proposition}

\begin{proof}
Since $u(x_i)=\Tk u(x_i)$ for all $x_i\in\XX^k$, the two comparison functionals are identical.
\end{proof}

The next proposition shows how 
a query with the specified 
outcomes in ${\cal X}^k$ may be constructed.



\begin{proposition}[Breakpoint-supported implementation of a CPM cut]\label{cor:bpquery_new}
Let $\XX^k=\{x^k_1,\cdots,x_{N_k+1}^k\}$ be a fixed set of breakpoints 
and 
$
 \mathcal H=\{\vv\in\R^{\Nk-1}:\bal^\top(\vv-\cc)=0\}
$
be a prescribed hyperplane in the current increment space.
Let $
 r_0=x_1^k $ and 
 $ r_i=x_{i+1}^k$ for 
 $i=1,\ldots,\Nk$ in Algorithm~\ref{alg:GQ}.
Then 
there exists a query $q=(A,B)$ whose support
is contained in $\XX^k$ and 
the resulting separating hyperplane is precisely $\mathcal{H}$. Moreover, for every
$u\in\UU_L$,
$
 \Phi_q(u)=\Phi_q(\Tk u)=\gamma\,\bal^\top\big(\Pik(u)-\cc\big)
$
for some $\gamma>0$.
\end{proposition}

The result follows from the fact that if $
 r_0=x_1^k $ and 
 $ r_i=x_{i+1}^k$ for 
 $i=1,\ldots,\Nk$ in Algorithm~\ref{alg:GQ}, then the conditions of Theorem~\ref{theo:solution} are satisfied.
Inspired by this observation, we design an adaptive-breakpoint algorithm in which every inner CPM query is
supported only on the current set of breakpoints, and every outer iteration adds exactly one new midpoint as stated in Step 2 of Algorithm \ref{alg:VCPM_CU} .

\subsection{Adaptive-breakpoint CPM Algorithm}

Section \ref{sec:univariate} introduces the tractable/implementable CPM method to generate $N_k-1$ queries at each major iteration $k$ under a $N_k$-piecewise-linear nominal utility function on $\mathcal{X}^k$. 
However, when the DM’s true utility is a general continuous utility function, the CPM on ``fixed'' $\mathcal{X}^k$ can only converge to $T_k u$ rather than $u$, even if we run \underline{many  
major iterations on $\mathcal{X}^k$}. 

To close this gap, we design 
an  algorithmic framework with two loops: 
in the outer loop called \underline{rounds}, 
we adaptively increase 
the number of breakpoints by one
in each round.
In the inner loop, called major iterations, at each round, 
we 
implement CPM 
with a fixed 
set of breakpoints.
We start 
with an initial bounded polyhedron $\PP^{0,0}\subset\R^{N_0-1}$:
\bgeqn \label{def:P^0,0}
\PP^{0,0}= \Big\{\vv\in\R^{N_0-1}\mid 0\leqslant (\RR\vv)_i\leqslant L(x_{i+1}^0-x_i^0), \ \text{for } i=1,\ldots,N_0,\; \e^\top(\RR\vv)\leqslant1\Big\}.
\edeqn 
Since $u^*\in\UU_L$, then $\Pi_0(u^*)\in \PP^{0,0}$.
At the beginning of 
round
$k$,  we 
denote the polyhedron 
of the increment vectors 
with the current set of breakpoints $\mathcal{X}^k$ as 
$\PP^{k,0}$
and 
the corresponding ambiguity set of utility functions 
by $\Uk:=
\big\{u\in\UU_L\mid\Pik(u)\in \PP^{k,0}\big\}.$

\paragraph{Inner CPM major iterations at each round.}
Fix a threshold $\epsilon\in(0,1)$. At each outer round
$k$, run $M_k$ \emph{CPM major iterations} on the current grid $\XX^k$ according to Algorithm~\ref{alg:VCPM_PLU} until
\begin{equation}\label{eq:stop_cpm}
    r(\B^{k,M_k})\leqslant \epsilon\,r(\B^{k,0}),
\end{equation}
where $\B^{k,i}$ is the MEB of $\PP^{k,i}$.
Define
\bgeqn 
 \Pkp:=\PP^{k,M_k},
 \qquad
 \Ukp:=\{u\in\UU_L:\Pik(u)\in \Pkp\}.
\edeqn

\begin{proposition}\label{prop:M_k}
When $N_0\geqslant3$, condition \eqref{eq:stop_cpm} holds when the number of CPM iterations at round $k\;(k\geqslant0)$ satisfies
\bgeqn M_k\geqslant\left\lceil\frac{2\ln\epsilon}{\ln(\Nk-2)-\ln(\Nk-1)} \right\rceil.
\edeqn 
\end{proposition}
\begin{proof}
    By Theorem~\ref{theo:convergence_linear}, $r(\B^{k,M_k})\leqslant \sqrt{\frac{\Nk-2}{\Nk-1}}^{M_k}r(\B^{k,0})$. Thus, inequality  \eqref{eq:stop_cpm} gives rise to $\sqrt{\frac{\Nk-2}{\Nk-1}}^{M_k}\leqslant\epsilon$, and hence 
    the conclusion.
\end{proof}
At the end of the 
inner CPM major 
iterations based on grid ${\cal X}^k$, we update the grid
by adding a new breakpoint/grid point  which is the 
 midpoint of 
the interval between two consecutive breakpoints with largest mesh size,  
i.e.
\begin{equation}\label{eq-jk}
    \XX^{k+1}:=\XX^k\cup \left\{\frac{x_{j_k}^k+x_{j_k+1}^k}{2}\right\},
\qquad
j_k\in\mathop{\arg\max}_{1\leqslant i\leqslant N_k}(x_{i+1}^k-x_i^k).
\end{equation}
Since the increment vectors on the coarse and refined grids are in spaces with different dimensions, we
describe the lift first in the full increment space and then return to the standard reduced representation.

\paragraph{Polyhedron lift.}
	Let
	\bgeqn 
	x_{\mathrm{new}}^{k}:=\frac{x_{j_k}^k+x_{j_k+1}^k}{2}.
	\edeqn 
    Since $(x_{j_k}^k,x_{j_k+1}^k)$ is the widest interval, $x_{j_k+1}^k-x_{j_k}^k=h_k$.
	Define the full-increment polyhedron associated 
    with 
    $\PP^{k,+}$ by 
    $
	\widehat \PP^{k,+}
	:=
	\left\{
	\RR\vv\in\mathbb{R}^{N_k}
	\;\middle|\;
	\vv\in \PP^{k,+}
	\right\},
$
where $\RR{\vv}\in \widehat \PP^{k,+}$ satisfies $v_i\geqslant 0$ for $i=1,\dots,N_k$ and
	$\e^\top (\RR{\vv})=1$.
Given $\RR{\vv}=({v}_1,\dots,{v}_{N_k})^\top\in \widehat \PP^{k,+}$, where ${v}_{N_k}=1-\sum_{i=1}^{N_k-1}v_i$, we first obtain its widest interval $j_k$ by \eqref{eq-jk}. 
Then we define the lifted polyhedron on the refined grid in the full increment space 

\bgeqn\label{eq:P_hat}
	\widehat \PP^{k+1,0}
	:=
	\left\{
({v}_1,\dots,{v}_{j_k-1},\,a,\,{v}_{j_k}-a,\,{v}_{j_k+1},\dots,{v}_{N_k})^\top
	\in\mathbb{R}^{N_k+1}
	\;\middle|\;
\begin{array}{c}
\RR{\vv}=(v_1,\ldots,v_{N_k})^\top\in \widehat \PP^{k,+},\\
	0\leqslant a\leqslant v_{j_k}, \\
a\leqslant \frac{Lh_{k}}{2},\
	v_{j_k}-a\leqslant\frac{Lh_{k}}{2}
\end{array}
\right\}.
\edeqn
	By construction, every $\bm\eta\in \widehat \PP^{k+1,0}$ still satisfies
$\eta_i\geqslant 0$, $\e^\top \bm\eta =1.
$
Given $\widehat \PP^{k+1,0}$, we map it back to the space of the first $N_k$ attributes:
\bgeqn\label{eq:P^{k+1,0}}
\PP^{k+1,0}	:=
\left\{ \vv\in\mathbb{R}^{N_k}
	\;\middle|\;
	R\vv\in \widehat \PP^{k+1,0}
	\right\}.
\edeqn
Compared with $\widehat\PP^{k,+}$, all increments except the $j_k$-th one in $\widehat\PP^{k+1,0}$ remain unchanged, while the $j_k$-th increment is split into two nonnegative child increments, each of which is bounded by the original parent increment. We denote the MEB of $\PP^{k+1,0}$ by $\B^{k+1,0}$.
The reason we lift $\PP^{k,+}$ to $\widehat\PP^{k,+}$ and make the splitting on $\widehat\PP^{k,+}$ rather than $\PP^{k,+}$ is that the widest interval 
may happen to be the last interval $[x^k_{N_k-1},x^k_{N_k}]$ and thus we make the splitting in the lifted space.


We are now ready to present 
the complete adaptive-breakpoint CPM algorithmic procedures for 
estimating a 
general true nonlinear utility function 
in Algorithm \ref{alg:VCPM_CU}.

\vspace{0.3cm}
\begin{breakablealgorithm}\label{alg:VCPM_CU}
    \caption{Adaptive-breakpoint CPM for estimating true nonlinear utility} 
\begin{flushleft}
\textbf{Input:} initial set of breakpoints $\XX^0$ with 
$|\XX^0|=N_0+1$,
initial polyhedron $\PP^{0,0}$ defined as in \eqref{def:P^0,0}, inner precision $\epsilon\in(0,1)$, mesh precision $\tau>0$, initial mesh size $h_0$ and {round index} $k = 0$.
\\
\textbf{While} mesh size $\hk> \tau$ \textbf{do}
\end{flushleft}
\begin{itemize}[leftmargin=0pt,label={},itemsep=0.6ex]
\item\textbf{Step 1.} Run $M_k$ major CPM 
iterations with $r_0=x^k_1$ and $r_i=x^k_{i+1}$ for $i=1,\ldots,N_k$
on the fixed grid $\XX^k$ and obtain updated polyhedron $\Pkp$ satisfying condition \eqref{eq:stop_cpm}.

    \item\textbf{Step 2.}  Update ${\cal X}^k$ by adding the midpoint of one of the widest intervals: $\XX^{k+1}:=\XX^k\cup\{\frac{x_{j_k}^k+x_{j_k+1}^k}{2}\}$, where $j_k\in\mathop{\arg\max}_i(x_{i+1}^k-x_{i}^k)$. Compute mesh size $h_{k+1}$ of $\XX^{k+1}$ defined in \eqref{eq:X_k-h_k}.

   \item\textbf{Step 3.}  Construct $\PP^{k+1,0}$ on grid $\XX^{k+1}$ according to \eqref{eq:P^{k+1,0}}.

    \item\textbf{Step 4.} Set $k:=k+1$. Go back to Step 1.
\end{itemize}
\begin{flushleft}
    \textbf{Output:} $\XX^{k}$ and $\PP^{k,0}$. 
\end{flushleft}
\end{breakablealgorithm}

\subsection{Convergence of Algorithm~\ref{alg:VCPM_CU}}

One of the important features of Algorithm~\ref{alg:VCPM_CU} is that
the true utility function $u^*$ remains in the ambiguity set $\mathcal{U}^{k}$ at any stage of preference elicitation. The next lemma 
states this. 

\begin{lemma}
[Projection and consistency of the lift]\label{prop:lift_projection}
Define the projection $ \pi_k:\R^{\Nk}\to\R^{\Nk-1}$
by
\bgeqn\label{def:projetion}
 \pi_k\left(v_1,\ldots,v_{j_k-1},v_{j_k}^-,v_{j_k}^+,v_{j_k+1},\ldots,v_{\Nk-1}\right)
 :=\left(v_1,\ldots,v_{j_k-1},v_{j_k}^-+v_{j_k}^+,v_{j_k+1},\ldots,v_{\Nk-1}\right).
\edeqn
Then the following statements hold for $k=1,2,\cdots.$
\begin{enumerate}[label=(\roman*),leftmargin=1.5em]
    \item $\pi_k(\PP^{k+1,0})=\Pkp$;
    \item if $u^*\in\Ukp$, then $u^*\in\UU^{k+1}$;
    \item $\UU^{k+1}\subseteq \Ukp$.
\end{enumerate}
\end{lemma}
The mapping $\pi_k$ 
recombines two child increments caused by the bisection of the interval
$[x^k_{j_k},x^k_{j_k+1}]$ into the original parent
increment. 
Part (i) of the lemma ensures that the lifted polyhedron projects exactly onto $\Pkp$.
Part (ii) asserts that
the true utility is not excluded by inserting the new breakpoint, i.e., $\Pi_k(u^*)\in\Pkp,\PP^{k+1,0}$.
Part (iii) ensures that every utility function feasible on the
refined grid remains feasible when restricted to the previous grid.
With the lemma, we can 
derive the relationship between $\B^{k+1,0}$ (the MEB of $\PP^{k+1,0}$ lifted from $\PP^{k,+}$)
and $\B^{k,+}$ (the MEB of $\PP^{k,+}$) in terms of their radius.

\begin{theorem}[Quantitative radius control under one-coordinate lift]\label{thm:lift_bound}
Let $\eta_k:=\max\{(\RR\vv)_j\mid \RR{\vv}\in \widehat \PP^{k,+},\;j=1,\ldots,N_k\}.$
Then
\bgeqn
\label{eq:Thm5.1-radius-lift}
r(\B^{k+1,0})
 \leqslant
 \sqrt{\Big(r(\B^{k,+})\Big)^2+{\eta_k^2}}
 \leqslant
 r(\B^{k,+})+\eta_k
 \leqslant
 r(\B^{k,+})+L\hk.
\edeqn 
\end{theorem}

The inequalities \eqref{eq:Thm5.1-radius-lift} ensure that 
the radius of $\B^{k+1,0}$ does not increase from the radius of $\B^{k,+}$
by ${L\hk}$.
The term 
arises from the estimation of the new
components $a$ and 
$v_{j_k}-a$ of the increment vector in the lifted polyhedral $\PP^{k+1,0}$.
Next, we derive a relationship between mesh size $h_k$ and the number of outer rounds $k$. Specifically, we derive an upper bound on $h_k$ in terms of $k$.  The next proposition states this.

\begin{proposition}[Reduction of mesh size]
\label{prop:h_k}
Consider Algorithm~\ref{alg:VCPM_CU}.
Let $h_k :=\max_{1\leqslant i\leqslant \Nk}(x_{i+1}^k-x_i^k)$.
Then 
$h_k\leqslant\frac{2N_0h_0}{k+N_0}$.
\end{proposition}



We are now ready to state the main convergence result of this section.

\begin{theorem}[Convergence of Algorithm~\ref{alg:VCPM_CU} in the utility function space]
	\label{thm:utility_space_convergence}
	Let 
    $\delta:=\overline{x}-\underline{x}$. Assume 
    $u^*\in \UU^0$. Then
    for every $u\in \UU^k$,
\bgeqn 
\label{eq:one-step_delta_u}
\dd_K(u^*,u)\ \leqslant\ \dd_I(u^*,u)\ 
\leqslant	\frac{\delta L h_k}{2}+2\delta\sqrt{N_k-1
}\,r(\B^{k,0}) \leqslant \frac{\delta L h_k}{2}	
+	2\delta\sqrt{N_k-1
}
		\left[\epsilon^k r(\B^{0,0})+L
        \sum_{i=0}^{k-1}\epsilon^i h_{k-1-i}
		\right]    
\edeqn 
 and 
$\sup_{u\in\UU^k}\|u-u^*\|_\infty \to 0$ as $k\to\infty$.
\end{theorem}

The theorem quantifies convergence of the utility estimates in $\UU^k$ toward the true utility $u^*$. 
The term $\delta L h_k/2$ reflects the error introduced by the piecewise-linear approximation. 
The term $2\delta\sqrt{N_k-1}\,r(\B^{k,0})$ accounts for the residual ambiguity remaining after the CPM cuts. 
To the best of our knowledge, this provides the first theoretical guarantee of convergence for general nondecreasing Lipschitz-continuous univariate utility functions using a polyhedral method.

\subsection{Numerical illustration}\label{subsec:continuous_example}

In this subsection, we present a numerical example to illustrate how adaptive-breakpoint CPM elicits a general nonlinear
utility function in practice. 
We consider a true utility function $u^*:[0,4]\rightarrow[0,1]$ 
$$
u^*(x)= \frac{1}{1-e^{-24}}(1-e^{-6x}), \quad \text{for}\; x\in[0,4],
$$
which is monotone increasing, normalized and $L=\frac{6}{1-e^{-24}}$-Lipschitz continuous on $[0,4]$. 
We carry out two tests with different initial breakpoints 
$
\mathcal X^0_{(1)}=\{0,0.6,2,4\}
$ and 
$
\mathcal X^0_{(2)}=\{0,3,3.5,4\}.
$
For both tests, the initial ambiguity set is chosen as
$
\mathcal P^{0,0}
=
\left\{
\vv\in\mathbb R^2
\;\middle|\;
0\leqslant v_i\leqslant \min\{1,L(x_{i+1}-x_i)\}\; \text{for}\;i=1,2;
\vv^\top \bm e\leqslant 1
\right\},
$
which 
provides monotonicity, normalization, and Lipschitz continuity.
All other settings are the same.

We then run Algorithm~\ref{alg:VCPM_CU}
in the two cases 
to 
construct queries and use the true utility to generate the responses.
We set 
$\epsilon=0.5$ and the stopping tolerance $\tau=0.5$. 
We report the results of the two tests in Tables~\ref{tab:continuous_example_1}~and~\ref{tab:continuous_example_2},
with the following information: 
the number of rounds $k$, 
the dimension of the reduced increment 
$N_k-1$, the mesh size $h_k$, the radius of MEB $r(\mathcal B^{k,0})$ at the beginning of round $k$, 
the radius of MEB $r(\mathcal B^{k,+})$ after $M_k(N_k-1)$ cuts in round $k$,
the number $M_k$ of major CPM iterations, the number of cuts implemented in each round, the cumulative number of queries (N.q),
the distance from the true utility to the current ambiguity set
in the sense of Kantorovich distance 
$
\max_{u\in\mathcal U^k} \dd_K(u^*,u) 
$
and 
scaled Kolmogorov distance 
$
\max_{u\in\mathcal U^k} \dd_I(u^*,u).
$


\begin{table}[!htbp]
	\setlength\tabcolsep{6pt}
	\renewcommand{\arraystretch}{1.2}
	\centering
	\caption{Performance
    of 
    Algorithm~\ref{alg:VCPM_CU}
    in Test 1 with initial breakpoint set $\mathcal X^0_{(1)}=\{0,0.6,2,4\}$.}
	\label{tab:continuous_example_1}
	\begin{tabularx}{\textwidth}{
        p{1.5cm}<{\centering} 
        p{1cm}<{\centering}
        p{1cm}<{\centering}
        p{1.5cm}<{\centering}
        p{1.5cm}<{\centering}
        p{0.6cm}<{\centering}
        p{0.6cm}<{\centering}
        p{0.8cm}<{\centering}
        p{2.3cm}<{\centering}
        p{2.3cm}<{\centering} }
		\hline
round $k$ &  $\Nk-1$   & $\hk$ & $r(\B^{k,0})$ 
& $r(\B^{k,+})$  &$M_k$ & 
Cuts &N.q&$\max\dd_K(u^*,u)$ &$\max\dd_I(u^*,u)$
 \\
		\hline	
0     & 2     & 2     & 0.7071  
& 0.1768  & 2     & 4  &4   & 0.3837  & 1.8596  \\
    1     & 3     & 1.4   & 0.1768       & 0.0786  & 1     & 3  &7   & 0.2585  & 1.6527  \\
    2     & 4     & 1     & 0.0969      & 0.0292  & 2     & 8   &15  & 0.1704  & 1.5429  \\
    3     & 5     & 1     & 0.0292       & 0.0139  & 1     & 5  &20   & 0.1569  & 1.5125  \\
    4     & 6     & 0.7   & 0.0139       & 0.0036  & 2     & 12  &32  & 0.1540  & 1.5101  \\
    5     & 7     & 0.7   & 0.0214       & 0.0058  & 2     & 14 &46   & 0.1493  & 1.5095  \\
    6     & 8     & 0.6   & 0.0058      & 0.0015  & 2     & 16 &62   & 0.1481  & 1.5063  \\
    7     & 9     & 0.5   & 0.6882       & 0.1720  & 2     & 18 &80   & 0.0750  & 0.9152  \\
    8     & 10    & 0.5   & 0.1720       & 0.0430  & 2     & 20  &100  & 0.0566  & 0.7982  \\
		\hline
	\end{tabularx}
\end{table}

\begin{table}[!htbp]
	\setlength\tabcolsep{6pt}
	\renewcommand{\arraystretch}{1.2}
	\centering
	\caption{Performance
    of 
    Algorithm~\ref{alg:VCPM_CU}
    in Test 2 with initial breakpoint set $\mathcal X^0_{(2)}=\{0,3,3.5,4\}$.}
	\label{tab:continuous_example_2}
	\begin{tabularx}{\textwidth}{
        p{1.5cm}<{\centering} 
        p{1cm}<{\centering}
        p{1cm}<{\centering}
        p{1.5cm}<{\centering}
        p{1.5cm}<{\centering}
        p{0.6cm}<{\centering}
        p{0.6cm}<{\centering}
        p{0.8cm}<{\centering}
        p{2.3cm}<{\centering}
        p{2.3cm}<{\centering} }
		\hline
round $k$ &  $\Nk-1$   & $\hk$ & $r(\B^{k,0})$ 
& $r(\B^{k,+})$  &$M_k$ & 
Cuts&N.q &$\max\dd_K(u^*,u)$ &$\max\dd_I(u^*,u)$
 \\
		\hline	      
    0     & 2     & 3     & 0.7071  
    & 0.1768  & 2     & 4  &4   & 1.7711  & 3.3037  \\
    1     & 3     & 1.5   & 0.7071      & 0.1768  & 2     & 6  &10   & 0.9584  & 2.8384  \\
    2     & 4     & 1.5   & 0.7071       & 0.2099  & 2     & 8  &18   & 0.5833  & 2.2408  \\
    3     & 5     & 0.75  & 0.2099       & 0.0623  & 2     & 10  &28  & 0.3058  & 1.9022  \\
    4     & 6     & 0.75  & 0.7070       & 0.1809  & 2     & 12  &40  & 0.1576  & 1.2456  \\
    5     & 7     & 0.75  & 0.1809       & 0.0462  & 2     & 14 &54   & 0.0889  & 0.9994  \\
    6     & 8     & 0.75  & 0.0462       & 0.0127  & 2     & 16 &70   & 0.0723  & 0.9662  \\
    7     & 9     & 0.5   & 0.0127       & 0.0042  & 2     & 18  &88  & 0.0661  & 0.9451  \\
		\hline
	\end{tabularx}
\end{table}

We have the following major observations from Tables~\ref{tab:continuous_example_1}~and~\ref{tab:continuous_example_2}:

\begin{itemize}
\item[(a)] \textbf{Convergence of the mesh size $h_k$.}
    The breakpoint updating rule bisects one of the longest intervals at the end of each outer round. 
    Hence, in Tables~\ref{tab:continuous_example_1} and~\ref{tab:continuous_example_2}, although the mesh size $h_k$ may remain unchanged for several consecutive rounds (as the widest interval is not unique in this case), it decreases in a stepwise manner and converges to zero, in alignment with Proposition~\ref{prop:h_k}.
    This explains why the approximation error induced by the piecewise-linear interpolation gradually reduces to zero as the outer round index $k$ increases.

    \item[(b)] \textbf{
    Non-monotonic reduction of the MEB radii.}
    The sequences $\{r(\mathcal B^{k,0})\}$ and $\{r(\mathcal B^{k,+})\}$ are not monotonically decreasing 
    across outer rounds of iterations $k$ in  Tables~\ref{tab:continuous_example_1} and~\ref{tab:continuous_example_2}.
For instance, in Table~\ref{tab:continuous_example_1}, 
$r(\mathcal B^{4,0})=0.0139$ whereas $r(\mathcal B^{5,0})=0.0214$,
and $r(\mathcal B^{4,+})=0.0036$ whereas $r(\mathcal B^{5,+})=0.0058$.
A similar phenomenon can be observed when $k$ changes from $6$ to $7$.
This phenomenon can be explained 
from the following relationships between $\PP^{k,+}$ and $\PP^{k+1,0}$:
    $$
        \PP^{k,0}
        \xrightarrow{\text{inner CPM cuts}}
        \PP^{k,+}
        \xrightarrow{\text{breakpoint updating and lift}}
        \PP^{k+1,0}.
    $$
    The first operation is an inner iteration 
    process of CPM 
    with a fixed 
    set of breakpoints 
    $\XX^k$. By the stopping rule of 
    Algorithm~\ref{alg:VCPM_CU},
    the number $M_k$ of inner major iterations is chosen so that the radius is reduced by at least a prescribed factor $\epsilon=0.5$, i.e.,
    $r(\mathcal B^{k,+}) \leqslant 0.5 
    r(\mathcal B^{k,0})$. This explains why 
    the numbers in column $r(\B^{k,+})$ are strictly less than those in column $r(\B^{k,0})$ at each row.

    The second operation lifts the ambiguity set to a higher-dimensional increment space after a new breakpoint is added.
    The radii of the new MEB 
    in the lifted space
    might increase, i.e., 
 $r(\mathcal B^{k+1,0})>r(\mathcal B^{k,+})$. 
 As a result of the two operations,
  $r(\mathcal B^{k+1,0})$ is not guaranteed to be smaller than $r(\mathcal B^{k,0})$.

 




    \item[(c)] \textbf{Controlled increase from $r(\mathcal B^{k,+})$ to $r(\mathcal B^{k+1,0})$.}
   By Theorem~\ref{thm:lift_bound}, 
   the increase from $r(\mathcal B^{k,+})$ to $r(\B^{k+1,0})$ is bounded 
   by $Lh_k$.
To see how the bound works, observe for instance that
in Table~\ref{tab:continuous_example_1},
$r(\B^{7,0})-r(\B^{6,+})=0.6867<Lh_6=3.60$,
and in Table~\ref{tab:continuous_example_2},
 $r(\mathcal B^{4,0})-r(\mathcal B^{3,+})=0.6447<{Lh_3}=4.50$.
The bounds are 
loose in the two cases. However,
as $h_k$ decreases further,
the bound effect 
will be more significant.
Together with the 
guaranteed reduction from $r(\mathcal B^{k,0})$
to $r(\mathcal B^{k,+})$,
we can
manage the convergence of the radii in Algorithm~\ref{alg:VCPM_CU}.

\item[(d)] \textbf{ 
 The number of queries in a round.}
The number of queries in round~$k$, given by $M_k (N_k-1)$ (running CPM $M_k$ times, each requiring $N_k-1$ queries), is not predetermined. Instead, $M_k$ is adaptively chosen to ensure
$
r(\mathcal B^{k,+}) \leqslant \epsilon\, r(\mathcal B^{k,0}),
$
as stated in Proposition~\ref{prop:M_k}.
In both tables, we set $\epsilon = 0.5$, although it can be chosen as any positive constant less than~$1$ to ensure convergence. Moreover, if $\epsilon$ is adaptively selected such that
$
\epsilon=\epsilon_k < 1-\frac{L h_k}{r(\B^{k,0})},
$
then 
$$r(\B^{k+1,0})\leqslant r(\B^{k,+})+ L h_k
\leqslant 
\epsilon_k r(\B^{k,0})+L h_k
<
\left[1-\frac{L h_k}{
r(\B^{k,0})}\right]r(\B^{k,0})+L h_k
<
r(\B^{k,0}),
$$
in which case the 
sequence
$\{ r(\mathcal B^{k,0}) \}$ 
is column-wise monotonically decreasing.

    \item[(e)] \textbf{Convergence in the functional space.}
The monotonicity of the sequences $\{r(\mathcal B^{k,0})\}$ and $\{r(\mathcal B^{k,+})\}$ is not the key concern. Rather, what truly matters is the monotonic decrease in the size of the ambiguity set of utility functions $\mathcal{U}^k$ corresponding to $\mathcal P^k$ in the functional space. This behavior can be observed from the last two columns in Tables~\ref{tab:continuous_example_1} and~\ref{tab:continuous_example_2}, which show that the size of the ambiguity set decreases monotonically in terms of both functional distances.
    

\item[(f)] \textbf{Role of breakpoints.} 
The choice of the initial 
set of  
breakpoints 
may affect 
the performance of the algorithm.
Comparing round 1 in Tables~\ref{tab:continuous_example_1} and~\ref{tab:continuous_example_2}, the two tests have similar values of $h_k$, but have 
remarkably different values of $\max_{u\in\mathcal{P}^k} \dd_K(u^*,u)$ ($\max_{u\in\mathcal{P}^k} \dd_I(u^*,u)$, resp.). 
This shows that the mesh size alone does not capture fully the size of the ambiguity set
in the utility function 
space as indicated in inequality \eqref{eq:one-step_delta_u}.
The key issue here lies with
the spread
of the breakpoints.
In Test 1, the widest initial interval lies in a relatively flat region of the true utility (see Figure~\ref{fig:Test1}), so its early refinement leads only to a limited reduction in $\max_{u\in\mathcal{P}^k} \dd_K(u^*,u)$ (see Table~\ref{tab:continuous_example_1}). By contrast, in Test 2, although many initial breakpoints are located in the flat part of the utility function, the widest initial interval lies in the region where the utility changes most rapidly (see Figure~\ref{fig:Test2}). 
Hence, the bisection rule 
refines the most relevant interval while the dimension of the increment space is still relatively small. As a result, the worst-case Kantorovich distance falls below $0.1$ 
after 54 queries in Table~\ref{tab:continuous_example_2} (with even larger initial $h_k$), compared with 80 queries  in Table~\ref{tab:continuous_example_1}.
However, the impact of the spread of the initial
breakpoints is reduced after more
breakpoints are added,
for example,
after 100 queries in this experiment.
Thus, prior information about where the utility function varies most may be 
used to set initial breakpoints 
and improve the 
efficiency of the algorithm 
in the initial rounds
without affecting 
the overall 
convergence.

\item[(g)] \textbf{Shape information.}
In this test, we do not include any information about concavity, convexity or S-shapedness of the true utility function in the implementation of CPM. 
However, if there is prior information about shape, then we need to incorporate additional constraints on the polyhedra in the algorithm. 
In Appendix~\ref{sec:convergence_concavity}, we report the numerical test results when the true utility function is 
concave.

\end{itemize}

\begin{figure}[!htbp]
    \centering
    \begin{subfigure}
        {0.49\linewidth}
\includegraphics[width=\linewidth]{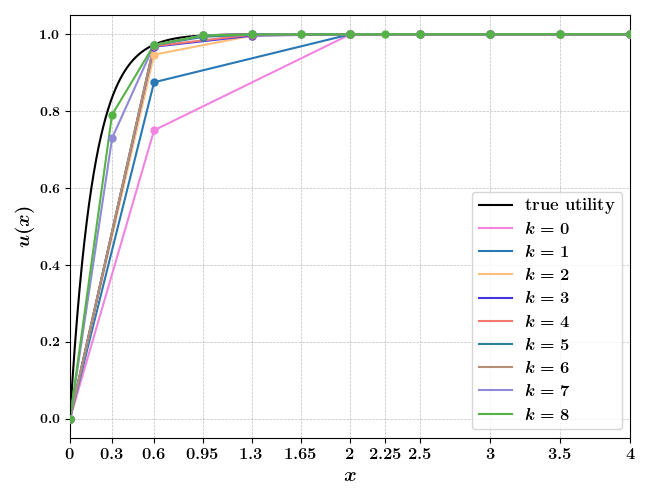}
\caption{\label{fig:WPL1}}
     \end{subfigure}
    \begin{subfigure}{0.49\linewidth}
    \centering
    \includegraphics[width=0.9\linewidth]{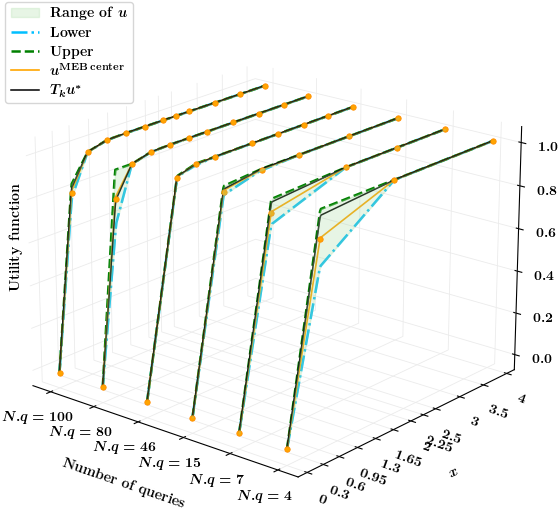}
   \caption{\label{fig:Range1}}    
    \end{subfigure}
    \caption{Test 1 with $\mathcal X^0_{(1)}=\{0,0.6,2,4\}$. (a) Worst-case utility function $\mathop{\arg\max}_{u\in\mathcal U^k} \dd_K(u^*,u)$.
    (b) Range of utility function values corresponding to $\mathcal U^k$.\label{fig:Test1}}
\end{figure}
\begin{figure}[!htbp]
    \centering
    \begin{subfigure}{0.49\linewidth}
\includegraphics[width=\linewidth]{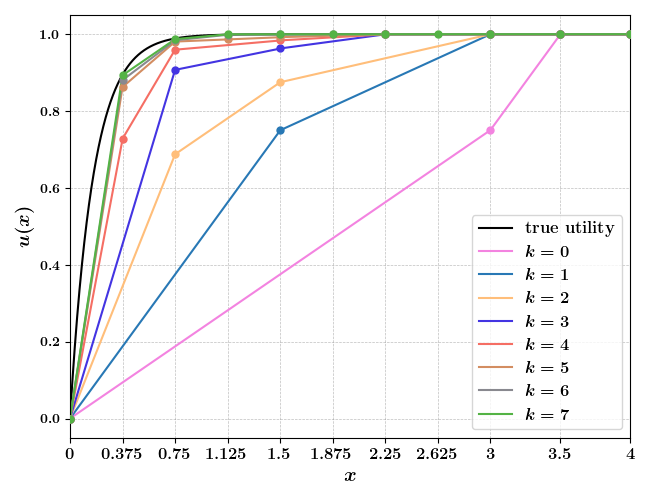}
\caption{\label{fig:WPL2}}
    \end{subfigure}
    \centering
    \begin{subfigure}{0.49\linewidth}
    \centering
    \includegraphics[width=0.9\linewidth]{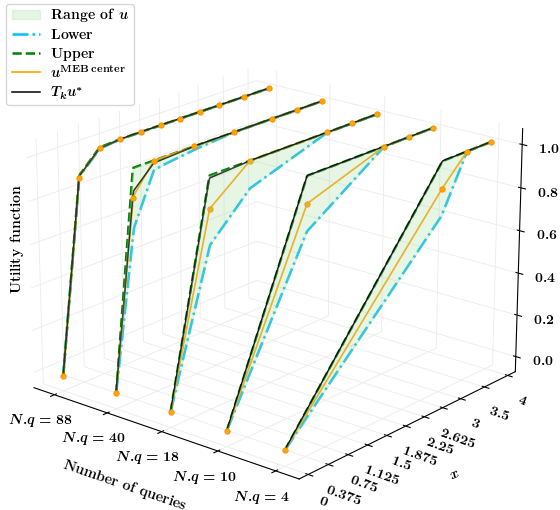}
   \caption{\label{fig:Range2}}    
    \end{subfigure}
\caption{Test 2 with $\mathcal X^0_{(2)}=\{0,3,3.5,4\}$. (a) Worst-case utility function $\mathop{\arg\max}_{u\in\mathcal U^k} \dd_K(u^*,u)$.
(b) Range of utility function values corresponding to $\mathcal U^k$.
\label{fig:Test2}}
\end{figure}

\section{Application to a pricing problem for a wireless headset}
\label{sec:numerical}


In this section, we apply the proposed CPM method
to  
a robust pricing problem for a wireless headset in a multivariate linear utility setting and compare its performance with that of the polyhedral method over a set of metrics introduced below.
We consider customers' preferences 
on a wireless headset product. The product is characterized by four attributes: 
battery life (hours per charge) 
$x_{\text{battery}}$, 
wearing comfort
$x_{\text{comfort}}$, 
active noise cancellation (ANC) level (performance in noisy settings) 
$x_{\text{ANC}}$, 
and price
$p$.
We denote the vector of the four 
attributes 
by $\x=(p,x_{\text{battery}},x_{\text{comfort}},x_{\text{ANC}})^\top$. As in the literature of discrete choice models, 
we assume that customers' preferences can be characterized by a linear utility function 
\begin{equation}
    u(\x)= 
    (\RR \vv)^\top \x =  
    v_1 p+ v_2 x_{\text{battery}} + v_3 x_{\text{comfort}} +v_4 x_{\text{ANC}},    
\end{equation}
where $\vv=(v_1,v_2,v_3)^\top$ is the reduced partworth vector and
$\RR \vv=(v_1,v_2,v_3,v_4)^\top$ is the lifted vector with $v_4=1-\bm e^\top\vv$.  The four partworth components 
are constrained with 
$-1\leqslant v_1
< 0$, $0
<
v_j\leqslant 1$ for $j=2,3,4$, 
and $ \bm e^\top (\RR \vv)=1$.


\subsection{Setup of the tests
\label{Sec:Setup}}

We randomly generate 100 customers by sampling from
truncated Gaussian distributions
$v_1\sim\mathcal{N}_{[-1,0]}(-0.70,0.05^2)$,
${v}_{2}\sim\mathcal{N}_{[0,1]}(0.70,0.05^2)$, 
${v}_{3}\sim\mathcal{N}_{[0,1]}(0.40,0.05^2)$ and computing the last partworth by $v_{4}=1-v_1-v_2-v_3$
to form 100 true partworths $\RR\vv^*_{n}=(v^*_{n,1},v^*_{n,2},v^*_{n,3},v^*_{n,4})^\top$, $n=1,\ldots,100$.
We apply the proposed CPM to each of the customers
with the same initial polyhedron of the partworth vector:
$
\mathcal{P}_n^0
= \{{\vv}\in\R^3\mid -\bm{e}\leqslant {\vv}\leqslant\bm{e}\},
$
for $n=1,\cdots,100$ without knowing
the sign of the first three components of the true 
partworth vector. In each test, we compare the results with those of the polyhedral method 
by \citet{toubia2004polyhedral}.
We set the 
maximum number of 
queries to $T=30$, 
reflecting practical query-budget limitations.
In the case of CPM, it means that there 
are 10 major iterations ($k=0,\ldots,9$) and 3 cuts in each iteration ($i=1,2,3$) in 
Algorithm~\ref{alg:VCPM_LU}, totaling 30 queries. 
To ease the notation, we
 denote the polyhedron of partworth vector
updated after $t$ queries with customer $n$ 
by $\mathcal{P}_n^t$, $t=1,\ldots,T$. For CPM, $\mathcal{P}_n^t$ corresponds to $\mathcal{P}_n^{k,i}$ with $t=k*(N-1)+i$.
We denote the center (AC for POLY and MEB center for CPM) of the polyhedron by $\widehat{\vv}^k_n$.

\subsection{Diagnostic metrics}


We compare the 
CPM and POLY from two perspectives: 
(a) reduction of the 
diameter of 
$\mathcal{P}_n^t$ reflecting the reduction of ambiguity of the true partworth after $t$ queries, 
(b) the quality of the estimates of the partworth vector obtained from the center of the 
polyhedra (AC for POLY and MEB center for CPM) in terms of attribute sign error and utility regret. We denote the center by $\RR\widehat \vv^{%
t}_{n}=(\widehat v^{%
t}_{n,1},\widehat v^{%
t}_{n,2},\widehat v^{%
t}_{n,3},\widehat v^{%
t}_{n,4})$ 
after $t$ interactions 
on DM $n$.
\begin{itemize}
  
\item 
\textbf{The size of the ambiguity set $\mathcal{P}_n^t$ (Polyhedron diameter)}. 
The diameter of the polyhedron is calculated by
$
d_n^t = \max_{\vv,\vv' \in \mathcal{P}_n^t }\lVert \vv - \vv' \rVert_2.
$
As we discussed in Theorem \ref{theo:convergence_linear},
this is one of the key metrics
to be used to examine 
the efficiency of the CPM. 
Figure~\ref{fig:Ep1}(e) depicts the comparative performance 
of CPM and POLY on this metric over 100 independent 
customers. The solid curves represent the mean values across the 100 customers
and shaded 
areas represent the range 
of the values across the 100 customers.

\item 
\textbf{
Partworth sign correctness}.
At each iteration, we define {\color{black}the sign correctness indicator of the 
AC/MEB
center}
\bgeqn
\mathrm{SC}^{t}_{n}
=
\sum_{j=1}^{4} \mathds{1}
\!\Big(
\operatorname{sign}(\widehat v^{t}_{n,j})
= 
\operatorname{sign}({v}^{*}_{n,j})
\Big),
\edeqn
where $\mathds{1}_S$ is the indicator function of set $S$,
and 
{\color{black}the worst-case sign correctness indicator of the polyhedron
\bgeqn
\mathrm{WSC}^{t}_{n}
=
\inf_{\vv\in\mathcal{P}^{t}_n}
\sum_{j=1}^{4} \mathds{1}
\!\Big(
\operatorname{sign}( (\RR\vv)_{j})
= 
\operatorname{sign}({v}^{*}_{n,j})
\Big).
\edeqn
}
This metric examines the
accuracy of the 
estimated attribute components in terms of their signs instead of 
 magnitude.
The metric is closely related to the “hit rate” in \citep{huber1993effectiveness,saure2019ellipsoidal}, which is used to measure predictive accuracy at the choice level, defined as the proportion of correctly predicted choices for each respondent.
Figure~\ref{fig:Ep1}(a) shows the attribute sign correctness of the 
AC/MEB center and Figure~\ref{fig:Ep1}(b)
shows the worst-case sign correctness  
indicator of the polyhedron.

\item \textbf{Utility estimation regret.}
Let 
$\mathcal{X}_{\text{cand}}=[0,1]^4$
be the feasible set of the candidate product designs.
We define the regret
\bgeqn
\mathrm{Regret}_n^{t}
=
\mathop{\max}\limits_{\x\in\mathcal X_{\text{cand}}} \left[(\RR{\vv}_n^*)^{\top}\x
-
(\RR\widehat \vv^{t}_n)^\top \x\right],
\edeqn
which 
measures the maximum 
difference between the true utility and the utility with the partworth vector estimated via 
the AC and MEB centers respectively over feasible set of attributes.
We also 
compare the worst-case utility regret of the polyhedron defined by
\bgeqn
\mathrm{WRegret}_n^{t}
= \sup_{\vv\in\mathcal{P}^{t}_n}
\mathop{\max}\limits_{\x\in\mathcal X_{\text{cand}}}\left[ (\RR{\vv}_n^*)^{\top}\x
-
(\RR \vv)^\top \x\right],
\edeqn
which 
measures the worst-case difference 
over the feasible set of partworth vectors.


\end{itemize}


\begin{figure}[htbp]
 \vspace{-3mm}
    \centering
    \begin{subfigure}[t]{0.48\textwidth}
        \centering
        \includegraphics[width=\textwidth]{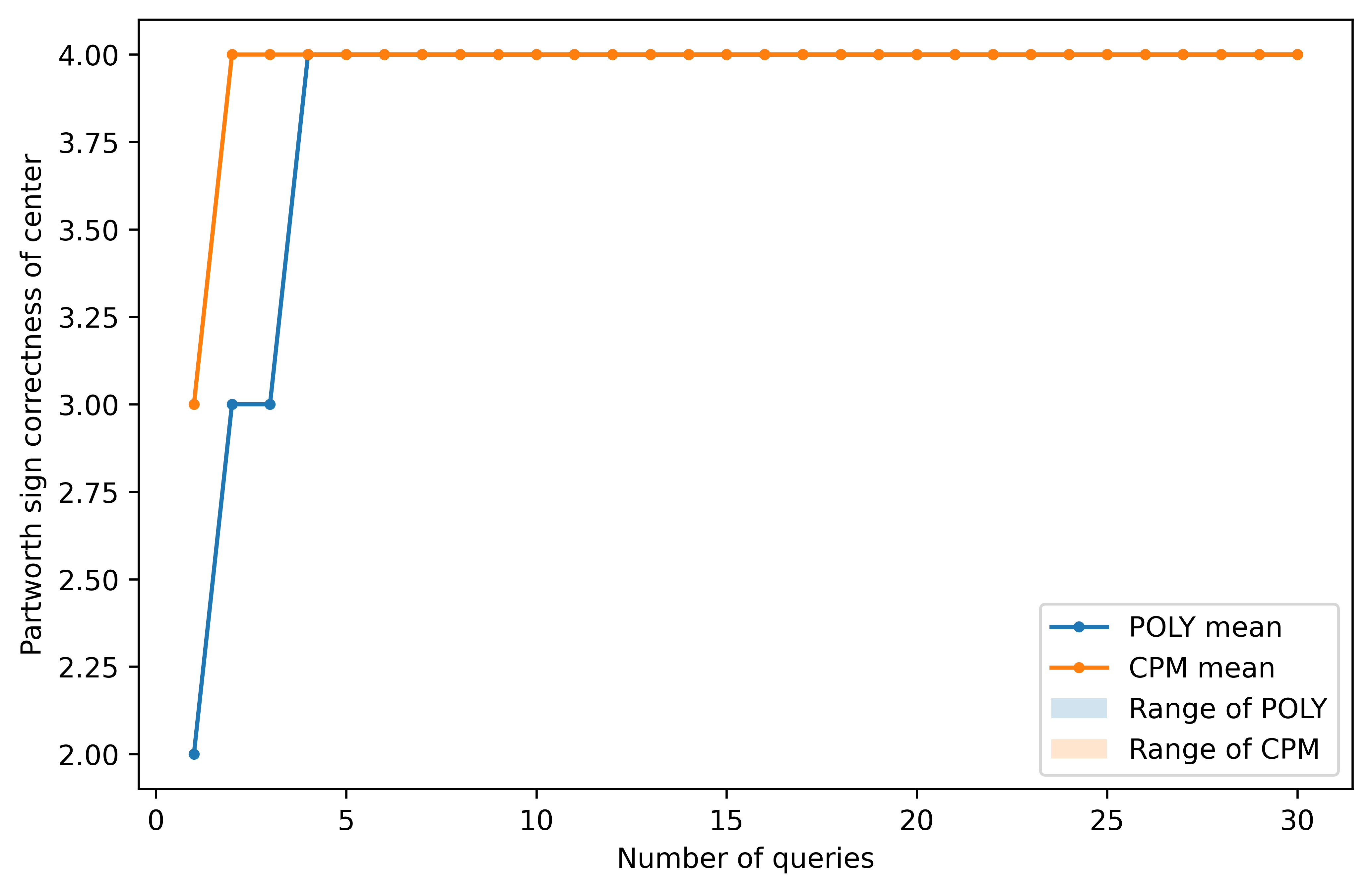}
        \caption{
        }\label{fig:SC1}
    \end{subfigure}
    \begin{subfigure}[t]{0.48\textwidth}
        \centering
       \includegraphics[width=\textwidth]{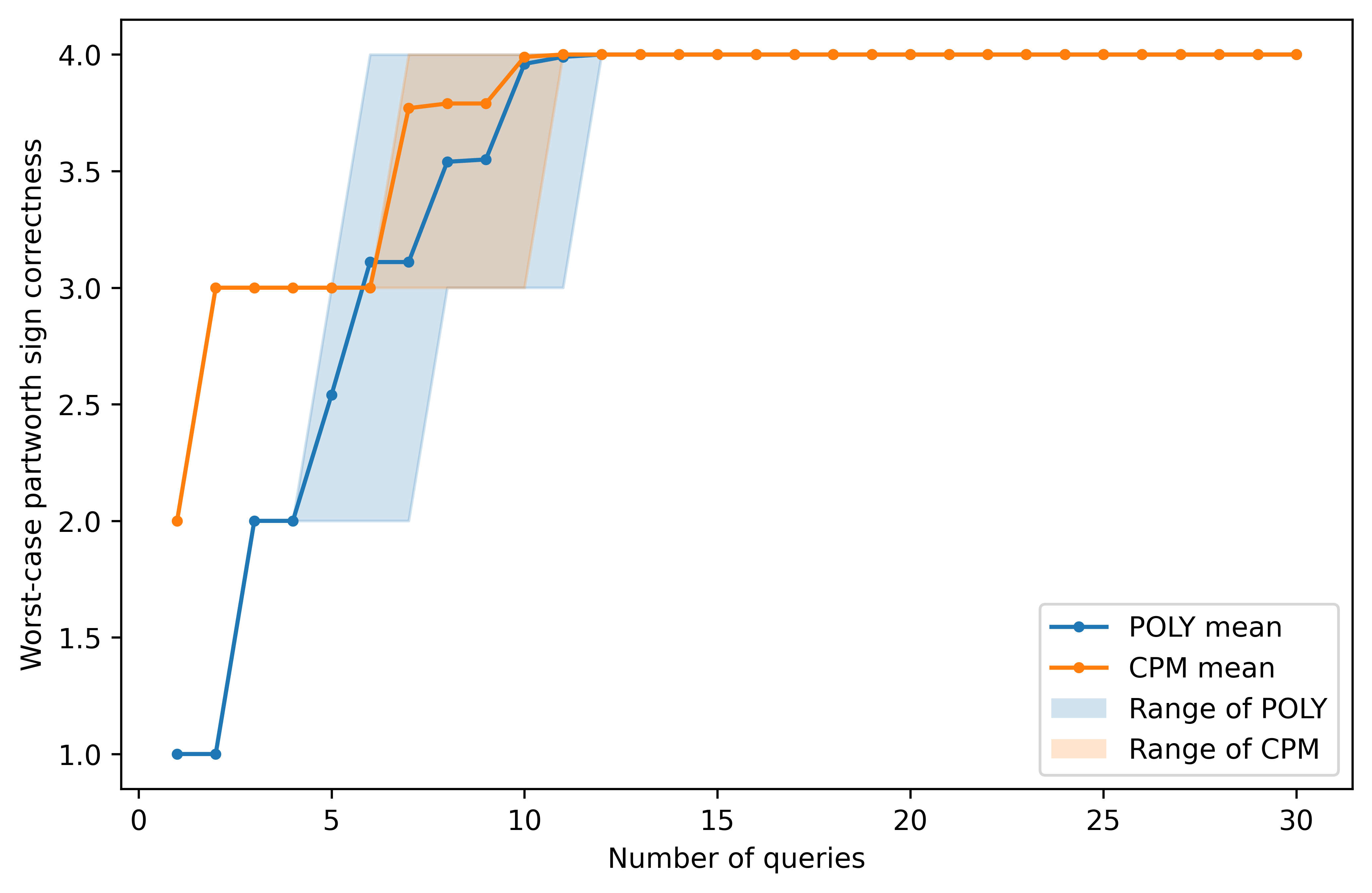}
        \caption{}\label{fig:wc-SC1}
    \end{subfigure}\\
    \begin{subfigure}[t]{0.48\textwidth}
        \centering
        \includegraphics[width=\textwidth]{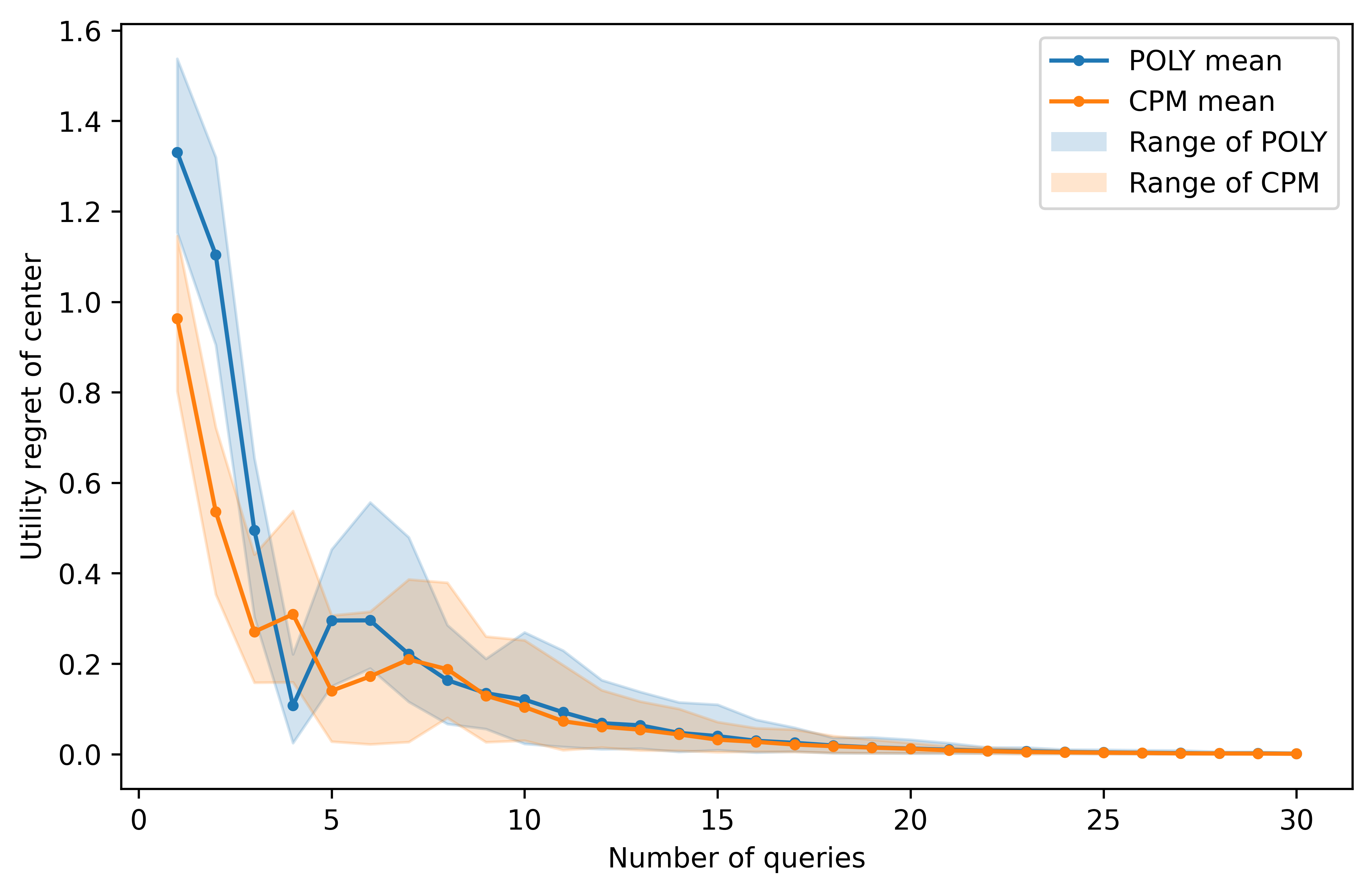}
        \caption{}\label{fig:UR1}
    \end{subfigure}
       \begin{subfigure}[t]{0.48\textwidth}
        \centering
        \includegraphics[width=\textwidth]{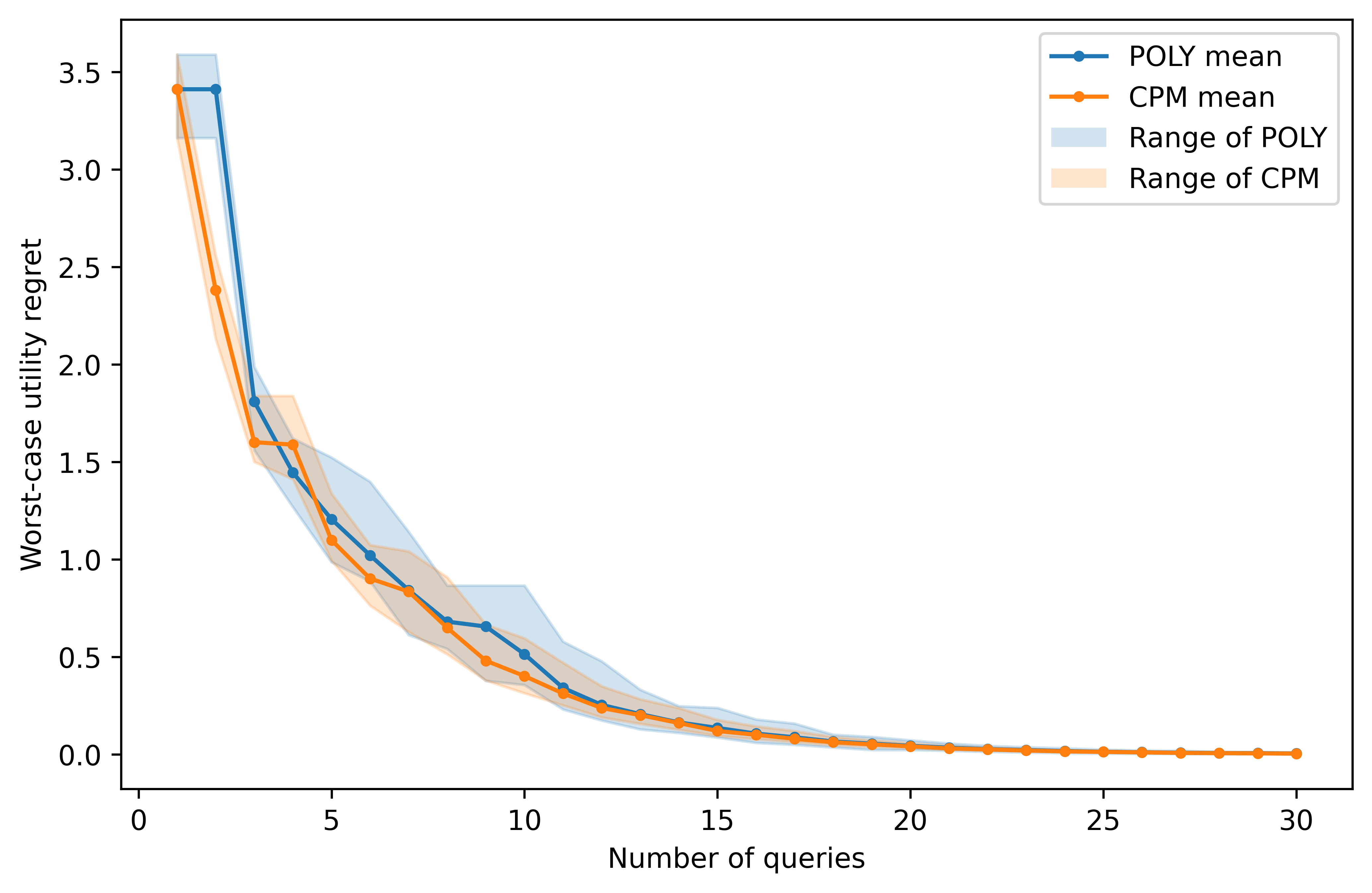}
        \caption{\label{fig:wc-UR1}}
    \end{subfigure}
    \\
    \begin{subfigure}[t]{0.48\textwidth}
        \centering
        \includegraphics[width=\textwidth]{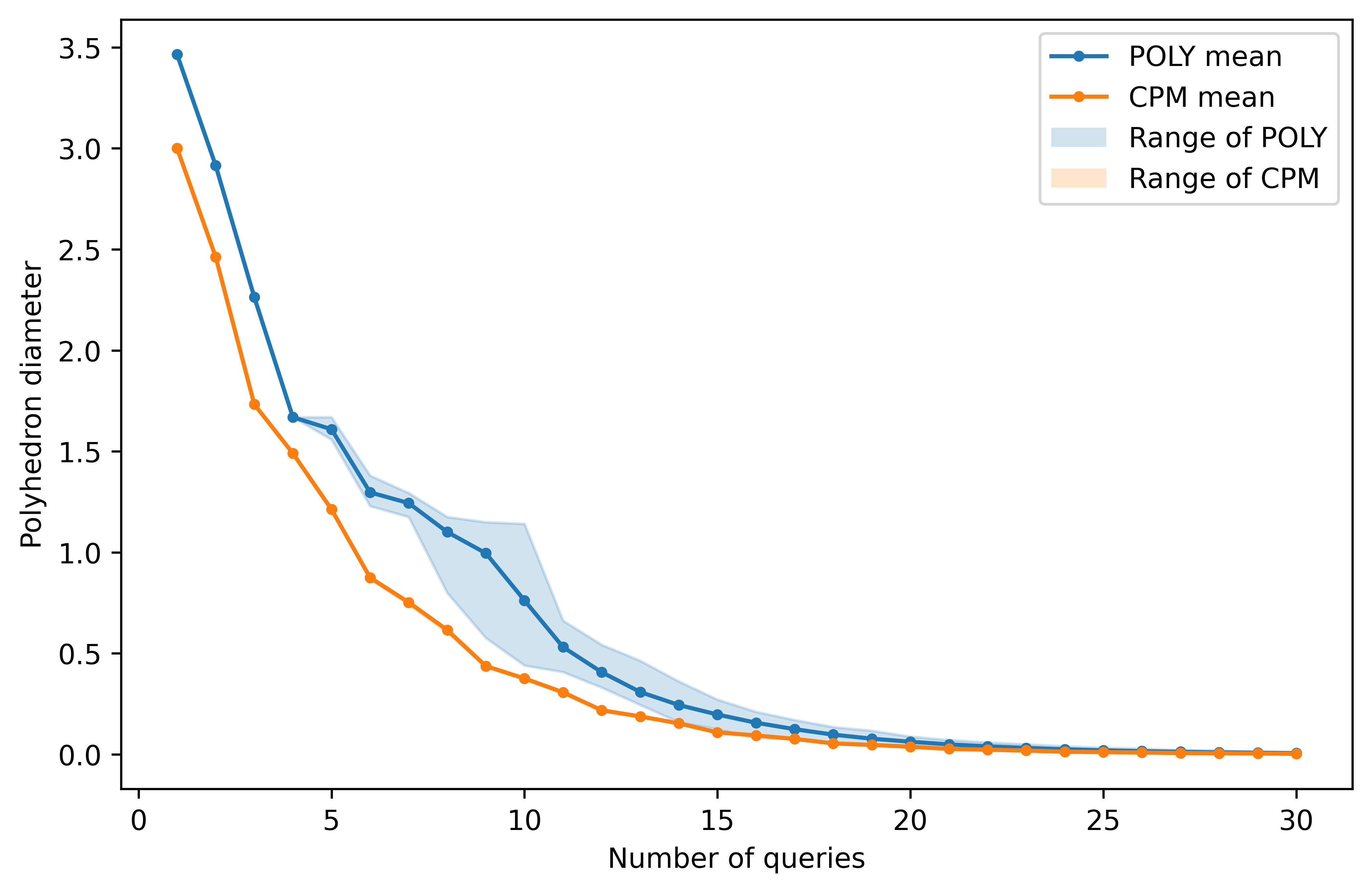}
        \caption{\label{fig:PD1}}
    \end{subfigure}
    \caption{\footnotesize Comparative analysis of CPM and POLY 
    over 100 customers (solid curves indicate the means and shaded areas indicate the ranges).
    (a) 
    Partworth sign correctness based on the AC/MEB centers.
    (b) Worst-case partworth sign correctness.
    (c) Utility regret based on the AC/MEB centers.
    (d) Worst-case utility regret.
    (e) Polyhedron diameter.
    \label{fig:Ep1}
    }

    \end{figure}
We have made the following major observations from the test results 
plotted in Figure \ref{fig:Ep1}.

\begin{itemize}

\item[(a)]  
Figure~\ref{fig:Ep1}(a)--(b) depict the change of $\mathrm{SC}^{k}_{n}$
and $\mathrm{WSC}^{k}_{n}$ respectively as the number of queries increases.
We can see that
CPM achieves better sign correctness overall than POLY in both cases. 
For each customer, both the sign correctness of the AC/MEB center and the worst-case sign correctness over the polyhedron are nondecreasing; the former (center-based) is customer-specific, whereas the latter (worstcase-based) holds in general and is guaranteed by the nestedness of the polyhedra.
After sufficient queries, CPM identifies the correct signs for all customers: specifically, ensuring that all partworths in the polyhedron have the correct signs requires about 10 queries, whereas the MEB center reaches full sign correctness after only 2 queries. 
These empirical numbers of queries required to reach full sign correctness are substantially less than the sufficient worst-case bound (39 queries in this case) in Corollary~\ref{cor:sign correctness},
indicating that the theoretical guarantee is conservative and that CPM may  perform considerably better in practice.


\item[(b)]  
In Figures~\ref{fig:Ep1}(c)--(d), CPM displays better overall performance in terms of utility regret. 
In the case of estimates based on the AC/MEB centers, CPM achieves a smaller average regret in most rounds and reduces the regret to nearly zero faster, with a generally narrower range across 
the 100 customers. 
As for the worst-case utility regret,
 CPM performs generally better than POLY except 
 at round 4. 
The worst-case
regret under CPM decreases more rapidly and approaches zero earlier. 


\item[(c)] Figure~\ref{fig:Ep1}(e) shows that 
the diameter of the polyhedron under CPM decreases monotonically as the number of queries increases, and 
the 
curves are identical across 
all 100 customers. The ``deterministic'' nature of the behavior of the CPM is specific to the present experimental setting, primarily because all decision makers (DMs) share the same symmetric initial polyhedron. More importantly, the monotonic decrease in the polyhedral diameter under CPM is theoretically guaranteed by Theorem~\ref{theo:convergence_linear}.
Likewise, the diameter under POLY also decreases monotonically. However, when the number of queries increases from 5 to 20, the magnitude of reduction varies across the 100 customers, which
forms the blue shaded area. 
This variation highlights that CPM achieves a more consistent and reliable reduction in diameter than POLY.

\end{itemize}
In summary, CPM displays better overall performance under the specified metrics.






\subsection{Downstream decision: robust pricing under preference uncertainty}\label{sec:pricing}

We go one step further to consider a downstream robust personalized pricing problem 
\begin{subequations}\label{eq:robust_pricing}
\begin{align}
\max_{p\in[0,1]} \quad & p, \label{eq:robust_pricing_obj}\\
\text{s.t.}\quad &(\RR{\vv})^\top(p,(\x^{\text{new}})^\top)^\top \;\geqslant\; (\RR{\vv})^\top(p^{\text{base}},(\x^{\text{base}})^\top)^\top,
\qquad \forall 
{\vv}\in \mathcal{P}_n^k,
\label{eq:robust_pricing_con}
\end{align}
\end{subequations}
where, 
$\x^{\text{new}}=(x^{\text{new}}_{\text{battery}},x^{\text{new}}_{\text{comfort}},x^{\text{new}}_{\text{ANC}})=(0.8,0.2,0.6)$ denotes the non-price attribute vector of the new product to be priced,
$\x^{\text{base}}=(x^{\text{base}}_{\text{battery}},x^{\text{base}}_{\text{comfort}},x^{\text{base}}_{\text{ANC}})=(0.7,0.4,0.5)$ and $p^{\text{base}}=0.7$ are the attributes and the price of a reference product offered by the competitor.  
Model~\eqref{eq:robust_pricing} aims to determine the maximal price such that the DM would choose the new product over the base product, when the DM's preference is ambiguous and characterized by the polyhedron $\mathcal{P}_n^k$ generated through the CPM/POLY interaction process.
This model reflects the practical use of individual-level preference information in personalized pricing, for instance on some online retail platforms.
The solution method for \eqref{eq:robust_pricing} is discussed in Appendix~\ref{app:robust-pricing}.

We denote the optimized price with respect to $\mathcal{P}_n^k$ generated by CPM and POLY as $\hat{p}_n^k(\mathrm{CPM})$ and $\hat{p}_n^k(\mathrm{POLY})$, respectively.
As a benchmark, we set $\mathcal{P}_n^k$ to the singleton containing the true utility $u_n^*$ and derive the ground-truth optimized price $p_n^*$.
We compare CPM and POLY through the price gaps $p_n^*-\hat{p}_n^k(\mathrm{CPM})$ and $p_n^*-\hat{p}_n^k(\mathrm{POLY})$.
Figure~\ref{fig:pg1} reports how the price gap varies with the number of queries $k$, where the connected line represents the average over 100 DMs and the shaded region represents the range.

\begin{figure}[htbp]
    \centering
        \centering
        \includegraphics[width=0.6\textwidth]{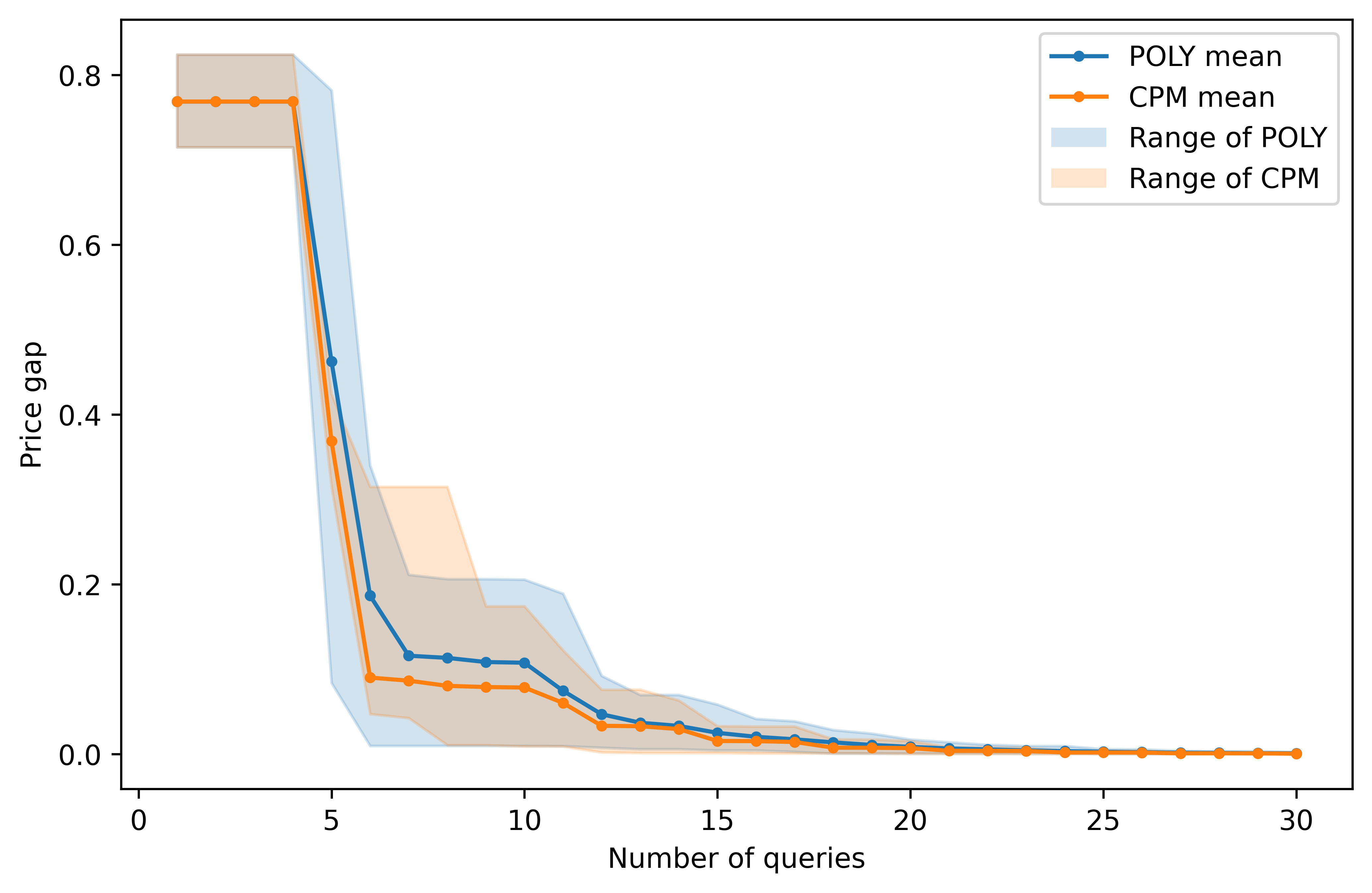}
    \caption{Average price gap across 100 DMs after each query\label{fig:pg1}}
    \end{figure}

From Figure~\ref{fig:pg1}, we observe that the price gaps of both CPM and POLY converge to zero as the number of queries increases. 
For the same query budget, CPM generally yields a smaller price gap than POLY, which is consistent with its stronger contraction of the ambiguity set shown in Figure~~\ref{fig:Ep1}(e) and its better worst-case utility regret shown in Figure~\ref{fig:Ep1}(d). 
Moreover, the range of the price gap under CPM is narrower than that under POLY, indicating more stable pricing performance across DMs.

\section{Concluding remarks}\label{sec:conclusion}

In this paper, we propose a coordinate-wise polyhedral method for eliciting a decision maker’s true utility function. While the approach addresses a gap in the literature by establishing theoretical, quantitative convergence guarantees for polyhedral methods in preference elicitation, several avenues for further research remain.
First, our method constructs cutting planes using minimum enclosing balls (MEBs), and subsequently designs queries in a sequential manner. An alternative approach would be to employ minimum-volume outer ellipsoids to generate cuts and guide query design, which may lead to more efficient contraction of the feasible polyhedron. We adopt MEBs primarily because they allow for a more tractable derivation of convergence rates, despite the extensive literature on ellipsoid-based methods for reducing polyhedral volume (see, e.g., \cite{damla2008linear}).
Second, to mitigate directional errors, our current framework restricts pairwise comparison lotteries to outcomes that lie within the existing set of breakpoints of the piecewise-linear utility function.
This restriction could be relaxed in future work by allowing outcomes outside the breakpoint set, for example, through the introduction of slack variables into the linear constraints that define the ambiguity set.
Third, we assume that no errors occur during the preference elicitation process. In practice, however, errors may arise from both the decision maker (DM) and the modeler. In particular, CPM involves queries with multiple outcomes, which may increase the likelihood of incorrect responses. Although statistical approaches can account for such response errors, they often lack the computational efficiency of deterministic polyhedral methods. A promising direction for future research is to combine the computational efficiency of CPM with probabilistic modeling techniques, thereby developing more robust and effective methods for preference elicitation under response noise or random utility models.
Fourth, the current CPM framework for nonlinear utility functions focuses on the univariate case. Extending it to the multivariate setting would be of considerable interest, as many practical decision-making problems involve multiple attributes. When the multivariate utility function is additively separable into univariate components, we anticipate that our approach can be extended with only minor modifications.
Finally, although this paper focuses on CPM for utility function elicitation, the framework could be generalized to other settings, such as the elicitation of risk measures, including spectral risk measures and utility-based risk measures. We leave these directions for future research.

\vspace{0.3cm}
\noindent
\textbf{Acknowledgments.}
This project is supported by RGC 
grant (No. 14204624), the National Key R\&D Program of China (No. 2022YFA1004000) and National Natural Science Foundation of China (No. 12371324).

\bibliographystyle{plainnat}
\bibliography{reference}

\appendix
\newpage
\section{Appendix: proofs}
\subsection{Proof of Lemma~\ref{lemma:MEBcenter}}
The result is perhaps well-known. We include a proof for completeness.
It suffices to show that for any $\x_{\text{out}}\notin\ \Omega$, there always exists a point $\x_{in}\in \Omega$ such that 
$\max_{\bm{y}\in\Omega}\|\x_{in}-\bm{y}\|_2<\max_{\bm{y}\in\Omega}\|\x_{\text{out}}-\bm{y}\|_2.$
Let
$\x^*$ be an orthogonal projection of $\x_{\text{out}}$ on $\Omega$.
Since $\Omega$ is a nonempty bounded closed convex set and $\x_{\text{out}}\notin\ \Omega$, we have $\|\x_{\text{out}}-\x^*\|_2>0$ and $\langle \x_{\text{out}}-\x^*,\bm{y}-\x^*\rangle\ <0,\ \forall \bm{y}\in\Omega$.
Thus
\begin{eqnarray*}
\|\x_{\text{out}}-\bm{y}\|^2_2
&=&\|\x_{\text{out}}-\x^*\|^2_2+ \|\x^*-\bm{y}\|^2_2
+2\langle\x_{\text{out}}-\x^*,\x^*-\bm{y}\rangle 
\\
 &>& \|\x_{\text{out}}-\x^*\|^2_2+\|\x^*-\bm{y}\|^2_2\\
 &>& \|\x^*-\bm{y}\|^2_2,
   \end{eqnarray*}
   which implies
%
$\max_{\bm{y}\in\Omega}\|\x^*-\bm{y}\|_2<\max_{\bm{y}\in\Omega}\|\x_{\text{out}}-\bm{y}\|_2.
$
The conclusion follows.
\hfill $\Box$

\subsection{Proof of Proposition~\ref{prop:C_radius-reduction}}
To ease the exposition, 
consider ball $\mathcal B(\bm0,r)$ with $\cc=0$. 
Consequently
$
\mathcal O=\{\vv\in\R^{N-1}_+:\|\vv\|_2\leqslant r
\}.
$
We seek the MEB of $\mathcal O$,
which reduces to solving the following minimization problem:
\bgeqn 
\label{eq:MEB-minimax}
c^*\in \mathop{\arg\min}_{c'\in\mathbb R^{N-1}}\sup_{v\in\mathcal O}\|v-c'\|_2.
\edeqn 
{By definition, $
r^*=\sup_{v\in\mathcal O}\|v-c^*\|_2$. Moreover,}  by permutation symmetry of $\mathcal O$,
the 
minimizer $\cc^*$ must have all coordinates equal, hence $\cc^*=s\cdot\bm{e}$ for some scalar $s\in\R_+$. Thus 
solving \eqref{eq:MEB-minimax} corresponds to solving
the following minimax optimization problem
\bgeqn 
\label{eq:MEB-minimax-1}
\min_{s\in\R_+} \Phi(s):=\max_{\vv\in\mathcal O} h_s(\vv):=\|\vv-s\bm{e}\|^2_2.
\edeqn 
	We first characterize $\Phi(s)$ for a fixed $s\geqslant0$. For any
	$\vv\in\mathcal O$ with $0<\|\vv\|_2<r$, define
$$
	\bm{w}:=\frac{r}{\|\vv\|_2}\vv,
	\qquad
	\lambda:=\frac{\|\vv\|_2}{r}\in(0,1).
$$
Then $\bm{w}\in\mathcal O$ with $\|\bm{w}\|_2=r$, and $\vv$ can be written as $\vv=\lambda \bm{w}+(1-\lambda)\bm0.$
Since $h_s(\cdot)$ is convex, we have
$$
h_s(\vv)\leqslant
\lambda h_s(\bm{w})+(1-\lambda)h_s(\bm0)
\leqslant
\max\{h_s(\bm{w}),h_s(\bm0)\}.
$$
Hence the maximum of $h_s$ over $\mathcal O$ is attained either at
$\bm0$ or on the spherical boundary
$\{\vv\in\mathbb R^{N-1}_+\mid\|\vv\|_2=r\}$.
On the spherical boundary, we have
$$
h_s(\vv)=\|\vv-s\bm e\|^2_2=r^2-2s\sum_{i=1}^{N-1} v_i+{(N-1)}s^2=r^2-2s\|\vv\|_1+({N-1})s^2.$$
Since $\vv\in\mathbb R^{N-1}_+$ and $\|\vv\|_2=r$, we have
	$\|\vv\|_1\geqslant \|\vv\|_2=r$, and equality holds when $\vv=r \bm{e}_j$ for $j\in\{1,\ldots,{N-1}\}$,
$\bm e_j=(0,\dots,0,\underset{j\text{-th}}{1},0,\dots,0)^\top$.
Therefore,
$\max\limits_{
\vv\in\mathbb R^{N-1}_+, \|\vv\|_2=r}h_s(\vv)=
	r^2-2sr+{(N-1)}s^2$.
On the other hand, $h_s(\bm0)=(N-1)s^2$.
Consequently, $\Phi(s)	=	\max\left\{
	(N-1)s^2,\,
	r^2-2sr+(N-1)s^2
	\right\}$.
	
If $0\leqslant s\leqslant r/2$, then
$r^2-2sr+(N-1)s^2\geqslant (N-1)s^2,$
and hence
$\Phi(s)=r^2-2sr+(N-1)s^2$.
This quadratic function is minimized at
$s^*=\frac{r}{N-1}$.
Since $N-1\ge2$, we have $s^*=\frac{r}{N-1}\leqslant r/2$. Therefore,
$
\Phi(s^*)
=
r^2-\frac{2r^2}{N-1}+\frac{r^2}{N-1}
=
r^2\left(1-\frac{1}{N-1}\right).
$

If $s\geqslant r/2$, then
$
\Phi(s)=(N-1)s^2,
$
whose minimum over $[\frac{r}{2},\infty)$ is $\frac{(N-1)r^2}{4}$. Since
$N-1\geqslant2$, $\frac{(N-1)r^2}{4}\geqslant r^2\left(1-\frac{1}{N-1}\right)$.

Therefore, the optimal solution of \eqref{eq:MEB-minimax-1} is $s^*=\frac{r}{N-1}$.
Thus, for $\mathcal O=\{\vv\in\R^{N-1}_+:\|\vv\|_2\leqslant r
\}$, its MEB center is
$
\cc^*=\frac{r}{N-1}\bm e,
$
and its radius is
$r^*=r\sqrt{1-\frac{1}{N-1}}.$
Summarizing the discussions above and transforming 
back to the original center yields the claimed MEB center and radius:
$
\cc^*=\cc+\frac{r}{N-1}\bm{e}, r^*=r\sqrt{1-\frac{1}{N-1}}.
$
The proof is complete. \hfill $\Box$

\subsection{Proof of Theorem~\ref{thm:find_query}}
By the definition of the separating hyperplane (see \eqref{eq:sep_def}), 
we only need to verify that 
$\x^A, \x^B$ 
defined by \eqref{eq:def_x} satisfies: (i) for every $\vv\in \mathcal{H}$,  $(\RR\vv)^\top(\x^A-\x^B)=0$; 
and (ii) For any $\vv^+\in {\cal H}_{+}$ and $\vv^-\in {\cal H}_{-}$, $[(\RR\vv^{+})^\top(\x^A-\x^B)][(\RR\vv^{-})^\top(\x^A-\x^B)] \leqslant0$.

We first prove (i). For any $\vv\in {\cal H}=\{\vv\in \R^{N-1}\mid\bal^\top(\vv-\bm{c})=0\}$,  we write $\vv=(v_1,\dots,v_{N-1})^\top$ and recall that
$\RR\vv=(v_1,\dots,v_{N-1},1-\sum_{i=1}^{N-1}v_i)^\top$. Then
\begin{subequations}
    \begin{eqnarray}
   ({\RR\vv})^\top(\x^A-\x^B) &=&\sum^{N-1}_{i=1}v_i(x^A_i-x^B_i)+\left(1-\sum_{i=1}^{N-1}v_i\right)(x^A_N-x^B_N) \\
   &\overset{\eqref{eq:query_B}}{=}&\sum^{N-1}_{i=1}v_i[(x^A_N-x^B_N)+\alpha_i]+\left(1-\sum_{i=1}^{N-1}v_i\right)(x^A_N-x^B_N) 
        \\
      &  =&\sum^{N-1}_{i=1}v_i\alpha_i+(x_N^A-x_N^B) \label{eq:query_diffrence}
        \\
       & \overset{\eqref{eq:query_A}}{=}&\bal^\top(\vv-\cc) \\
       & =&0 .
    \end{eqnarray}
\end{subequations}
Next,  we prove (ii). For any $\vv^+\in\mathcal{H}_+$, it satisfies $\bal^\top(\vv^+-\cc)\geqslant0$ and for any $\vv^-\in\mathcal{H}_-$, it satisfies $\bal^\top(\vv^--\cc)\leqslant0$. Consequently
\begin{subequations}
    \begin{align*}
        &[({\RR\vv^+})^\top(\x^A-\x^B)][({\RR\vv^-})^\top(\x^A-\x^B)]\\
      \overset{\eqref{eq:query_diffrence}}{=}  &\bigg[\sum^{N-1}_{i=1}v_i^+\alpha_i+(x_N^A-x_N^B)\bigg]\bigg[\sum^{N-1}_{i=1}v_i^-\alpha_i+(x_N^A-x_N^B)\bigg]\\
      \overset{\eqref{eq:query_A}}{=}   & \bigg(\sum^{N-1}_{i=1}v_i^+\alpha_i-\sum^{N-1}_{i=1}\alpha_ic_i\bigg)\bigg(\sum^{N-1}_{i=1}v_i^-\alpha_i-\sum^{N-1}_{i=1}\alpha_ic_i\bigg)\\
      =   & [\bal^\top(\vv^+-\cc)][\bal^\top(\vv^--\cc)]\leqslant0.
    \end{align*}
\end{subequations}
Combining (i) and (ii) shows that the hyperplane $\mathcal H$ is indeed the separating
hyperplane of the pairwise query $(A,B)$. 
The proof is complete. \hfill $\Box$

\subsection{Proof of Theorem~\ref{theo:convergence_linear}}
The inclusions follow directly from \eqref{eq:update_rule}.
Below, we prove \eqref{eq:convergence_linear}.
By assumption, $\vv^*\in\mathcal{P}^0$.
Since there is no response error in the elicitation process, 
then $\vv^*
\in \mathcal{P}^{k,i}$, for $i=1,\ldots,N-1$.
Denote the MEB of the polyhedron $\mathcal{P}^{k}$
by $\mathcal{B}^k$ with center $\cc^k$ and radius $r(\B^k)$.
We prove that
$r(\mathcal{B}^{k})$
decreases at least at a rate of $\sqrt{\frac{N-2}{N-1}}$.


Define
$
\mathcal{E}:=\left\{\bm\sigma=(\sigma_1,\cdots,\sigma_{N-1}) \;\middle|\; \sigma_i\in\{1,-1\},\;\text{for}\;  i=1,\cdots,N-1\right\}.
$
For each $\bm\sigma\in\mathcal E$, let
\bgeq
\mathcal O_{\bm\sigma}^{k}:=
\left\{\bm v\in\mathcal B^{k}
\;\middle|\;\sigma_s(v_s-c^{k}_s)\geqslant0,\ s=1,\cdots,N-1
\right\}. 
\edeq
Then $\mathcal O_{\bm\sigma}^{k}$
corresponds to one of the 
 $2^{N-1}$ congruent orthants of $\mathcal{B}^{k}$.
We denote the MEB of $\mathcal O_{\bm\sigma}^k$ by $\widehat{\mathcal{B}}_{\bm\sigma}^k$ with radius $r(\widehat{\mathcal{B}}_{\bm\sigma}^k)$.

At the $k$-th iteration, 
we conduct $N-1$ cuts 
on ${\cal P}^k$
with
mutually orthogonal hyperplanes in the $\R^{N-1}$.
All these cutting planes pass through the center $\bm{c}^k$ of $\mathcal{B}^k$.
By rotation, after these $N-1$ cuts, there exists a $\bm\sigma\in\mathcal{E}$, such that $\PP^{k+1}=\mathcal P^{k,N-1}\subseteq\mathcal O_{\bm\sigma}^k$. 
 By the definition MEB, 
the radius of the MEB of $\PP^{k+1}$ is less or equal to the radius of 
$O_{\bm\sigma}^k$, i.e.,
$r(\B^{k+1})\leqslant r(\widehat{\mathcal{B}}_{\bm\sigma}^k)$.
Moreover, by Proposition~\ref{prop:C_radius-reduction},
$r(\widehat{\mathcal{B}}_{\bm\sigma}^k)=\sqrt{\frac{N-2}{N-1}} r(\mathcal{B}^{k})$.
Hence,
$$
r(\mathcal{B}^{k+1}) \leqslant \sqrt{\frac{N-2}{N-1}} r(\mathcal{B}^{k})\leqslant  \cdots \leqslant \left(\sqrt{\tfrac{N-2}{N-1}}\right)^{\,k+1}r(\mathcal{B}^0).
$$
Moreover, since $\mathcal P^{k,i}\subset\mathcal B^k$, for $i=1,\cdots, N-1$,
then for every $\vv\in\mathcal P^{k,i}$,
$$
\| \vv - \vv^* \|_\infty\leqslant\|\vv-\vv^*\|_2 \leqslant\max\limits_{\vv',\vv''\in \mathcal{P}^{k,i}
}\|\vv'-\vv''\|_2
\leqslant 
2r(\mathcal B^k)\leqslant 2 r(\mathcal{B}^0)\left(\sqrt{\tfrac{N-2}{N-1}}\right)^{\,k},
$$
which is \eqref{eq:convergence_linear}.
The convergence $\mathcal P^{k}\to\{\vv^*\}$ 
follows from the inequality above since $\left(\sqrt{\tfrac{N-2}{N-1}}\right)^{\,k}\to 0$,
as $k\to\infty$.
The proof is complete. \hfill $\Box$

\subsection{Proof of Corollary~\ref{cor:sample-complex}}
By Theorem~\ref{theo:convergence_linear}, 
when $k\geqslant \left\lceil\frac{2[\log\epsilon-\log2r(\B^0)]}{\log(N-2)-\log(N-1)}\right\rceil$,
we have 
$2r({\cal B}^k) \leqslant 2r(\B^0)\left(\sqrt{\tfrac{N-2}{N-1}}\right)^{\,k}\leqslant\epsilon$.
The total number of queries
is $(N-1)\left\lceil\frac{2[\log\epsilon-\log2r(\B^0)]}{\log(N-2)-\log(N-1)}\right\rceil$
in that at each iteration $k$,
there are $(N-1)$-cuts (queries).
The conclusion follows as $\PP^{k,i} \subset \B^{k}$.
The proof is complete. \hfill $\Box$

\subsection{Proof of Corollary~\ref{cor:sign correctness}}
For $v_i,\; i=1,\ldots,N-1$, by \eqref{eq:convergence_linear},  $\|\vv-\vv^*\|_\infty\leqslant 2 r(\mathcal{B}^0)\left(\sqrt{\tfrac{N-2}{N-1}}\right)^{\,k}$. 
While for $v_N$, since $v_N=1-\sum_{i=1}^{N-1}v_i$, $\vert v_N-v_N^*\vert\leqslant(N-1)\|\vv-\vv^*\|_\infty\leqslant 2(N-1) r(\mathcal{B}^0)\left(\sqrt{\tfrac{N-2}{N-1}}\right)^{\,k}$.
For $k\geqslant \left\lceil\frac{2[\log v_{\min}^*-\log2(N-1)r(\B^0)]}{\log(N-2)-\log(N-1)}\right\rceil$, we have $2 (N-1)r(\mathcal{B}^0)\left(\sqrt{\tfrac{N-2}{N-1}}\right)^{\,k}\leqslant v_{\min}^*$.
If $v_i^*>0$, then 
$v_i \geqslant v_i^*-\|v-v^*\|_\infty \geqslant v_i^*-v_{\min}^* \geqslant 0.$
Similarly, if $v_i^*<0$, then
$v_i \leqslant v_i^*+\|v-v^*\|_\infty \leqslant v_i^*+v_{\min}^* \leqslant 0.$
The proof is complete. \hfill $\Box$
\subsection{Proof of Proposition~\ref{prop:find_query_G}}
It suffices to verify that $\G^A-\G^B$ satisfy \eqref{eq:sep_def_PLU}(a)-\eqref{eq:sep_def_PLU}(b). 
Let $\Delta\G=\G^A-\G^B$. By \eqref{eq:def_G},
\begin{align*}
(\RR\vv)^\top\Delta\G
&=\sum_{i=1}^{N-1} v_i\,\Delta G_i +\Big(1-\sum_{i=1}^{N-1}v_i\Big)\Delta G_N \\
&=\sum_{i=1}^{N-1} v_i(\Delta G_N+\gamma\alpha_i)
+\Delta G_N-\sum_{i=1}^{N-1} v_i\Delta G_N \\
&=\gamma\sum_{i=1}^{N-1} v_i\alpha_i+\Delta G_N \\
&=\gamma\,\bal^\top\vv-\gamma\,\bal^\top\bm{c}
=\gamma\,\bal^\top(\vv-\bm{c}).
\end{align*}
Thus, for any $\vv\in{\cal H}$, $(\RR\vv)^\top\Delta\G=\bal^\top(\vv-\bm{c})=0$ 
which verifies \eqref{eq:sep_def1_PLU}.
Moreover, for any $\vv^+\in{\cal H}_+$ and $\vv^-\in{\cal H}_-$, since
$\bal^\top(\vv^+-\bm{c})>0$ and $\bal^\top(\vv^--\bm{c})<0$, and $\gamma>0$, then 
$(\RR\vv^+)^\top\Delta\G>0$ and $(\RR\vv^-)^\top\Delta\G<0$, yielding
$
\big[(\RR\vv^+)^\top\Delta\G\big]\big[(\RR\vv^-)^\top\Delta\G\big]<0,
$
which verifies \eqref{eq:sep_def2_PLU}. 
\qed 


\subsection{Proof of Theorem~\ref{theo:solution}}\label{sec:proof_thm4.1}
Let $\rr=(r_0,\ldots,r_{N})\in\R^{N+1}$ be a given vector satisfying 
     \eqref{eq:so_ex_con}.
     Then
     $\g(r_0)=\bm{0}\in\R_+^N$, and 
     \bgeq 
     \g(r_i)=\Big( \underbrace{1, \dots, 1}_{i-1}, \frac{r_i - x_i}{x_{i+1} - x_i}, \underbrace{0, \dots, 0}_{N-i} \Big)^\top:=\Big(\underbrace{1, \dots, 1}_{i-1}, \tilde{r}_i, \underbrace{0, \dots, 0}_{N-i} \Big)^\top\in\R_+^N,\ i=1,\ldots,N,
     \edeq 
     where $\tilde{r}_i\in(0,1].$
Let ${\cal G}({\bm r})=(\g(r_1),\cdots,\g(r_N)) 
$. Then ${\cal G}({\bm r})$ is an $N\times N$
upper triangle matrix with diagonals $\tilde{r}_i$, for $i=1,\cdots,N$. Then \eqref{eq:def_G} can be 
rewritten concisely as
\bgeqn
\label{eq:G(p-q)=beta}
{\cal G}({\bm r})(\p^A-\p^B) =\gamma\bbe
\edeqn
where 
\bgeqn \label{eq:def_beta}
\bbe=\left(-\sum^{N-1}_{i=1}\alpha_ic_i+\alpha_1,\ldots,-\sum^{N-1}_{i=1}\alpha_ic_i+\alpha_{N-1},-\sum^{N-1}_{i=1}\alpha_ic_i\right)^\top,
\edeqn 
 The solution to the system of equations \eqref{eq:G(p-q)=beta}
  can be obtained from the last  equation with 
  \begin{subequations}
  \label{eq:p^A-p^B=beta}
  \bgeqn 
  p^A_N-p^B_N &=&\frac{\gamma\beta_N}{\tilde{r}_N},\\
p^A_i-p^B_i &=&\frac{\gamma}{\tilde{r}_i}\left[\beta_i-\sum_{s=i+1}^N(p^A_s-p^B_s)\right],\ \text{for}\; i=N-1,\cdots,1.
  \edeqn 
  \end{subequations}
 Note that $\gamma$ is not specified. To identify it, let
  \begin{subequations}\label{eq:L=p^A-p^B}
  \bgeqn 
  L_N &=&\frac{\beta_N}{\tilde{r}_N},\\
L_i &=&\frac{1}{\tilde{r}_i}\left[\beta_i-\sum_{s=i+1}^NL_s\right],\ \text{for}\; i=N-1,\cdots,1.
  \edeqn 
  \end{subequations}
Let ${\bm L}=(L_1,\cdots,L_N)$. 
We consider a specific solution to the system of equations \eqref{eq:p^A-p^B=beta}. We define
${\bm p}^A-{\bm p}^B = {\gamma}{\bm L}$.
Then we can use \eqref{eq:p^A-p^B=beta}
and the fact that $p^A_i, p_i^B\in[0,1]$ for $i=0,1\cdots,N$, and 
$\sum_{m=0}^N p_m^A=\sum_{m=0}^N p_m^B=1$ to
identify $\gamma=\frac{1}{\sum^{N}_{m=1}|L_m|}$. 
Let
\bgeqn \label{eq:L+_L-}
\left\{L^+_i:=\max\{L_i,0\}\right\}_{i=1}^N, \quad \text{and} 
\quad \left\{L^-_i:=\max\{-L_i,0\} \right\}_{i=1}^N.
\edeqn
We define
\bgeqn \label{eq:p^A}
p^A_i=\frac{L^+_i}{\sum_{m=1}^N|L_m|},\ \text{for} \; i=1,\ldots, N,
\quad \text{and} \quad
p^A_0=1-\sum_{i=1}^Np^A_i=\frac{\sum_{i=1}^NL^-_i}{\sum_{m=1}^N|L_m|},
\edeqn 
and 
\bgeqn \label{eq:p^B}
p^B_i=\frac{L^-_i}{\sum_{m=1}^N|L_m|},\;\text{for}\;  i=1,\ldots, 
N, 
\quad \text{and} \quad p^B_0=1-\sum_{i=1}^N p^B_i=\frac{\sum_{i=1}^NL^+_i}{\sum_{m=1}^N|L_m|}.
\edeqn 
It is easy to verify that $\p^A,\p^B$ defined in \eqref{eq:p^A} and \eqref{eq:p^B} and the $\gamma$ satisfy \eqref{eq:p^A-p^B=beta}. Then we obtain a solution $A=(\rr,\p^A)$ and $B=(\rr,\p^B)$. 
The proof is complete. \hfill $\Box$

\subsection{Proof of Proposition~\ref{prop:G_convex}}
For any
$(\rr^A,\p^A),\;(\rr^B,\p^B)\in\hat{\mathcal D}$,
define
$$
\tilde r_i^A:=\frac{r_i^A-x_i}{x_{i+1}-x_i}\in(0,1],
\qquad
\tilde r_i^B:=\frac{r_i^B-x_i}{x_{i+1}-x_i}\in(0,1],
\qquad i=1,\ldots,N.
$$
According to the definition of $\G$ (see \eqref{eq:G}),
$$
\G(\rr^A,\p^A)
=
{\left(
p_i^A\tilde r_i^A+
1-\sum_{m=0}^i p_m^A
\right)_{i=1}^N} \quad
\text{and}\quad
\G(\rr^B,\p^B)
=
{\left(
p_i^B\tilde r_i^B+
1-\sum_{m=0}^i p_m^B
\right)_{i=1}^N}. 
$$
Therefore, for any $\lambda\in[0,1]$,
\begin{equation}\label{eq:lambda_G_c}
\lambda\G(\rr^A,\p^A)+(1-\lambda)\G(\rr^B,\p^B)
=
{\left(
\lambda p_i^A\tilde r_i^A+(1-\lambda)p_i^B\tilde r_i^B
+
1-\sum_{m=0}^i\big(\lambda p_m^A+(1-\lambda)p_m^B\big)
\right)_{i=1}^N}.
\end{equation}

We now show that the right-hand side of \eqref{eq:lambda_G_c} belongs to
$\G(\hat{\mathcal D})$. Define
$
\hat p_i:=\lambda p_i^A+(1-\lambda)p_i^B,\; i=0,1,\ldots,N.
$
Let $\hat r_0:=x_1$ and, for $i=1,\ldots,N$, define
$$
\hat r_i:=
\begin{cases}
\dfrac{x_i+x_{i+1}}{2}, & \text{if } p_i^A=p_i^B=0,\\[1.2ex]
x_i+
\dfrac{\lambda p_i^A\tilde r_i^A+(1-\lambda)p_i^B\tilde r_i^B}{\hat p_i}
(x_{i+1}-x_i), & \text{otherwise}.
\end{cases}
$$
Denote $\hat\rr=(\hat r_0,\ldots,\hat r_N)$ and
$\hat\p=(\hat p_0,\ldots,\hat p_N)$. Then according to the definition of $\G$, we obtain
\bgeqn\label{eq:G_hat}
\G(\hat\rr,\hat\p)
=
{\left(
\hat p_i\frac{\hat r_i-x_i}{x_{i+1}-x_i}
+
1-\sum_{m=0}^i\hat p_m
\right)_{i=1}^N}.
\edeqn
Substitute $\hat\rr$ and $\hat\p$ into \eqref{eq:G_hat}, the result is exactly the right-hand side of \eqref{eq:lambda_G_c}. Hence
$
\G(\hat\rr,\hat\p)
=
\lambda\G(\rr^A,\p^A)+(1-\lambda)\G(\rr^B,\p^B).
$
It remains to verify that $(\hat\rr,\hat\p)\in\hat{\mathcal D}$. First,
$\hat p_i\in[0,1]$ for all $i=0,\ldots,N$, and
$
\sum_{i=0}^N\hat p_i
=
\lambda\sum_{i=0}^N p_i^A
+
(1-\lambda)\sum_{i=0}^N p_i^B
=
1.
$
Thus $\hat\p$ is a well-defined probability vector.

Next, we verify that $\hat\rr$ satisfies \eqref{eq:so_ex_con}. If
$p_i^A=p_i^B=0$, then $\hat r_i=(x_i+x_{i+1})/2\in(x_i,x_{i+1}]$.
Otherwise, $\hat p_i>0$. Since
$\tilde r_i^A,\tilde r_i^B\in(0,1]$, we have
$$
0<
\frac{\lambda p_i^A\tilde r_i^A+(1-\lambda)p_i^B\tilde r_i^B}{\hat p_i}
\leqslant
\frac{\lambda p_i^A+(1-\lambda)p_i^B}{\hat p_i}
=1.
$$
Therefore $\hat r_i\in(x_i,x_{i+1}]$ for every $i=1,\ldots,N$.
Combining the above arguments gives $(\hat\rr,\hat\p)\in\hat{\mathcal D}$.

Above all, $
\lambda\G(\rr^A,\p^A)+(1-\lambda)\G(\rr^B,\p^B)
\in \G(\hat{\mathcal D}),
$
hence $\G(\hat{\mathcal D})$ is convex.
The proof is complete. \hfill $\Box$

\subsection{Proof of Theorem~\ref{thm:conv_univariate_metric}}
Fix $k\geqslant 0$ and $i\in\{1,\ldots,N-1\}$, and take any $u\in\mathcal U_{\mathcal{X}}^{k,i}$.
Then there exists $\vv\in\mathcal P^{k,i}$ such that $u=u(\cdot;\vv)=(\RR \vv)^\top g(\cdot)$. 
By Definition~\ref{def:kan_kolmo},
$$
\dd_I(u^*,u)=(\overline{x}-\underline{x})\sup_{z\in[\overline{x},\underline{x}]}|u^*(z)-u(z)|=\delta\|u^*-u\|_\infty,
$$
where $\delta:=\overline{x}-\underline{x}$ and $\dd_K(u^*,u)\leqslant\dd_I(u^*,u)$.

\underline{Bounding $\|u^*-u\|_\infty$.}
For a fixed piecewise-linear utility function $u(x)$ with increment vector $\vv$, let $\Delta \vv:=\vv^*-\vv$.
For any $x\in(x_\ell,x_{\ell+1}]$ with $\ell\in\{1,\ldots,N\}$, we have
$$
u(x)=\sum_{j=1}^{\ell-1} v_j + \theta v_\ell,
\quad
\text{where}\;\theta:=\frac{x-x_\ell}{x_{\ell+1}-x_\ell}\in[0,1].
$$
For $\ell\leqslant N-1$, this expression involves only $(v_1,\ldots,v_{\ell})$; for $\ell=N$, the last increment
is determined by the normalization $\bm e^\top (\RR \vv)=1$, hence the same bound below holds.
Therefore, in all cases,
$$
|u^*(x)-u(x)|
\leqslant \sum_{j=1}^{N-1} |\Delta v_j|
= \|\Delta \vv\|_1.
$$
Hence $\|u^*-u\|_\infty \leqslant \|\vv^*-\vv\|_1 \leqslant \sqrt{N-1}\,\|\vv^*-\vv\|_2$.

Algorithm~\ref{alg:VCPM_PLU} performs the same MEB-centered coordinate-wise orthogonal cuts in the
$(N-1)$-dimensional increment space as Algorithm~1 does in the multivariate linear case.
Therefore, the radius contraction argument used in Theorem~\ref{theo:convergence_linear} applies verbatim to $\{\mathcal P^{k,i}\}$:
for every $\vv\in\mathcal P^{k,i}$,
$$
\sqrt{N-1}\,\|\vv^*-\vv\|_2 \leqslant 2\sqrt{N-1}\,r(\mathcal B^k)\leqslant  2 \sqrt{N-1} r(\mathcal B^0)\Big(\sqrt{\tfrac{N-2}{N-1}}\Big)^{\,k}.
$$
This completes the proof of \eqref{eq:conv_metric_bound}.
 \hfill $\Box$

\subsection{Proof of Proposition~\ref{lem:pla_error}}
Error bound on piecewise-linear approximation of a Lipschitz-continuous utility function is well documented, see e.g.~\cite{guo2024utility}. However, here we derive the bound in a slightly different manner which leads to a sharper bound. 
Consider any interval 
of two consecutive breakpoints, e.g., 
$[x_i^k,x_{i+1}^k]$. Let 
$z\in[x_i^k,x_{i+1}^k]$ be any point in the interval. Then we can write $z$ as
$z=(1-\lambda)x_i^k+\lambda x_{i+1}^k$ 
for some $\lambda\in[0,1]$. 
Since $u$ is nondecreasing and $L$-Lipschitz, and 
$(1-\lambda)a-\lambda b\leqslant \max\{(1-\lambda)a, \lambda b\}
$ for any $a,b\in\R_+$,
then the quantity 
$
 \big\vert u(z)-(\Tk u)(z)\big\vert
 =\big\vert u(z)-\big((1-\lambda)u(x_i^k)+\lambda u(x_{i+1}^k)\big)\big\vert
$
is bounded by
$
 \max\big\{(1-\lambda)|u(z)-u(x_i^k)|,\ \lambda|u(x_{i+1}^k)-u(z)|\big\}
 \leqslant \max\big\{(1-\lambda)L(z-x_i^k),\ \lambda L(x_{i+1}^k-z)\big\}.
$
Observe that
 $z-x_i^k=\lambda(x_{i+1}^k-x_i^k)$ and $
x_{i+1}^k-z=(1-\lambda)(x_{i+1}^k-x_i^k)$. Thus,
$
\max\big\{(1-\lambda)L(z-x_i^k),\ \lambda L(x_{i+1}^k-z)\big\}
=
L\lambda(1-\lambda)(x_{i+1}^k-x_i^k).
$
Since
 \(\lambda(1-\lambda)\leqslant 1/4\) for all \(\lambda\in[0,1]\),
 then the above quantity is bounded by
$
\frac{L(x_{i+1}^k-x_i^k)}{4}.
$
Taking the maximum over all intervals yields \eqref{eq:PLA-Lip}.
The proof is complete. \hfill $\Box$

\subsection{Proof of Lemma~\ref{prop:lift_projection}}
Part (i). For any $\bm w\in \PP^{k+1,0}$, 
let $v_i=w_i$ for $i\neq j_k$ and $v_{j_k}=w_{j_k}^-+w_{j_k}^+$. 
Then $\vv\in\Pkp$. This shows $\pi_k(\PP^{k+1,0})\subseteq \Pkp$.
Conversely, let
$\vv\in\Pkp$ and choose $a=v_{j_k}/2$. According to the interval-wise Lipschitz bounds on $\XX^k$, we
have $0\leqslant v_{j_k}\leqslant Lh_k$, and therefore
$$
 0\leqslant a=\frac{v_{j_k}}{2}\leqslant \frac{Lh_k}{2},
 \qquad
 0\leqslant v_{j_k}-a=\frac{v_{j_k}}{2}\leqslant \frac{Lh_k}{2}.
$$
Hence the corresponding split vector belongs to $\PP^{k+1,0}$.
Thus
$\pi_k(\PP^{k+1,0})\supseteq\Pkp$.

Part (ii). Let $u^*\in\Ukp$. Denote $x_{\text{new}}^k=\frac{x_{j_k+1}^k+x_{j_k}^k}{2}$ the new point added into $\XX^{k+1}$. Let
$
 a^*:=u^*(x_{\text{new}}^k)-u^*(x_{j_k}^k),
 \
 b^*:=u^*(x_{j_k+1}^k)-u^*(x_{\text{new}}^k).
$
Then $a^*,b^*\geqslant 0$, $a^*+b^*=\big(\Pi_k(u^*)\big)_{j_k}$. By the $L$-Lipschitzness,
$
 a^*\leqslant \frac{Lh_k}{2}, \ b^*\leqslant \frac{Lh_k}{2}.
$
Therefore $\Pi_{k+1}(u^*)\in \PP^{k+1,0}$, i.e., $u^*\in\UU^{k+1}$. 

Part (iii). If $u\in\UU^{k+1}$, then $\Pi_{k+1}(u)\in \PP^{k+1,0}$. By part (i),
$
 \Pi_k(u)=\pi_k(\Pi_{k+1}(u))\in\Pkp.
$
Hence $u\in\Ukp$.
The proof is complete. \hfill $\Box$

\subsection{Proof of Theorem~\ref{thm:lift_bound}}
We first prove 
$
 r(\B^{k+1,0})
 \leqslant
 \sqrt{\Big(r(\B^{k,+})\Big)^2+\eta_k^2},
$
where $\B^{k+1,0}$ is the MEB of $ \PP^{k+1,0} $ and $\B^{k,+}$ is the MEB of $\PP^{k,+}$.
Let $\cc=(c_1,\ldots,c_{\Nk-1})$ be the center of $\B^{k,+}$. By the definition of MEB,  
$\sup_{\vv\in\Pkp}\|\vv-\cc\|_2\leqslant r(\B^{k,+})$. We proceed with the rest of the proof in two cases according to the location of the 
widest interval (see \eqref{eq-jk}).

\underline{Case 1.} $j_k\neq N_k$.
In this case,
$$
\PP^{k+1,0}	:=
\left\{
({v}_1,\dots,{v}_{j_k-1},\,a,\,{v}_{j_k}-a,\,{v}_{j_k+1},\dots,{v}_{N_k-1})^\top
	\in\mathbb{R}^{N_k}
	\;\middle|\;
\begin{array}{c}
\
	0\leqslant a\leqslant v_{j_k}, \\
a\leqslant \frac{L\Delta_{k}}{2},\
	v_{j_k}-a\leqslant\frac{L\Delta_{k}}{2}
\end{array}
\right\}.
$$
We consider a fixed point
$
 \bm y:=(c_1,\ldots,c_{j_k-1},c_{j_k}/2,c_{j_k}/2,c_{j_k+1},\ldots,c_{\Nk-1})^\top \in\PP^{k+1,0}\subset\R^{\Nk}.
$
For any $\bm{z} \in\PP^{k+1,0}$,
there exists 
$\vv\in\Pkp$ and 
$
a\in [0,
v_{j_k}]$ 
such that
$
 \bm{z}=(v_1,\ldots,v_{j_k-1},a,v_{j_k}-a,v_{j_k+1},\ldots,v_{\Nk-1})^\top$. 
Observe that
\begin{align*}
 \|\bm{z}-\bm y\|_2^2
 &=\sum_{i\ne j_k}(v_i-c_i)^2+
 \left(a-\frac{c_{j_k}}{2}\right)^2+
 \left(v_{j_k}-a-\frac{c_{j_k}}{2}\right)^2 \\
 &=\sum_{i\ne j_k}(v_i-c_i)^2+
 a^2+(v_{j_k}-a)^2-v_{j_k}c_{j_k}\\
 &=\sum_{i\ne j_k}(v_i-c_i)^2+
 \frac{(v_{j_k}-c_{j_k})^2}{2}+2\left(\frac{v_{j_k}}{2}-a\right)^2.
\end{align*}
Since $0\leqslant a\leqslant v_{j_k}$, $|\frac{v_{j_k}}{2}-a|\leqslant \frac{v_{j_k}}{2}$, then  $2(\frac{v_{j_k}}{2}-a)^2\leqslant \frac{v_{j_k}^2}{2}$, and consequently we obtain
$$
 \|\bm z-\bm y\|_2^2
 \leqslant
 \sum_{i\ne j_k}(v_i-c_i)^2+(v_{j_k}-c_{j_k})^2+\frac{v_{j_k}^2}{2}
 =\|\vv-\cc\|_2^2+\frac{v_{j_k}^2}{2}.
$$
By Lemma~\ref{prop:lift_projection}(i), $\pi_k(\PP^{k+1,0})=\Pkp$. 
Taking the supremum over $\bm z\in\PP^{k+1,0}$ on the lhs of the inequality above
and accordingly 
$\vv$ over $\PP^{k,+}$ at the rhs of the inequality, we obtain
\bgeqn 
\sup_{\bm z\in{\PP}^{k+1,0}}
 \|\bm z-\bm y\|_2^2
 \leqslant \sup_{\vv\in\PP^{k,+}}\|\vv-\cc\|_2^2+\sup_{j_k\in\{1,\cdots,N_k-1\}}\frac{v_{j_k}^2}{2}
\leqslant \Big(r(\B^{k,+})\Big)^2+\frac{\eta_k^2}{2}.
\edeqn 
On the other hand, 
since $\bm y\in{\PP}^{k+1,0}$, then 
$\sup_{\bm z\in{\PP}^{k+1,0}}
 \|\bm z-\bm y\|_2\geqslant r({\B}^{k+1,0})$.
 Therefore,
\bgeqn 
 r(\B^{k+1,0})
 \leqslant \sup_{\bm z\in{\PP}^{k+1,0}}
 \|\bm z-\bm y\|_2
 \leqslant
 \sqrt{\Big(r(\B^{k,+})\Big)^2+\frac{\eta_k^2}{2}}\leqslant r(\B^{k,+})+\frac{\eta_k}{\sqrt 2}.
\edeqn 
The last inequality comes from the fact that $\sqrt{a^2+b^2}\leqslant a+b$ when
$a,b\geqslant 0$.

\underline{Case 2}. $j_k= N_k$.
In this case, 
$$
\PP^{k+1,0}	:=
\left\{
({v}_1,\dots,{v}_{N_k-1},a)^\top
	\in\mathbb{R}^{N_k}
	\;\middle|\;
\begin{array}{c}
\
	0\leqslant a\leqslant v_{N_k},\;
a\leqslant \frac{L(x_{N_k+1}-x_{N_k})}{2}
\end{array}
\right\}.
$$
We consider a fixed point
$
 \bm y:=(c_1,\ldots,,c_{\Nk-1},c_{N_k}/2)\in\PP^{k+1,0}\subseteq\R^{\Nk}
$, where $c_{N_k}=1-\sum_{i=1}^{N_k-1}c_i$.
For any $
 \bm{z}\in\PP^{k+1,0}$, there exists $\vv\in\PP^{k,+}$ and $a\in[0,v_{N_k}]$ such that
 $
 \bm z=(v_1,\ldots,v_{\Nk-1},a)\in\PP^{k+1,0}$, 
 where $v_{N_k}=1-\sum_{i=1}^{N_k-1}v_i$. 
Hence,
\begin{align*}
 \|\bm{z}-\bm y\|_2^2
 &=\sum_{i=1}^{N_k-1}(v_i-c_i)^2+
 \left(a-\frac{c_{N_k}}{2}\right)^2=
 { \|\vv-\cc\|_2^2
 }
 +\left(a-\frac{c_{N_k}}{2}\right)^2.
\end{align*}
Since $0\leqslant a\leqslant v_{N_k}\leqslant\eta_k$ and $0\leqslant c_{N_k}\leqslant\eta_k$ due to the fact that $\cc\in \PP^{k,+}$, we have $|a-\frac{c_{N_k}}{2}|\leqslant \eta_k.$
Therefore
$
 \|\bm z-\bm y\|_2^2
 \leqslant
 \|\vv-\cc\|_2^2+\eta_k^2
 \leqslant
 \|\vv-\cc\|_2^2+
 \eta_k^2.
$
Taking the supremum over $\bm z\in\PP^{k+1,0}$ gives
$$
\sup_{\bm z\in{\PP}^{k+1,0}}
 \|\bm z-\bm y\|_2^2
 \leqslant \sup_{\vv\in\PP^{k,+}}\|\vv-\cc\|_2^2+
 \eta_k^2
 \leqslant \Big(r(\B^{k,+})\Big)^2+\eta_k^2.
$$
Since $\bm y\in{\PP}^{k+1,0}$, 
we have 
$$
 r(\B^{k+1,0})
 \leqslant \sup_{\bm z\in{\PP}^{k+1,0}}
 \|\bm z-\bm y\|_2\leqslant
 \sqrt{\Big(r(\B^{k,+})\Big)^2+\eta_k^2}
 \leqslant r(\B^{k,+})+\eta_k.
$$
Above all, we always have
$
r(\B^{k+1,0})\leqslant r(\B^{k,+})+\eta_k.
$
Finally, since $\Pkp$ retains the interval-wise Lipschitz bounds on $\XX^k$,
$$
 0\leqslant \eta_{k}\leqslant L\max_{j_k\in{1,\cdots,N_k}}(x_{j_k+1}^k-x_{j_k}^k)=L\hk,
$$
which completes the proof.
\hfill $\Box$

\subsection{Proof of Proposition~\ref{prop:h_k}}
For $k=0$, $h_0 = \max_{1\leqslant i\leqslant N_0}(x_{i+1}^0-x_i^0)$.
The interval $[x_1^0,x_{N_0+1}^0]=[\underline{x},\overline{x}]$ contains 
 $N_0$ mutually exclusive sub-intervals.  
Therefore, it takes  at most $N_0$ rounds of updating to 
reduce the mesh size $h_0$ to less or equal to $h_0/2$ (when the breakpoints are evenly spread), i.e., $h_{N_0}\leqslant h_0/2$.
The number of intervals in the worst case (after $N_0$ rounds of updating) will increase from $N_0$
to $2N_0$.
Repeating the same argument,
it will take $2N_0$ rounds in the worst case to 
reduce the mesh size to less or equal to 
$\frac{h_0}{2^2}$, i.e., $h_{3N_0}\leqslant h_0/4$, 
in which case 
the total number of breakpoints is $4N_0$.
Continuing the process, we deduce by induction that
$h_{(2^r-1)N_0}
= h_{(N_0+2N_0+\cdots+2^{r-1}N_0)} 
\leqslant\frac{1}{2^{r}}h_0,\; r=1,2,\ldots
$
Let $K_r:=(2^r-1)N_0$. 
Then  we can rewrite the inequality above 
as $h_{K_r}\leqslant\frac{N_0h_0}{{K_r}+N_0}$.
Finally, for any $k\geqslant 0$, let
$
r:=\left\lfloor \log_2\!\left(1+\frac{k}{N_0}\right)\right\rfloor.
$
Then 
we have 
$
(2^r-1)N_0\leqslant k.
$
Since $h_k$ is non-increasing in $k$, then
$
h_k\leqslant h_{(2^r-1)N_0}\leqslant \frac{h_0}{2^r}.
$
Thus, for the specified $r$, we have 
$
h_k
\leqslant
\frac{h_0}{2^r}
\leqslant
\frac{2h_0}{1+k/N_0}
=
\frac{2N_0h_0}{k+N_0}
$ for any $k\ge0$.
\qed 


\subsection{Proof of Theorem~\ref{thm:utility_space_convergence}}
For every $u\in \UU^k$, by the proof of Theorem~\ref{thm:conv_univariate_metric}, we have \bgeqn\label{eq:d_k+d_I}
\dd_K(u^*,u)\ \leqslant\ \dd_I(u^*,u)=(\overline{x}-\underline{x})\|u-u^*\|_\infty.
\edeqn
In what follows, we derive 
a bound for $\sup_{u\in {\cal U}^k}\|u-u^*\|_\infty$.
Fix $u\in \UU^k$ and let $T_k u$ be its piecewise-linear approximation defined as in \eqref{eq:T_k-PLA}.
        By  triangle inequality,
\bgeqn 
\label{eq:thm-proof-triangle}
\|u-u^*\|_\infty\leqslant	\|u-T_k u\|_\infty	+	\|T_k u-T_k u^*\|_\infty	+	\|u^*-T_k u^*\|_\infty .
\edeqn 
Proposition~\ref{lem:pla_error} guarantees that 
$\|u-T_k u\|_\infty\leqslant \frac{Lh_k}{4}$ and $\|u^*-T_k u^*\|_\infty\leqslant \frac{Lh_k}{4}.$ Here we prove
\begin{eqnarray}\label{eq:Tk_bound}
\|T_k u-T_k u^*\|_\infty
\leqslant 
\|\Pi_k(u)-\Pi_k(u^*)\|_1
\leqslant 	\sqrt{N_k-1}\,\|\Pi_k(u)-\Pi_k(u^*)\|_2.
\end{eqnarray}
The second inequality is well known, so we only prove the first.
Let
$\Delta:=\Pi_k(u)-\Pi_k(u^*)=(\Delta_1,\ldots,\Delta_{N_k-1})^\top\in\mathbb{R}^{N_k-1}.$
By the definition of recovery mapping (see \eqref{map:recover}), 
$\RR\Pi_k(u)-\RR\Pi_k(u^*)=
\Bigl(\Delta_1,\ldots,\Delta_{N_k-1},-\sum_{i=1}^{N_k-1}\Delta_i\Bigr)^\top
\in\mathbb{R}^{N_k}.
$
For any fixed $x\in[\underline{x},\overline{x}]$, there exists $j\in\{1,\ldots,N_k\}$ such that
$x\in[x_j^k,x_{j+1}^k]$.
Let
$
\theta:=\frac{x-x_j^k}{x_{j+1}^k-x_j^k}\in[0,1].
$
We go through two cases.

\underline{Case 1}. $j\neq N_k$. By the definition of $\g^k(\cdot)$,
$
\g^k(x)=(1,\ldots,1,\theta,0,\ldots,0)^\top\in\R^{N_k}
$. Therefore,
\begin{eqnarray*}
&\Big\vert(T_k u-T_k u^*)(x)\Big\vert
&=\Bigg\vert\Bigl(\RR\Pi_k(u)-\RR\Pi_k(u^*)\Bigr)^\top \g^k(x)\Bigg\vert
=\Bigg\vert\sum_{i=1}^{j-1}\Delta_i+\theta \Delta_j\Bigg\vert\\
&&\leqslant
\sum_{i=1}^{j-1}|\Delta_i|+\theta |\Delta_j|
\leqslant
\sum_{i=1}^{N_k-1}|\Delta_i|
=
\|\Delta\|_1.
\end{eqnarray*}

\underline{Case 2}. $j=N_k$.
Since $x$ is in the last interval,
$
\g^k(x)=(1,\ldots,1,\theta)^\top\in \R^{N_k},
$
then
\begin{eqnarray*}
&\Big\vert(T_k u-T_k u^*)(x)\Big\vert
&=
\Bigg\vert\sum_{i=1}^{N_k-1}\Delta_i
+\theta\Bigl(-\sum_{i=1}^{N_k-1}\Delta_i\Bigr)\Bigg\vert 
=(1-\theta)\Bigg\vert\sum_{i=1}^{N_k-1}\Delta_i\Bigg\vert\\
& &\leqslant
(1-\theta)\sum_{i=1}^{N_k-1}|\Delta_i|
\leqslant
\|\Delta\|_1.
\end{eqnarray*}
Summarizing the two cases and 
since $x\in[\underline{x},\overline{x}]$ is taken arbitrarily, we obtain the first inequality in \eqref{eq:Tk_bound}. 
A combination of \eqref{eq:Tk_bound} with \eqref{eq:thm-proof-triangle} leads to 
\bgeqn 
\label{eq:Thm5.2-intermediate-eqstmate}
\|u-u^*\|_\infty\leqslant \frac{Lh_k}{4}+\sqrt{N_k-1}\,\|\Pi_k(u)-\Pi_k(u^*)\|_2+\frac{Lh_k}{4}.
\edeqn 
Since $\Pi_k(u),\Pi_k(u^*)\in \PP^{k,0}$, then 
$	\|\Pi_k(u)-\Pi_k(u^*)\|_2	\leqslant	\diam(\PP^{k,0})	\leqslant	2\,r(\B^{k,0}).$
Substituting 
the inequality into \eqref{eq:Thm5.2-intermediate-eqstmate}, we obtain  
$\sup_{u\in {\cal U}^k}\|u-u^*\|_\infty
\leqslant	\frac{L h_k}{2}+2\sqrt{N_k-1}\,r(\B^{k,0})$.

Next, we show  the last inequality of \eqref{eq:one-step_delta_u}. By the recursive radius relationship established in Theorem \ref{thm:lift_bound} and the stopping rule of the inner CPM major iterations, we conclude 
$	
r(\B^{k,0})	\leqslant	\epsilon\,r(\B^{k-1,0})+Lh_{k-1}
,	 \text{for}\; k\geqslant 1.$
By induction from this inequality, we obtain 
$
 r(\B^{k,0})	\leqslant	\epsilon^k r(\B^{0,0})+L
 \sum_{i=0}^{k-1}\epsilon^i h_{k-1-i}.
$
Substituting this bound into the second inequality in \eqref{eq:one-step_delta_u} and using inequality  \eqref{eq:d_k+d_I}
immediately give rise to the second inequality of \eqref{eq:one-step_delta_u}.
	
Finally, we prove 
$\sup_{u\in\UU^k}\|u-u^*\|_\infty \to 0$ as $k\to\infty$.
By Proposition~\ref{prop:h_k}, $h_k={O}(\frac{1}{k})$. By 
the updating rule of breakpoint set (see Step 2 in Algorithm~\ref{alg:VCPM_CU}), $N_k=N_0+k$.
Hence, $\sqrt{N_k-1}\hk=O(\frac{1}{\sqrt{k}})\rightarrow0$ and $\sqrt{N_k-1}\epsilon^kr(\B^{0,0})\to 0,$ since $\epsilon^k$ decays to zero exponentially.
	Therefore,
$$	\sup_{u\in\UU^k}\|u-u^*\|_\infty
\overset{\eqref{eq:one-step_delta_u}}{\leqslant} \frac{Lh_k}{2}	+	2\sqrt{N_k-1}
		\left[\epsilon^k r(\B^{0,0})+L
        \sum_{i=0}^{k-1}\epsilon^i h_{k-1-i}
		\right]=O(\frac{1}{\sqrt{k}}) 
\to 0,	\qquad \text{as}\; k\to\infty.$$
Since the updates of  breakpoints in Algorithm~\ref{alg:VCPM_CU} do not change the endpoints $\underline{x}$ and $\overline{x}$, $\delta$ remains a fixed constant. Therefore, $\delta \sup_{u\in\UU^k}\|u-u^*\|_\infty\rightarrow 0$ as $k\rightarrow \infty$. This completes the proof.
 \hfill $\Box$

\section{Comparison between 
MEB center in CPM and  
AC in POLY}
The polyhedral method by~\cite{toubia2004polyhedral}~employs the analytic center as an estimate of the decision maker’s  partworth vector under the collected preference information. By contrast, CPM adopts the MEB center as an estimate.
The advantage of using the MEB center lies in its robustness to extreme
locations of the unknown true partworth vector. Unlike the analytic center,
which is determined by the slacks to the defining halfspaces, the MEB center
directly controls the worst-case Euclidean distance to feasible points in the
polyhedron. Thus, when $\vv^*$ lies at or near a vertex, the analytic center may
be far from $\vv^*$, whereas the MEB center remains minimax-optimal over the
entire ambiguity set. The next proposition formalizes this property, which states that of all candidate points in $\mathbb{R}^{N-1}$, the MEB center is the point at which the maximum distance 
from the point to  any point in
$\mathcal{P}$ is minimized.


\begin{proposition}[Worst-case optimality of the MEB center]\label{prop:wc-MEB}
Let $\mathcal{P} \subset \mathbb{R}^{N-1}$ be a bounded polyhedron, and  $\mathcal{B}(\cc,r)$ be the MEB of ${\cal P}$. 
Then
$
\cc = \mathop{\arg\min}_{\vv' \in \mathbb{R}^{N-1}} \max_{\vv \in \mathcal{P}} \|\vv-\vv'\|_2.
$
\end{proposition}

\begin{proof}
By Definition \ref{def:circle}, the MEB of $\PP$ is the 
ball 
containing 
$\PP$ with minimal radius:
$$
(\cc,r) = \mathop{\arg\min}_{\vv'\in \mathbb{R}^{N-1}, r'\geqslant 0} r' \quad \text{s.t.} \quad \|\vv-\vv'\|_2 \leqslant r', \forall \vv \in \PP.
$$
By construction, $r = \max_{\vv \in \PP} \|\vv - \cc\|_2$. Suppose for the sake of a contradiction that there exists a point $\cc' \neq \cc$ such that 
$\max_{\vv \in \PP} \|\vv - \cc'\|_2 < r$. Then ${\cal B}(c', \max_{\vv\in \PP}\|\vv-\cc'\|_2)$ would be a smaller enclosing ball than $\B(\cc,r)$, contradicting the minimality of $\B(\cc,r)$. Hence, no point can have a smaller maximal distance to $\PP$ than $\cc$.
\end{proof}
Table \ref{tb:Maxdis} 
gives a comparative list of the distances for the case when ${\cal P}$  is an irregular polygon with 3 to 8 edges. And Figure~\ref{fig:polygons} illustrates the property in Proposition~\ref{prop:wc-MEB} geometrically.
\begin{figure}[htbp]
	\begin{subfigure}[t]{0.32\textwidth}
		\centering		\includegraphics[width=0.8\textwidth]{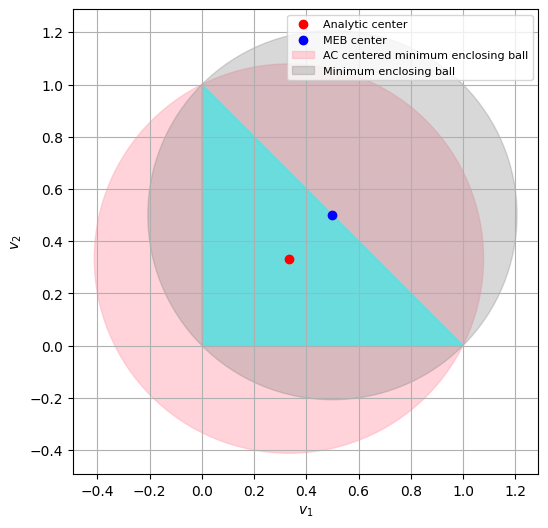}
          \caption{Triangle}
	\end{subfigure}
	\begin{subfigure}[t]{0.32\textwidth}
		\centering
		\includegraphics[width=0.75\textwidth]{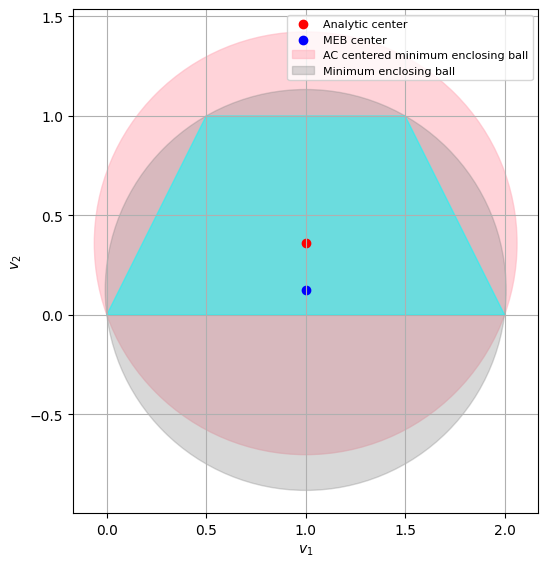}
          \caption{Tetragon}
	\end{subfigure}
    \begin{subfigure}[t]{0.32\textwidth}
		\centering
\includegraphics[width=0.8\textwidth]{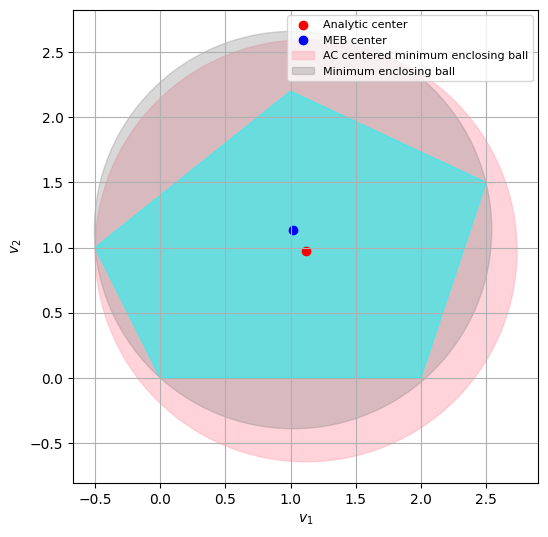}
    \caption{Pentagon}
	\end{subfigure}
    \begin{subfigure}[t]{0.32\textwidth}
		\centering		\includegraphics[width=0.8\textwidth]{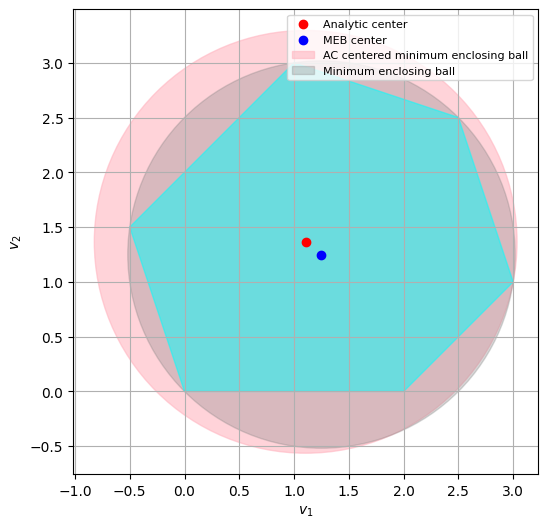}
    \caption{Hexagon}
	\end{subfigure}
	\begin{subfigure}[t]{0.32\textwidth}
		\centering
\includegraphics[width=0.8\textwidth]{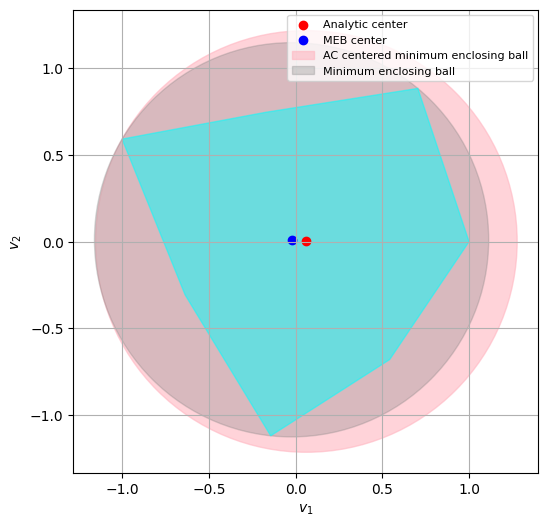}
    \caption{Heptagon}
	\end{subfigure}
    \begin{subfigure}[t]{0.32\textwidth}
		\centering
\includegraphics[width=0.8\textwidth]{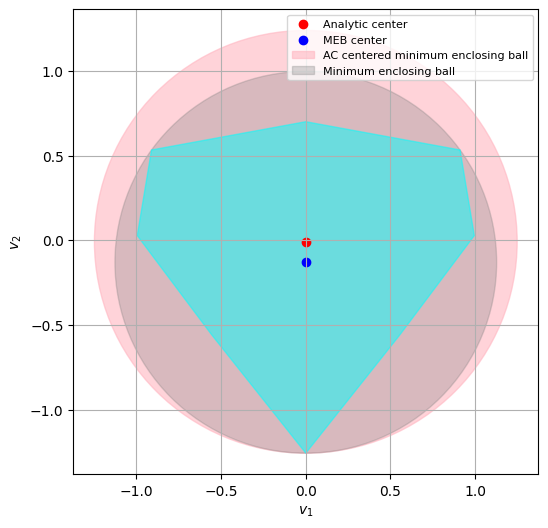}
	 \caption{Octagon}
          \end{subfigure}
          \caption{Analytic centers and MEB centers of some examples of polygons\label{fig:polygons}}
\end{figure}

\begin{table}[htbp]
		\setlength\tabcolsep{6pt}
		\renewcommand{\arraystretch}{1.2}
		\centering
		\caption{Maximum distance between estimators and extreme points} 
			\label{tb:Maxdis} 
		\begin{tabularx}{0.6\textwidth}{p{3cm}<{\centering} |p{3cm}<{\centering}p{3cm}<{\centering} }
			\hline
		Number of edges& AC&MEB center \\
		\hline
		3&0.745 &0.707\\
		4&1.063 &1.008\\
		5&1.619 &1.525\\
		6&1.929 &1.768\\
            7&1.216 &1.137\\
            8&1.248 &1.127\\
			\hline
		\end{tabularx}\label{tab:example1}
	\end{table}

\section{Functional distance}
\label{sec-app-distance}

{\color{black} 
In this paper, we need to quantify not only the size
of polyhedra in the space of partworth vectors and the increment vectors of piecewise-linear utility functions but also
the size of the ambiguity sets corresponding to these polyhedra. To this end, we need to introduce the distance between two utility functions.
Here we recall 
two well-known distances of utility functions in the literature of preference robust optimization.


\begin{definition}[Kantorovich distance $\dd_K$ and scaled Kolmogorov distance $\dd_I$\label{def:kan_kolmo}]
Let 
$$
\begin{aligned}
   &{\cal U}:=\left\{u:[a,b]\rightarrow[0,1] \mid u(a)=0, \ u(b)=1,\
u \text{ is 
nondecreasing
and right-continuous}
\right\},\\
&\mathscr{G}_{L}:=\Big\{\varphi:[a,b]\to\R \ \big|\ \varphi \ \text{is Lipschitz with modulus bounded by}\ 1\Big\},
\quad \text{and}\quad
\mathscr{G}_{I}:=\Big\{\mathbbm{1}_{(a,z]}(\cdot)\ \big|\ z\in[a,b]\Big\}. 
\end{aligned}
$$
For any $u,v\in {\cal U}$, the Kantorovich distance between $u$ and $v$ is defined by 
\bgeqn
\label{eq:uv-Kant-dist}
\dd_K(u,v):= \sup_{\varphi\in \mathscr{G}_{L}} \left[\int_{a}^{b} \varphi(z)du(z)    -\int_{a}^{b} \varphi(z)dv(z)\right],
\edeqn
and the uniform Kolmogorov metric between $u$ and $v$ is defined by
\begin{subequations}
\bgeqn
\dd_{\mathrm{Kol}}(u,v) 
&:=& \sup_{\varphi\in \mathscr{G}_{I}} 
\left|\int_{a}^{b} \varphi(t)d u(t)
-\int_{a}^{b} \varphi(t)d v(t)\right|\\
&=& \sup_{z\in[a,b]}
\left|
\int_{a}^{b} \mathbbm{1}_{(a,z]}(t)d u(t)
-\int_{a}^{b} \mathbbm{1}_{(a,z]}(t)d v(t)
\right| \\
&=& \sup_{z\in[a,b]}
\left|[u(z)-u(a)]-[v(z)-v(a)]\right|\\
&=& \sup_{z\in[a,b]}|u(z)-v(z)|= \|u-v\|_\infty .
\label{eq:uv-Kolm-dist}
\edeqn
\end{subequations}
We further define the scaled Kolmogorov metric by
$
\dd_I(u,v):=(b-a)\dd_{\mathrm{Kol}}(u,v)
=(b-a)\|u-v\|_\infty .
$
\end{definition}
Since $u,v$ are normalized utility functions, then
$
\int_{a}^{b} \varphi(z)d (u - v)(z)=
\int_{a}^{b} (\varphi(z)-\varphi(a))d (u - v)(z).
$
We can replace $\mathscr{G}_L$ with 
$
\mathscr{G}_{L}:=\left\{\varphi: [a,b]\to\mathbb{R} \mid \varphi(a)=0,\ \varphi\text{ is Lipschitz-continuous with modulus bounded by 1}
\right\}.
$
Hence, for any $z\in[a,b]$, $\varphi(z)$ can be represented as 
$\varphi(z)=\int_{a}^z h(t)dt$, where $h(t)$ is a measurable function and $|h(t)|\leqslant1$.
Then 
\begin{eqnarray}
  &\dd_K(u,v)&=\Bigg|\int_{a}^{b} \varphi(z)d (u - v)(z)\Bigg|=\Bigg|\int_{a}^{b} \Big(\int_{a}^z h(t)dt\Big)d (u - v)(z)\Bigg|\nonumber\\
    &&=\Bigg|\int_{a}^{b} h(t)\Big(\int_t^{b} d (u - v)(z)\Big)dt\Bigg|
    =\Bigg|\int_{a}^{b} h(t)\Big[\Big(u(b)-v(b)\Big)-\Big(u(t)-v(t)\Big)\Big]dt\Bigg|\nonumber\\
    &&=\Bigg|\int_{a}^{b} -h(t)\Big(u(t)-v(t)\Big)dt\Bigg|\leqslant \int_{a}^{b} \Big|h(t)\Big|\cdot\Big|u(t)-v(t)\Big|dt\nonumber\\
    &&\leqslant (b-a)\sup_{t\in[a,b]}\Big|u(t)-v(t)\Big|
    =(b-a)\|u - v\|_\infty=\dd_I(u,v).
    \label{eq:d_kant<=d_koml}
\end{eqnarray}

\begin{definition}[Diameter of a set of utility functions and 
the worst-case error of utility estimate]
Let $\mathcal P$ 
be a polyhedron 
in the space of 
increment vectors (see Section~\ref{sec:univariate})
and $\mathcal U_{\PP}$ be the corresponding set of piecewise-linear utility functions (see \eqref{eq:PLA_utility_org}), i.e.,
\begin{equation*}
    \mathcal U_{\PP}:=\{u\mid u(x)=(\RR\vv)^\top\g(x) ,\;\vv\in\mathcal P\}.
\end{equation*}
We call
$
\sup_{u,v\in\mathcal U_{\PP}}\dd_K(u,v)$ and $ \sup_{u,v\in\mathcal U_{\PP}}\dd_I(u,v)
$
the diameter of ${\cal U_{\PP}}$ under the Kantorovich distance and scaled Kolmogorov distance respectively.
Let $u^*\in \mathcal U_{\mathcal P}$ denote the true utility function. We refer to
$
\sup_{u\in\mathcal U_{\mathcal P}}\dd_K(u^*,u)
\quad\text{and}\quad
\sup_{u\in\mathcal U_{\mathcal P}}\dd_I(u^*,u)
$
as the worst-case utility estimation errors induced by $\mathcal U_{\mathcal P}$
under the Kantorovich and scaled Kolmogorov distances, respectively.
\end{definition}

{\color{black}\section{Effect of Concavity Constraints on Convergence}
\label{sec:convergence_concavity}

In Section~\ref{subsec:continuous_example}, we considered the elicitation of a concave utility function by the adaptive-breakpoint CPM. In that experiment, however, the prior information that the true utility function is concave was not explicitly incorporated into the ambiguity set. In this section, we investigate whether imposing concavity constraints improves the convergence behavior of the algorithm. We use the same experimental setting as in Section~\ref{subsec:continuous_example}. The true utility function is
$$
u^*(x)= \frac{1}{1-e^{-24}}(1-e^{-6x}), \quad x\in[0,4],
$$
which is increasing and concave on $[0,4]$. We consider the same two initial breakpoint sets:
$
\mathcal X^0_{(1)}=\{0,0.6,2,4\}$ and $\mathcal X^0_{(2)}=\{0,3,3.5,4\}$.
For both tests, the initial ambiguity set is modified by incorporating the slope-monotonicity constraints induced by concavity:
$$
\mathcal P^{0,0}
=
\left\{
\vv\in\mathbb R^2
\;\middle|\;
0\leqslant v_i\leqslant \min\{1,L(x_{i+1}-x_i)\},\;
\;
\frac{v_{2}}{x^0_{3}-x^0_2}
\leqslant
\frac{v_{1}}{x^0_2-x^0_{1}},
\;\vv^\top \bm e\leqslant 1,
\right\}.
$$
The last two inequalities enforces the nonincreasingness of the slopes of the piecewise-linear utility approximation, and therefore guarantees concavity. After each breakpoint update, the same type of concavity constraints is imposed on the lifted polyhedron $\PP^{k+1,0}$:
$$
\frac{v_{i}}{x^k_{i+1}-x^k_i}
\leqslant
\frac{v_{i-1}}{x^k_i-x^k_{i-1}},
\qquad i=2,\ldots,N_k-1.
$$
Thus, the ambiguity set is restricted throughout the elicitation process to contain piecewise-linear utility functions with the concavity.

Tables~\ref{tab:concavity_example_1} and~\ref{tab:concavity_example_2} report the results for the initial breakpoint sets
$\mathcal X^0_{(1)}=\{0,0.6,2,4\}$ and
$\mathcal X^0_{(2)}=\{0,3,3.5,4\}$, respectively. We compare Tables~\ref{tab:concavity_example_1}--\ref{tab:concavity_example_2} with
Tables~\ref{tab:continuous_example_1}--\ref{tab:continuous_example_2}, where concavity constraints are not imposed. The main observations are as follows.

\begin{table}[!htbp]
	\setlength\tabcolsep{6pt}
	\renewcommand{\arraystretch}{1.2}
	\centering
	\caption{Performance
    of 
    Algorithm~\ref{alg:VCPM_CU}
    in Test 1 with concavity constraints.}
	\label{tab:concavity_example_1}
	\begin{tabularx}{\textwidth}{
        p{1.5cm}<{\centering} 
        p{1cm}<{\centering}
        p{1cm}<{\centering}
        p{1.5cm}<{\centering}
        p{1.5cm}<{\centering}
        p{0.6cm}<{\centering}
        p{0.6cm}<{\centering}
        p{0.8cm}<{\centering}
        p{2.3cm}<{\centering}
        p{2.3cm}<{\centering} }
		\hline
round $k$ &  $\Nk-1$   & $\hk$ & $r(\B^{k,0})$ 
& $r(\B^{k,+})$  &$M_k$ & 
Cuts&N.q &$\max\dd_K(u^*,u)$ &$\max\dd_I(u^*,u)$
 \\
		\hline	      
    0     & 2     & 2     & 0.4950  & 0.2475  & 1     & 2     & 2     & 0.5053  & 2.0511  \\
    1     & 3     & 1.4   & 0.2476  & 0.0619  & 2     & 6     & 8     & 0.2263  & 1.5965  \\
    2     & 4     & 1     & 0.0619  & 0.0310  & 1     & 4     & 12    & 0.1624  & 1.5294  \\
    3     & 5     & 1     & 0.0310  & 0.0077  & 2     & 10    & 22    & 0.1552  & 1.5133  \\
    4     & 6     & 0.7   & 0.0077  & 0.0019  & 2     & 12    & 34    & 0.1516  & 1.5053  \\
    5     & 7     & 0.7   & 0.0097  & 0.0025  & 2     & 14    & 48    & 0.1481  & 1.5053  \\
    6     & 8     & 0.6   & 0.0025  & 0.0006  & 2     & 16    & 64    & 0.1479  & 1.5053  \\
    7     & 9     & 0.5   & 0.3300  & 0.1648  & 1     & 9     & 73    & 0.0777  & 0.9348  \\
    8     & 10    & 0.5   & 0.1698  & 0.0849  & 1     & 10    & 83    & 0.0777  & 0.9348  \\
		\hline
	\end{tabularx}
\end{table}
\begin{table}[!htbp]
	\setlength\tabcolsep{6pt}
	\renewcommand{\arraystretch}{1.2}
	\centering
	\caption{Performance
    of 
    Algorithm~\ref{alg:VCPM_CU}
    in Test 2 with concavity constraints.}
	\label{tab:concavity_example_2}
	\begin{tabularx}{\textwidth}{
        p{1.5cm}<{\centering} 
        p{1cm}<{\centering}
        p{1cm}<{\centering}
        p{1.5cm}<{\centering}
        p{1.5cm}<{\centering}
        p{0.6cm}<{\centering}
        p{0.6cm}<{\centering}
        p{0.8cm}<{\centering}
        p{2.3cm}<{\centering}
        p{2.3cm}<{\centering} }
		\hline
round $k$ &  $\Nk-1$   & $\hk$ & $r(\B^{k,0})$ 
& $r(\B^{k,+})$  &$M_k$ & 
Cuts&N.q &$\max\dd_K(u^*,u)$ &$\max\dd_I(u^*,u)$
 \\
		\hline	      
  0     & 2     & 3     & 0.1398  & 0.0388  & 2     & 4     & 4     & 1.4722  & 3.1806  \\
    1     & 3     & 1.5   & 0.3550  & 0.0895  & 2     & 6     & 10    & 0.7970  & 2.7126  \\
    2     & 4     & 1.5   & 0.3619  & 0.0948  & 2     & 8     & 18    & 0.3577  & 1.9815  \\
    3     & 5     & 0.75  & 0.0948  & 0.0248  & 2     & 10    & 28    & 0.2452  & 1.8265  \\
    4     & 6     & 0.75  & 0.3542  & 0.1769  & 1     & 6     & 34    & 0.1443  & 1.2353  \\
    5     & 7     & 0.75  & 0.1769  & 0.0724  & 1     & 7     & 41    & 0.0823  & 0.9889  \\
    6     & 8     & 0.75  & 0.0724  & 0.0184  & 2     & 16    & 57    & 0.0675  & 0.9455  \\
    7     & 9     & 0.5   & 0.0184  & 0.0049  & 2     & 18    & 75    & 0.0663  & 0.9455  \\
		\hline
	\end{tabularx}
\end{table}

\begin{itemize}
    \item For the initial breakpoint set
    $\mathcal X^0_{(1)}=\{0,0.6,2,4\}$, imposing concavity constraints does not lead to a uniformly faster reduction of the worst-case Kantorovich distance
    $
    \max_{u\in\mathcal U^k}\dd_K(u^*,u).
    $
    This is consistent with Figure~\ref{fig:Test1}(a), which shows that, after the first round, the worst-case piecewise-linear approximations induced by the polyhedra are already concave. Hence, in this test, concavity is effectively recovered by the algorithm itself during the elicitation process. Nevertheless, the concavity constraints reduce the number of queries required to reach a comparable level of worst-case Kantorovich error. This suggests that, although concavity does not significantly change the asymptotic trend in this case, it improves the efficiency of the elicitation process.

    \item For the initial breakpoint set
    $\mathcal X^0_{(2)}=\{0,3,3.5,4\}$, the effect of concavity constraints is more pronounced in the early rounds. In particular, the values of
    $\max_{u\in\mathcal U^k}\dd_K(u^*,u)$
    in Table~\ref{tab:concavity_example_2} are smaller than those in Table~\ref{tab:continuous_example_2} during the first 5 rounds. This improvement occurs because, under this initial breakpoint set, the worst-case piecewise-linear approximation generated without concavity constraints is not concave even after the first round; see Figure~\ref{fig:Test2}(a). As the elicitation process proceeds, however, the worst-case approximations generated without explicit concavity constraints also become concave. Consequently, the marginal benefit of imposing concavity constraints gradually diminishes. However, similar to Test 1, the total number of queries needed to reach a comparable worst-case Kantorovich distance is reduced when concavity constraints are imposed.
\end{itemize}

Overall, the numerical results indicate that prior concavity information may improve the query efficiency of the adaptive-breakpoint CPM. However, as more queries are collected and the breakpoint set is refined, the adaptive-breakpoint CPM is able to recover the concavity structure naturally. Therefore, the concavity constraints do not alter the final convergence behavior of the algorithm.
}

\section{Tractable Formulations of the Robust Pricing Problem\label{app:robust-pricing}}

In this section, we derive a tractable linear-programming reformulation of the robust pricing problem~\eqref{eq:robust_pricing}.

The robust pricing problem is solved for fixed DM $n$ and round index $k$, to simplify the notation, we rewrite the problem as following:
\begin{subequations}
\label{eq:sim_robust_pricing}
\begin{align}
\max_{p\in[0,1]} \quad
    & p
\label{eq:sim_robust_pricing_obj}
\\
\mathrm{s.t.}\quad
    &(\RR{\vv})^\top
      \bigl(p,(\x^{\text{new}})^\top\bigr)^\top
      \geqslant
      (\RR{\vv})^\top
      \bigl(p^{\text{base}},(\x^{\text{base}})^\top\bigr)^\top,
      \qquad \forall \vv\in \mathcal{P}.
\label{eq:sim_robust_pricing_con}
\end{align}
\end{subequations}
where
 $\PP:=
    \left\{
        \vv\in\mathbb R^3
        \,\middle|\,
        \bm{A}\vv\leqslant \bm{b},\;\bm{A}\in\mathbb R^{m\times 3},\; \bm{b}\in\mathbb R^{m}
    \right\}$.

Define the non-price attribute difference
$
    \Delta \x:=\x^{new}-\x^{base}
    =
    (\Delta x_1,\Delta x_2,\Delta x_3)^\top .
$
Then the robust constraint \eqref{eq:sim_robust_pricing_con} can be written as
$$
    (\RR\vv)^\top
    \bigl(p-p^{base},\Delta \x^\top\bigr)^\top
    \geqslant 0,
    \qquad \forall \vv\in \PP.
$$
Since $\RR\vv=(v_1,v_2,v_3,1-\bm e^\top \vv)^\top$, we obtain
$$
\begin{aligned}
    (\RR\vv)^\top
    \bigl(p-p^{base},\Delta \x^\top\bigr)^\top
    &=
    v_1(p-p^{base})
    +v_2\Delta x_1
    +v_3\Delta x_2
    +(1-\e^\top \vv)\Delta x_3                                      \\
    &=
    \Delta x_3
    +
    \vv^\top \bm z(p),
\end{aligned}
$$
where
$
    \bm z(p)
    :=
    \left(
        p-p^{base}-\Delta x_3,\,
        \Delta x_1-\Delta x_3,\,
        \Delta x_2-\Delta x_3
    \right)^\top .
$
Therefore, \eqref{eq:sim_robust_pricing_con}  is equivalent to the worst-case inequality
\begin{equation}\label{eq:app-worst-case}
    \min_{\vv\in \PP}
    \left\{
        \bm z(p)^\top \vv+\Delta x_3
    \right\}
    \geqslant 0. 
\end{equation}

We now dualize the inner minimization problem in \eqref{eq:app-worst-case}. Since
$\PP=\{\vv\in\mathbb R^3\mid \bm A\vv\leqslant \bm b\}$, \eqref{eq:app-worst-case} is a linear program
\begin{equation}
\begin{aligned}
    \min_{\vv\in\mathbb R^3}\quad
        & \bm z(p)^\top \vv+\Delta x_3, \quad
    \mathrm{s.t.}\quad
         \bm A \vv\leqslant \bm b.
\end{aligned}
\label{eq:app-inner-primal}   
\end{equation}
Its Lagrangian dual is
\begin{equation}
    \begin{aligned}
    \max_{\lambda\in\mathbb R_+^{m}}\quad
        & \Delta x_3-\bm b^\top \bm\lambda       \quad
    \mathrm{s.t.}\quad
         \bm A^\top \bm\lambda=-\bm z(p).
\end{aligned}
\label{eq:app-inner-dual}
\end{equation}
Because $\PP$ is nonempty and bounded,
the primal problem \eqref{eq:app-inner-primal} is feasible and has a finite optimal value. Hence strong
duality holds. It follows that \eqref{eq:app-worst-case} holds if and only if there exists
$\bm\lambda\in\mathbb R_+^{m}\) such that
$$
    \bm A^\top \bm\lambda=-\bm z(p),
    \quad
    \bm b^\top \lambda\leqslant \Delta x_3 .
$$

Consequently, the robust pricing problem \eqref{eq:sim_robust_pricing} admits the following equivalent
linear-programming reformulation:
\bgeqn\label{eq:app-robust-pricing-lp}
    &\max_{p,\bm \lambda}\quad
        & p                             \\
    &\mathrm{s.t.}\quad
        & \bm A^\top \bm\lambda=-\bm z(p),\;                           \bm b^\top \lambda\leqslant \Delta x_3,
        \;p\in[0,1],\quad\bm \lambda\in\R_+^m.\nonumber
\edeqn
Since $\bm z(p)$ is affine in $p$, problem \eqref{eq:app-robust-pricing-lp} is a linear program
with decision variables $p$ and $\bm\lambda$.

\section{The volume of projected ellipsoid}

\begin{remark}[Finite-dimensional volume may be misleading]
The volume of the breakpoint-value polyhedron (minimal-volume outer ellipsoid is the same) is not a reliable measure of the remaining ambiguity when the grid is refined. 
For the concept of volume, we refer to the proportion of the measure (e.g. area in two dimensions or volume in three dimensions) of the polyhedron over the measure of the unit box: $[0,1]\times\cdots\times[0,1]$.

We illustrate this point with a simple example. Consider the class of normalized nondecreasing $2$-Lipschitz utility functions on $[0,1]$:
\begin{equation}
\mathcal U_2
:=
\left\{
u:[0,1]\to[0,1]\mid
u(0)=0,\ u(1)=1,\ u \ {\rm is\ nondecreasing\ and}\ 2{\rm -Lipschitz}
\right\}.
\end{equation}
Let $X_n={0,1/n,\ldots,1}$ be a uniform grid, where $n$ is even, and denote
\begin{equation}
y_i=u(i/n),\quad i=0,\ldots,n .
\end{equation}
The breakpoint values satisfy $y_0=0$, $y_n=1$, and
\begin{equation}
0\leq y_i-y_{i-1}\leq \frac{2}{n},\quad i=1,\ldots,n .
\end{equation}
Thus the corresponding breakpoint-value polyhedron is
\begin{equation}
P_n
=
\left\{
(y_1,\ldots,y_{n-1})\in\mathbb R^{n-1}
\mid
y_0=0,\ y_n=1,\
0\leq y_i-y_{i-1}\leq \frac{2}{n},\ i=1,\ldots,n
\right\}.
\end{equation}
Let $v_i=y_i-y_{i-1}$. In the reduced coordinates $(v_1,\ldots,v_{n-1})$, the set $P_n$ is contained in the box $[0,2/n]^{n-1}$. Hence
\begin{equation}
{\rm Vol}_{n-1}(P_n)
\leq
\left(\frac{2}{n}\right)^{n-1}
\to 0
\quad {\rm as}\quad n\to\infty .
\end{equation}
However, this vanishing volume does not imply that the induced utility ambiguity becomes small in a functional metric. Indeed, define two breakpoint-value vectors by
\begin{equation}
y_i^L=\min\left\{\frac{2i}{n},1\right\},
\qquad
y_i^R=\max\left\{\frac{2i}{n}-1,0\right\},
\quad i=0,\ldots,n .
\end{equation}
Both vectors belong to $P_n$. Their piecewise-linear interpolations are
\begin{equation}
u_n^L(x)=\min\{2x,1\},
\qquad
u_n^R(x)=\max\{2x-1,0\}.
\end{equation}
At $x=1/2$, we have $u_n^L(1/2)=1$ and $u_n^R(1/2)=0$. Therefore,
\begin{equation}
||u_n^L-u_n^R||_\infty=1,
\qquad \forall n .
\end{equation}
Consequently, although the finite-dimensional volume ${\rm Vol}_{n-1}(P_n)$ converges to zero, the Kolmogorov diameter of the induced utility class does not shrink. This example shows that volumes of breakpoint-value polyhedra are not directly comparable across different grids and may vanish purely because of increasing dimension and changing coordinate scale. For continuous utility functions, it is therefore more appropriate to measure ambiguity by functional distances, such as the Kolmogorov or Kantorovich distance, rather than by the Euclidean volume of the finite-dimensional breakpoint polyhedron.
\end{remark}

\end{document}